\documentclass{amsart}

\usepackage{amssymb,amscd,amsxtra,soul
}

\usepackage{hyperref}

\usepackage{esint}

\usepackage[shortlabels]{enumitem}
\usepackage[capitalise]{cleveref}
\crefrangeformat{enumi}{items #3#1#4--#5#2#6}

\usepackage[all]{xy}

\usepackage[toc,page]{appendix}

\usepackage{imakeidx}
\makeindex[columns=3, title=Index of notation]

\hypersetup{colorlinks=false}

\usepackage{orcidlink}

\usepackage{color}
\definecolor{darkgreen}{cmyk}{1,0,1,.2}
\definecolor{m}{rgb}{1,0.1,1}

\numberwithin{equation}{section}

\newdimen\theight
\def\TeXref#1{%
             \leavevmode\vadjust{\setbox0=\hbox{{\tt
                     \quad\quad  {\small \textrm #1}}}%
             \theight=\ht0
             \advance\theight by \lineskip
             \kern -\theight \vbox to
             \theight{\rightline{\rlap{\box0}}%
             \vss}%
             }}%

\theoremstyle{plain}

\newtheorem{thm}{Theorem}[section]
\newtheorem{lem}[thm]{Lemma}
\newtheorem{cor}[thm]{Corollary}
\newtheorem{prop}[thm]{Proposition}

\theoremstyle{definition}

\newtheorem{ex}[thm]{Example}

\theoremstyle{remark}

\newtheorem{rem}[thm]{Remark}
\newtheorem{claim}[thm]{Claim}

\crefname{thm}{Theorem}{Theorems}
\crefname{lem}{Lemma}{Lemmas}
\crefname{cor}{Corollary}{Corollaries}
\crefname{prop}{Proposition}{Propositions}

\crefname{defn}{Definition}{Definitions}
\crefname{conj}{Conjecture}{Conjectures}
\crefname{ex}{Example}{Examples}
\crefname{exs}{Examples}{Examples}
\crefname{prob}{Problem}{Problems}
\crefname{quest}{Question}{Questions}

\crefname{rem}{Remark}{Remarks}
\crefname{claim}{Claim}{Claims}
\crefname{case}{Case}{Cases}
\crefname{hyp}{Hypothesis}{Hypotheses}
\crefname{notation}{Notation}{Notations}

\newcommand{\C}{\mathbb{C}}
\newcommand{\K}{\mathbb{K}}
\newcommand{\N}{\mathbb{N}}

\newcommand{\R}{\mathbb{R}}
\renewcommand{\S}{\mathbb{S}}
\newcommand{\Z}{\mathbb{Z}}

\renewcommand{\AA}{\mathcal{A}}

\newcommand{\FF}{\mathcal{F}}

\newcommand{\KK}{\mathcal{K}}

\renewcommand{\SS}{\mathcal{S}}

\newcommand{\VV}{\mathcal{V}}

\newcommand{\fX}{\mathfrak{X}}

\newcommand{\bfg}{\boldsymbol{g}}

\newcommand{\bfx}{\boldsymbol{x}}

\newcommand{\bfM}{\boldsymbol{M}}

\newcommand{\bfT}{\boldsymbol{T}}

\newcommand{\bfpi}{\boldsymbol{\pi}}

\newcommand{\supp}{\operatorname{supp}}
\newcommand{\esssup}{\operatorname{ess\,sup}}

\newcommand{\id}{\operatorname{id}}

\newcommand{\Diff}{\operatorname{Diff}}
\newcommand{\GL}{\operatorname{GL}}

\newcommand{\inj}{\operatorname{inj}}

\newcommand{\sing}{\operatorname{sing}}

\newcommand{\bOmega}{{}^{\text{\rm b}}\Omega}
\newcommand{\bT}{{}^{\text{\rm b}}T}

\newcommand{\Diffb}{\Diff_{\text{\rm b}}}

\newcommand{\Hb}{H_{\text{\rm b}}}

\newcommand{\Diffub}{\Diff_{\text{\rm ub}}}
\newcommand{\Cinftyub}{C^\infty_{\text{\rm ub}}}
\newcommand{\Cinftyc}{C^\infty_{\text{\rm c}}}
\newcommand{\Cinftycv}{C^\infty_{\text{\rm cv}}}

\newcommand{\fXub}{\fX_{\text{\rm ub}}}

\newcommand{\fXb}{\fX_{\text{\rm b}}}

\title{Topology of the space of conormal distributions}

\author[J.A. \'Alvarez L\'opez]{Jes\'us A. \'Alvarez L\'opez\,\orcidlink{0000-0001-6056-2847}}
\address{Department of Mathematics and CITMAga\\
         University of Santiago de Compostela\\
         15782 Santiago de Compostela\\ Spain}
\email{jesus.alvarez@usc.es}

\thanks{The authors are partially supported by the grants MTM2017-89686-P and PID2020-114474GB-I00 (AEI/FEDER, UE) and ED431C 2019/10 (Xunta de Galicia, FEDER)}

\author[Y.A. Kordyukov]{Yuri A. Kordyukov\,\orcidlink{0000-0003-2957-2873}}
\address{Institute of Mathematics\\ Ufa Federal Research Center\\ Russian Academy of Sciences\\
112 Chernyshevsky street\\ 450008 Ufa\\ Russia}
\email{yurikor@matem.anrb.ru}

\author[E. Leichtnam]{Eric Leichtnam\,\orcidlink{0000-0002-5058-5508}}
\address{Institut de Math\'ematiques de Jussieu-PRG\\ CNRS\\ Batiment Sophie Germain (bureau 740)\\ Case~7012\\ 75205 Paris Cedex 13, France}
\email{eric.leichtnam@imj-prg.fr}

\date{\today}



\begin{document}

\begin{abstract}
Given a closed manifold $M$ and a closed regular submanifold $L$, consider the corresponding locally convex space $I=I(M,L)$ of conormal distributions, with its natural topology, and the strong dual $I'=I'(M,L)=I(M,L;\Omega)'$ of the space of conormal densities. It is shown that $I$ is a barreled, ultrabornological, webbed, Montel, acyclic LF-space, and $I'$ is a complete Montel space, which is a projective limit of bornological barreled spaces. In the case of codimension one, similar properties and additional descriptions are proved for the subspace $K\subset I$ of conormal distributions supported in $L$ and for its strong dual $K'$. We construct a locally convex Hausdoff space $J$ and a continuous linear map $I\to J$ such that the sequence $0\to K\to I\to J\to 0$ as well as the transpose sequence $0\to J'\to I'\to K'\to 0$ are short exact sequences in the category of continuous linear maps between locally convex spaces. Finally, it is shown that $I\cap I'=C^\infty(M)$ in the space of distributions. In another publication, these results are applied to prove a Lefschetz trace formula for a simple foliated flow $\phi=\{\phi^t\}$ on a compact foliated manifold $(M,\mathcal F)$. It describes a Lefschetz distribution $L_{\text{\rm dis}}(\phi)$ defined by the induced action $\phi^*=\{\phi^{t\,*}\}$ on the reduced cohomologies $\bar H^\bullet I(\mathcal F)$ and $\bar H^\bullet I'(\mathcal F)$ of the complexes of leafwise currents that are conormal and dual-conormal at the leaves preserved by $\phi$.
\end{abstract}

\maketitle

\tableofcontents

\section{Introduction}\label{s: intro}

Given a smooth manifold $M$ and a closed regular submanifold $L\subset M$, the corresponding space $I=I(M,L)$ of conormal distributions was considered in \cite{KohnNirenberg1965,Hormander1971}, \cite[Section~18.2]{Hormander1985-III}, \cite[Chapters~3--5]{Simanca1990}, \cite[Chapters~4 and~6]{Melrose1996}, \cite[Chapters~3 and~9]{MelroseUhlmann2008}. But its study was mainly oriented to the important role played in the analysis of pseudodifferential operators. This space was also used as a tool to get appropriate generalizations of those operators to manifolds with boundary or corners, stratified spaces, etc. For instance, the existence of asymptotic expansions of their symbols was well analyzed. However we are not aware of any publication with a deep study of its natural topology; actually, this project was begun in \cite[Chapters~4 and~6]{Melrose1996}, but that publication remains incomplete. The main objective of this paper is to fill in that gap of study.

Assume $M$ is compact for the sake of simplicity. Let $\Diff(M,L)\subset\Diff(M)$ be the subalgebra of differential operators generated by the vector fields on $M$ tangent to $L$, and let $H^s(M)$ be the Sobolev space of order $s\in\R$. The space $I^{(s)}=I^{(s)}(M,L)$ of conormal distributions of Sobolev order $s$ consists of the distributions $u\in C^{-\infty}(M)$ satisfying $\Diff(M,L)u\subset H^s(M)$, endowed with the projective topology given by the maps $P : I^{(s)}\to H^s(M)$ ($P\in\Diff(M,L)$). 

By definition, $I=\bigcup_sI^{(s)}$ with the corresponding locally convex inductive topology, which is continuously contained in $C^{-\infty}(M)$. There is another expression $I=\bigcup_mI^m$, using spaces $I^m=I^m(M,L)$ of conormal distributions with symbol order $m$, which can be locally described in terms of symbol spaces with a partial Fourier transform.

We show that every $I^{(s)}$ is a totally reflexive Fr\'echet space (\Cref {p: I^(s)(M L) is a totally reflexive Frechet sp}), and the LF-space $I$ is a barreled, ultrabornological, webbed, acyclic Montel space, and therefore complete, boundedly retractive and reflexive (\Cref{c: I(M L) is barreled,c: I(M L) is acyclic and Montel}). As a first step, these properties are established for symbol spaces.

All notions and properties considered here have straightforward extensions for distributional sections of vector bundles. In particular, for the density bundle $\Omega=\Omega M$, the strong dual $I'(M,L)=I(M,L;\Omega)'$, simply denoted by $I'$, is also continuously contained in $C^{-\infty}(M)$. We prove that $I'$ is a complete Montel space and $I'=\varprojlim I^{\prime\,(s)}$, where $I^{\prime\,(s)}=I^{\prime\,(s)}(M,L)=I^{(-s)}(M,L;\Omega)'$ is bornological and barreled (\Cref{c: I'^(s)(M L) is bornological,c: I'(M L) is complete and Montel,c: I'(M L) equiv varprojlim I'^(s)(M L)}).

Now assume $L$ is of codimension one. For simplicity reasons, consider the case where $L$ is transversely oriented. Then cut $M$ along $L$ to obtain a compact manifold with boundary $\bfM$ and a projection $\bfpi:\bfM\to M$. In this way, we can take advantage of the machinery developed in \cite{Hormander1985-III,Melrose1996} to study conormal distributions at the boundary; in particular, some notions from small b-calculus are used. For instance, with the terminology and notation of \cite{Hormander1985-III,Melrose1996}, let $\AA(\bfM)$ (respectively, $\dot\AA(\bfM)$) be the locally convex space of extendable (respectively, supported) conormal distributions at the boundary. Then, via the push-down map $\bfpi_*$, the image of $\dot\AA(\bfM)$ is $I$, and $\AA(\bfM)$ becomes isomorphic to another locally convex space $J=J(M,L)$. Let $K=K(M,L)\subset I$ be the subspace of conormal distributions supported in $L$. Like in the definition of $I'$, consider also the strong dual spaces $K'(M,L)=K(M,L;\Omega)'$ and $J'(M,L)=J(M,L;\Omega)'$, simply denoted by $K'$ and $J'$. It is proved that $K$ is a limit subspace (\Cref{c: K(M L) = bigcup_s K^{(s)}(M L)}), the spaces $K$, $J$, $K'$ and $J'$ satisfy the properties stated for $I$ and $I'$ (\Cref{c: K(M L) is barreled,c: K(M L) is acyclic and Montel,c: K'(M L) and J'(M L) are complete and Montel,c: K^prime [s](M L) J^prime [s](M L) are bornological,c: J'(M L) equiv varprojlim J^prime [s](M L)}), and there are short exact sequences, $0\to K\to I\to J\to0$ and its transpose $0\leftarrow K'\leftarrow I'\leftarrow J'\leftarrow0$, in the category of continuous linear maps between locally convex spaces (\Cref{p: R: I(M L) -> J(M L) is a surj top hom,p: dual-conormal seq of (M L) is exact}). These sequences are relevant because $J$, $K$, $J'$ and $K'$ have better descriptions than $I$ and $I'$ (\Cref{c: AA(M) equiv bigcup_m x^m Hb^infty(M),c: dot AA'(M) equiv bigcap_m x^m Hb^-infty(M),p: bigoplus_m C^1_m -> K(M L)}). Finally, it is shown that $I\cap I'=C^\infty(M)$ (\Cref{t: bfpi^*(I(M L) cap I'(M L)) subset C^infty(bfM)}), extending a result of \cite{Hormander1985-III,Melrose1996} for the boundary case.  Most of these properties are first established in the boundary case (\Cref{s: conormal distribs at the boundary}).

Besides the extensions for distributional sections of vector bundles, some results are extended to non-compact manifolds. We also analyze the action of differential operators on these spaces, as well as the pull-back and push-forward homomorphisms induced by maps on these spaces (\Cref{s: conormal distribs,s: dual-conormal distribs,s: conormal seq,s: dual-conormal seq}).   

Via the Schwartz kernel theorem, the spaces of pseudodifferential and differential operators can be described as $\Psi(M)\equiv I(M^2,\Delta)$ and $\Diff(M)\equiv K(M^2,\Delta)$, where $\Delta$ is the diagonal of $M^2$. Thus $\Psi(M)$ and $\Diff(M)$ become examples of locally convex spaces satisfying the above properties.

The wave front set of any $u\in I(M,L)$ satisfies $\operatorname{WF}(u)\subset N^*L\setminus0_L$ (considering $N^*L\subset T^*M$) \cite[Chapter~VIII]{Hormander1983-I}, \cite[Chapter~XVIII]{Hormander1985-III}; this is the reason of the term ``conormal distribution.'' The larger space of all distributions whose wave front set is contained in any prescribed closed cone of $T^*M\setminus0_M$, like $N^*L\setminus0_L$, also has a natural topology which was studied in \cite{DabrowskiBrouder2014}.

Our results for codimension one can be clearly extended to arbitrary codimension. We only consider codimension one for simplicity reasons. It is also clear that there are further extensions to manifolds with corners, stratified spaces, etc.

The case of codimension one is also enough for our application in a trace formula for simple foliated flows \cite{AlvKordyLeichtnam-atffff}. These are simple flows $\phi=\{\phi^t\}$ that preserve the leaves of a foliation $\FF$ on  $M$. C.~Deninger conjectured the existence of a ``Lefschetz distribution''  $L_{\text{\rm dis}}(\phi)$ on $\R$ for the induced pull-back action $\phi^*=\{\phi^{t\,*}\}$ on the leafwise reduced cohomology $\bar H^\bullet(\FF)$, and predicted a formula for $L_{\text{\rm dis}}(\phi)$ involving data from the closed orbits and fixed points \cite{Deninger2008}. Here, $\bar H^\bullet(\FF)$ is the maximal Hausdorff quotient of the leafwise cohomology $H^\bullet(\FF)$, defined by the de~Rham derivative of the leaves acting on leafwise differential forms smooth on $M$, equipped with the $C^\infty$ topology. But we can not use leafwise forms smooth on $M$ if there are leaves preserved by $\phi$; they do not work well. Instead, we consider the spaces $I(\FF)$ and $I'(\FF)$ of distributional leafwise currents that are conormal and dual-conormal at the preserved leaves, giving rise to reduced cohomologies, $\bar H^\bullet I(\FF)$ and $\bar H^\bullet I'(\FF)$, with actions $\phi^*$. The spaces $K(\FF)$, $J(\FF)$, $K'(\FF)$ and $J'(\FF)$ are similarly defined, obtaining short exact sequences, $0\to\bar H^\bullet K(\FF)\to\bar H^\bullet I(\FF)\to\bar H^\bullet J(\FF)\to0$ and $0\leftarrow\bar H^\bullet K'(\FF)\leftarrow\bar H^\bullet I'(\FF)\leftarrow\bar H^\bullet J'(\FF)\leftarrow0$. In this way, the definition of $L_{\text{\rm dis}}(\phi)$ for both $\bar H^\bullet I(\FF)$ and $\bar H^\bullet I'(\FF)$ together can be reduced to the cases of $\bar H^\bullet K(\FF)$, $\bar H^\bullet J(\FF)$, $\bar H^\bullet K'(\FF)$ and $\bar H^\bullet J'(\FF)$. This can be done by using the descriptions of $K(\FF)$, $J(\FF)$, $K'(\FF)$ and $J'(\FF)$, and some additional ingredients. Using these ideas, we define in \cite{AlvKordyLeichtnam-atffff} the Lefschetz distribution $L_{\text{\rm dis}}(\phi)$, which has the desired expression plus a zeta invariant produced by the use of the b-trace of R.~Melrose \cite{Melrose1993}. However the ingredients can be chosen so that the zeta invariant vanishes \cite{AlvKordyLeichtnam-ziomf}, and the predicted formula becomes correct. We hope that this method will provide useful tools in future developments of Deninger's project.

\section{Preliminaries}\label{s: prelim}

\subsection{Topological vector spaces}\label{ss: TVS}

The field of coefficients is $\K=\R,\C$. For the general theory of topological vector spaces (TVSs), we follow the references \cite{Edwards1965,Horvath1966-I,Kothe1969-I,Schaefer1971,NariciBeckenstein2011}, assuming the following conventions. We always consider locally convex spaces (LCSs), which are not assumed to be Hausdorff (contrary to the definition of \cite{Schaefer1971}); the abbreviation LCHS is used in the Hausdorff case. Local convexity is preserved by all operations we use. For any inductive/projective system (or spectrum) of continuous linear maps between LCSs, we have its (locally convex) inductive/projective limit; in particular, when the inductive/projective spectrum consists of a sequence of continuous inclusions, their union/intersection is endowed with the inductive/projective limit topology. This applies to the locally convex direct sum and the topological product of LCSs. LF-spaces are not assumed to be strict. For any LCS $X$, its (continuous) dual $X'$ is always endowed with the strong topology; i.e., we write $X'=X'_\beta$ with the usual notation.

Some homological theory of LCSs will be used (see \cite{Wengenroth2003} and references therein) For instance, for an inductive spectrum of LCSs of the form $(X_k)=(X_0\subset X_1\subset\cdots)$, the condition of being \emph{acyclic} can be described as follows \cite[Theorem~6.1]{Wengenroth2003}: for all $k$, there is some $k'\ge k$ such that, for all $k''\ge k'$, the topologies of $X_{k'}$ and $X_{k''}$ coincide on some $0$-neighborhood of $X_k$. In this case, $X:=\bigcup_kX_k$ is Hausdorff if and only if all $X_k$ are Hausdorff \cite[Proposition~6.3]{Wengenroth2003}. It is said that $(X_k)$ is \emph{regular} if any bounded $B\subset X$ is contained and bounded in some step $X_k$. If moreover the topologies of $X$ and $X_k$ coincide on $B$, then $(X_k)$ is said to be \emph{boundedly retractive}. The conditions of being \emph{compactly retractive} or \emph{sequentially retractive} are similarly defined, using compact sets or convergent sequences. 

If the steps $X_k$ are Fr\'echet spaces, the above properties of $(X_k)$ depend only on the LF-space $X$ \cite[Chapter~6, p.~111]{Wengenroth2003}; thus it may be said that they are properties of $X$. In this case, $X$ is acyclic if and only if it is boundedly/compactly/sequentially retractive \cite[Proposition~6.4]{Wengenroth2003}. As a consequence, acyclic LF-spaces are complete and regular \cite[Corollary~6.5]{Wengenroth2003}. 

A topological vector subspace $Y\subset X$ is called a \emph{limit subspace} if $Y\equiv\bigcup_kY_k$, where $Y_k=X\cap Y_k$. This condition is satisfied if and only if the spectrum consisting of the spaces $X_k/Y_k$ is acyclic \cite[Chapter~6, p.~110]{Wengenroth2003}.

Assume the steps $X_k$ are LCHSs. It is said that $(X_k)$ is \emph{compact} if the inclusion maps are compact operators. In this case, $(X_k)$ is clearly acyclic, and so $X$ is Hausdorff. Moreover $X$ is a complete bornological DF Montel space \cite[Theorem~6']{Komatsu1967}.

The above concepts and properties also apply to an inductive/projective spectrum consisting of continuous inclusions $X_r\subset X_{r'}$ for $r<r'$ in $\R$ because $\bigcap_rX_r=\bigcap_kX_{r_k}$ and $\bigcup_rX_r=\bigcup_kX_{s_k}$ for sequences $r_k\downarrow-\infty$ and $s_k\uparrow\infty$.

In the category of continuous linear maps between LCSs, the exactness of a sequence $0\to X \to Y \to Z\to0$ means that it is exact as a sequence of linear maps and consists of topological homomorphisms \cite[Sections~2.1 and~2.2]{Wengenroth2003}.

\subsection{Smooth functions on open subsets of $\R^n$}\label{ss: smooth functions on U}

For an open $U\subset\R^n$ ($n\in\N_0=\N\cup\{0\}$), we will use the Fr\'echet space $C^\infty(U)$ of smooth ($\K$-valued) functions on $U$ with the topology of uniform approximation of all partial derivatives on compact subsets, which is described by the semi-norms
\begin{equation}\label{| u |_K C^k}
\|u\|_{K,C^k}=\sup_{x\in K,\ |\alpha|\le k}|\partial^\alpha u(x)|\;,
\end{equation}
for any compact $K\subset U$, $k\in\N_0$ and $\alpha\in\N_0^n$, using standard multi-index notation. (Sometimes the notation $C^\infty_{\text{\rm loc}}(U)$ is used for this space, and $C^\infty(U)$ is used for the uniform space denoted by $\Cinftyub(U)$ in this paper.) For any $S\subset U$, the notation $C^\infty_S(U)$ is used for the subspace of smooth functions supported in $S$ (with the subspace topology). (The common notation $C^\infty(S)=C^\infty_S(U)$ would be confusing when extended to other function spaces.) Recall also the strict LF-space of compactly supported functions,
\begin{equation}\label{Cinftyc(U)}
\Cinftyc(U)=\bigcup_KC^\infty_K(U)\;,
\end{equation}
for compact subsets $K\subset U$ (an exhausting increasing sequence of compact subsets is enough).

The above definitions have straightforward generalizations to the case of functions with values in $\K^l$ ($l\in\N$), obtaining 
\begin{equation}\label{C^infty_./c(U K^l) equiv C^infty_./c(U) otimes K^l}
C^\infty_{{\cdot}/\text{\rm c}}(U,\K^l)\equiv C^\infty_{{\cdot}/\text{\rm c}}(U)\otimes\K^l\;.
\end{equation}
(The notation $C^\infty_{{\cdot}/\text{\rm c}}$ or $C^\infty_{\text{\rm c}/{\cdot}}$ refers to both $C^\infty$ and $C^\infty_{\text{\rm c}}$.)

\subsection{Vector bundles}\label{ss: vector bundless}

The notation $M$ will be used for a smooth manifold of dimension $n$, and $E$ for a ($\K$-) vector bundle over $M$. The fibers of $E$ are denoted by $E_x$ ($x\in M$), the zero in every $E_x$ by $0_x$, and the image of the zero section by $0_M$. Let $\Omega^aE$ \index{$\Omega^aE$} ($a\in\R$) denote the line bundle of $a$-densities of $E$, let $\Omega E=\Omega^1E$, \index{$\Omega E$} and let $o(E)$ be the flat line bundle of orientations of $E$. We may use the notation $E_L=E|_L$ for the restriction of $E$ to a submanifold $L\subset M$. As particular cases, we have the tangent and cotangent $\R$-vector bundles, $TM$ and $T^*M$, and the associated $\K$-vector bundles $o(M)=o(TM)$, $\Omega^aM=\Omega^aTM$ and $\Omega M=\Omega TM$. \index{$\Omega^aM$} \index{$\Omega M$}

\subsection{Smooth and distributional sections}\label{ss: smooth/distributional sections}

Our notation for spaces of distributional sections mainly follows \cite{Melrose1996}, with some minor changes to fit our application in \cite{AlvKordyLeichtnam-atffff}. Some notation from \cite{Hormander1983-I,Hormander1985-III} is also used.

Generalizing $C^\infty(U,\K^l)$, we have the Fr\'echet space $C^\infty(M;E)$ of smooth sections of $E$, whose topology is described by semi-norms $\|{\cdot}\|_{K,C^k}$ defined as in~\eqref{| u |_K C^k} via charts $(U,x)$ of $M$ and diffeomorphisms of triviality $E_U\equiv U\times\K^l$, with $K\subset U$. This procedure is standard and will be used again with other section spaces.

Redundant notation will be removed as usual. For instance, we write $C^\infty(M)$ (respectively, $C^\infty(M,\K^l)$) in the case of the trivial vector bundle of rank $1$ (respectively, $l$). We also write $C^\infty(L,E)=C^\infty(L,E_L)$ and $C^\infty(M;\Omega^a)=C^\infty(M;\Omega^aM)$. We may write $C^\infty(E)=C^\infty(M;E)$ if $M$ is fixed, but this may also mean the space of smooth functions on $E$. In particular, $\fX(M)=C^\infty(M;TM)$ is the Lie algebra of vector fields. The subspace $C^\infty_S(M;E)$ \index{$C^{\pm\infty}_S(M;E)$} is defined like in \Cref{ss: smooth functions on U}. Similar notation will be used with any LCHS and $C^\infty(M)$-module continuously included in $C^\infty(M;E)$, or in the space $C^{-\infty}(M;E)$ defined below.

The notation $C^\infty(M;E)$, or $C^\infty(E)$, is also used with any smooth fiber bundle $E$, obtaining a completely metrizable topological space with the weak $C^\infty$ topology.

The strict LF-space $\Cinftyc(M;E)$ \index{$C^{\pm\infty}_{{\cdot}/\text{\rm c}}(M;E)$} of compactly supported smooth sections is defined like in~\eqref{Cinftyc(U)}, using compact subsets $K\subset M$. There is a continuous inclusion $\Cinftyc(M;E)\subset C^\infty(M;E)$. If $M$ is a fiber bundle, the LCHS $\Cinftycv(M;E)$ \index{$\Cinftycv(M;E)$} of smooth sections with compact support in the vertical direction is similarly defined using~\eqref{| u |_K C^k} and~\eqref{Cinftyc(U)} with closed subsets $K\subset M$ whose intersection with the fibers is compact (now an exhaustive increasing sequence of such subsets $K$ is not enough).

The space of distributional sections with arbitrary/compact support is
\begin{equation}\label{C^-infty_cdot/c(M;E)}
C^{-\infty}_{{\cdot}/\text{\rm c}}(M;E)=C^\infty_{\text{\rm c}/{\cdot}}(M;E^*\otimes\Omega)'\;.
\end{equation}
(In \cite{Hormander1983-I}, these dual spaces are endowed with the weak topology, contrary to our convention.) Integration of smooth densities on $M$ and the canonical pairing of $E$ and $E^*$ define a continuous dense inclusion $C^\infty_{{\cdot}/\text{\rm c}}(M;E)\subset C^{-\infty}_{{\cdot}/\text{\rm c}}(M;E)$. If $U\subset M$ is open, the extension by zero defines a TVS-embedding $C^{\pm\infty}_{\text{\rm c}}(U;E)\subset C^{\pm\infty}_{\text{\rm c}}(M;E)$. 

The above spaces of distributional sections can be also described in terms of the corresponding spaces of distributions as the algebraic tensor product as $C^\infty(M)$-modules
\begin{equation}\label{C^infty(M)-tensor product description of C^pm infty_./c(M E)}
C^{-\infty}_{{\cdot}/\text{\rm c}}(M;E)\equiv C^{-\infty}_{{\cdot}/\text{\rm c}}(M)\otimes_{C^\infty(M)}C^\infty(M;E)\;.
\end{equation}
To show this identity, $E$ can be realized as a vector subbundle of a trivial vector bundle $F=M\times\K^{l'}$ \cite[Theorem~4.3.1]{Hirsch1976}. Then, like in~\eqref{C^infty_./c(U K^l) equiv C^infty_./c(U) otimes K^l},
\begin{align*}
C^{-\infty}_{{\cdot}/\text{\rm c}}(M;F)
&\equiv C^{-\infty}_{{\cdot}/\text{\rm c}}(M)\otimes\K^{l'}
\equiv C^{-\infty}_{{\cdot}/\text{\rm c}}(M)\otimes_{C^\infty(M)}C^\infty(M)\otimes\K^{l'}\\
&\equiv C^{-\infty}_{{\cdot}/\text{\rm c}}(M)\otimes_{C^\infty(M)}C^\infty(M;F)\;,
\end{align*}
and the spaces of~\eqref{C^infty(M)-tensor product description of C^pm infty_./c(M E)} clearly correspond by these identities. Expressions like~\eqref{C^infty(M)-tensor product description of C^pm infty_./c(M E)} hold for most of the LCSs of distributional sections we will consider, which are also $C^\infty(M)$-modules. Thus, from now on, we will mostly define and study  those spaces for the trivial line bundle or density bundles, and then the notation for arbitrary vector bundles will be used without further comment, and the properties have straightforward extensions.

Consider also the Fr\'echet space $C^k(M)$ ($k\in\N_0$) of $C^k$ functions, with the semi-norms $\|{\cdot}\|_{K,C^k}$ given like in~\eqref{| u |_K C^k}, the LF-space $C^k_{\text{\rm c}}(M)$ of $C^k$ functions with compact support, defined like in~\eqref{Cinftyc(U)}, and the space $C^{\prime\,-k}_{{\cdot}/\text{\rm c}}(M)$ \index{$C^{\prime\,-k}_{{\cdot}/\text{\rm c}}(M)$} of distributions of order $k$ with arbitrary/compact support, defined like in~\eqref{C^-infty_cdot/c(M;E)}. (A prime is added to this notation to distinguish $C^{\prime\,0}_{{\cdot}/\text{\rm c}}(M)$ from $C^0_{{\cdot}/\text{\rm c}}(M)$.) There are continuous dense inclusions
\begin{equation}\label{C^prime -k'(M E) supset C^prime -k(M E)}
C^{k'}_{{\cdot}/\text{\rm c}}(M)\subset C^k_{{\cdot}/\text{\rm c}}(M)\;,\quad 
C^{\prime\,-k'}_{\text{\rm c}/{\cdot}}(M)\supset C^{\prime\,-k}_{\text{\rm c}/{\cdot}}(M)\quad(k<k')\;,
\end{equation}
with \cite[Exercise~12.108]{NariciBeckenstein2011}
\begin{equation}\label{bigcap_k C^k_./c(M) = C^infty_./c(M)}
\bigcap_kC^k_{{\cdot}/\text{\rm c}}(M)=C^{\infty}_{{\cdot}/\text{\rm c}}(M)\;,\quad
\bigcup_kC^{\prime\,-k}_{\text{\rm c}}(M)=C^{-\infty}_{\text{\rm c}}(M)\;.
\end{equation}
The space $\bigcup_kC^{\prime\,-k}(M)$ consists of the distributions with some order; it is $C^{-\infty}(M)$ just when $M$ is compact.

Let us recall some properties of the spaces we have seen. In addition of the fact that $C^{\infty}(M)$ and $C^{k}(M)$ are Fr\'echet spaces \cite[Example~2.9.3]{Horvath1966-I}, $C^{\infty}_{\text{\rm c}}(M)$ and $C^{k}_{\text{\rm c}}(M)$ are complete and Hausdorff \cite[Examples~2.12.6 and~2.12.8]{Horvath1966-I}. $C^{\infty}_{{\cdot}/\text{\rm c}}(M)$ and $C^{k}_{{\cdot}/\text{\rm c}}(M)$ are ultrabornological because this property is satisfied by Fr\'echet spaces and preserved by inductive limits \cite[Example~13.2.8~(d) and Theorem~13.2.11]{NariciBeckenstein2011}, and therefore they are barreled \cite[Observation~6.1.2~(b)]{PerezCarrerasBonet1987}. $C^{\pm\infty}_{{\cdot}/\text{\rm c}}(M)$ is a Montel space (in particular, barreled) \cite[Examples~3.9.3,~3.9.4 and~3.9.6 and Proposition~3.9.9]{Horvath1966-I}, \cite[Section~8.4.7, Theorem~8.4.11 and Application~8.4.12]{Edwards1965}, \cite[the paragraph before~IV.5.9]{Schaefer1971}, and therefore reflexive \cite[Section~8.4.7]{Edwards1965}, \cite[6.27.2~(1)]{Kothe1969-I}, \cite[IV.5.8]{Schaefer1971}. $C^{\infty}_{{\cdot}/\text{\rm c}}(M)$ is a Schwartz space \cite[Examples~3.15.2 and~3.15.3]{Horvath1966-I}, and therefore $C^{-\infty}_{{\cdot}/\text{\rm c}}(M)$ is ultrabornological \cite[Exercise~3.15.9~(c)]{Horvath1966-I}. $C^\infty(M)$ is distinguished \cite[Examples~3.16.1]{Horvath1966-I}. $C^{\pm\infty}_{{\cdot}/\text{\rm c}}(M)$ is webbed because this property is satisfied by LF-spaces and strong duals of strict inductive limits of sequences of metrizable LCSs \cite[Proposition~IV.4.6]{DeWilde1978}, \cite[7.35.1~(4) and~7.35.4~(8)]{Kothe1979-II}, \cite[Theorem~14.6.5]{NariciBeckenstein2011}.

\subsection{Linear operators on section spaces}\label{ss: ops}

Let $E$ and $F$ be vector bundles over $M$, and let $A:\Cinftyc(M;E)\to C^\infty(M;F)$ be a continuous linear map. Recall that the \emph{transpose} of $A$ is the continuous linear map \index{$A^t$}
\begin{gather*}
A^t:C^{-\infty}_{\text{\rm c}}(M;F^*\otimes\Omega)\to C^{-\infty}(M;E^*\otimes\Omega)\;,\\
\langle A^tv,u\rangle=\langle v,Au\rangle\;,\quad u\in\Cinftyc(M;E)\;,\quad v\in C^{-\infty}_{\text{\rm c}}(M;F^*\otimes\Omega)\;.
\end{gather*}
For instance, the transpose of $\Cinftyc(M;E^*\otimes\Omega)\subset C^\infty(M;E^*\otimes\Omega)$ is a continuous dense injection $C^{-\infty}_{\text{\rm c}}(M;E)\subset C^{-\infty}(M;E)$. If $A^t$ restricts to a continuous linear map $\Cinftyc(M;F^*\otimes\Omega)\to C^\infty(M;E^*\otimes\Omega)$, then $A^{tt}:C^{-\infty}_{\text{\rm c}}(M;E)\to C^{-\infty}(M;F)$ is a continuous extension of $A$, also denoted by $A$.

There are versions of the construction of $A^t$ and $A^{tt}$ when both the domain and codomain of $A$ have compact support, or no support restriction. For example, for any open $U\subset M$, the transpose of the extension by zero $\Cinftyc(U;E^*\otimes\Omega)\subset\Cinftyc(M;E^*\otimes\Omega)$ is the restriction map $C^{-\infty}(M;E)\to C^{-\infty}(U,E)$, $u\mapsto u|_U$, and the transpose of the restriction map $C^\infty(M;E^*\otimes\Omega)\to C^\infty(U,E^*\otimes\Omega)$ is the extension by zero $C^{-\infty}_{\text{\rm c}}(U;E)\subset C^{-\infty}_{\text{\rm c}}(M;E)$. In the whole paper, inclusion maps may be denoted by $\iota$ \index{$\iota$} and restriction maps by $R$, \index{$R$} without further comment.

Other related concepts and results, like singular support, Schwartz kernel and the Schwartz kernel theorem, can be seen e.g.\ in \cite{Melrose1993}.

\subsection{Pull-back and push-forward of distributional sections}\label{ss: pull-back and push-forward of distrib sections}

Recall that any smooth map $\phi:M'\to M$ induces the continuous linear pull-back map
\begin{equation}\label{phi^*: C^infty(M E) -> C^infty(M' phi^*E)}
\phi^*:C^\infty(M;E)\to C^\infty(M';\phi^*E)\;.
\end{equation}
Suppose that moreover $\phi$ is a submersion. Then it also induces the continuous linear push-forward map
\begin{equation}\label{phi_*: C^infty_c(M' phi^*E otimes Omega_fiber) -> C^infty_c(M E)}
\phi_*:C^\infty_{\text{\rm c}}(M';\phi^*E\otimes\Omega_{\text{\rm fiber}})\to C^\infty_{\text{\rm c}}(M;E)\;,
\end{equation}
where $\Omega_{\text{\rm fiber}}=\Omega_{\text{\rm fiber}}M'=\Omega\VV$ for the vertical subbundle $\VV=\ker\phi_*\subset TM'$. Since $\phi^*\Omega M\equiv\Omega(TM/\VV)\equiv\Omega^{-1}_{\text{\rm fiber}}\otimes\Omega M'$, the transposes of the versions of~\eqref{phi^*: C^infty(M E) -> C^infty(M' phi^*E)} and~\eqref{phi_*: C^infty_c(M' phi^*E otimes Omega_fiber) -> C^infty_c(M E)} with $E^*\otimes\Omega M$ are continuous extensions of~\eqref{phi_*: C^infty_c(M' phi^*E otimes Omega_fiber) -> C^infty_c(M E)} and~\eqref{phi^*: C^infty(M E) -> C^infty(M' phi^*E)} \cite[Theorem~6.1.2]{Hormander1983-I},
\begin{gather}
\phi_*:C^{-\infty}_{\text{\rm c}}(M';\phi^*E\otimes\Omega_{\text{\rm fiber}})\to C^{-\infty}_{\text{\rm c}}(M;E)\;,
\label{phi_*: C^-infty_c(M' phi^*E otimes Omega_fiber) -> C^-infty_c(M E)}\\
\phi^*:C^{-\infty}(M;E)\to C^{-\infty}(M';\phi^*E)\;,
\label{phi_*: C^-infty(M E) -> C^-infty(M' phi^*E)}
\end{gather}
also called push-forward and pull-back maps. The term integration along the fibers is also used for $\phi_*$.

If $\phi:M'\to M$ is a proper local diffeomorphism, then we can omit $\Omega_{\text{\rm fiber}}$ and the compact support condition in~\eqref{phi_*: C^infty_c(M' phi^*E otimes Omega_fiber) -> C^infty_c(M E)} and~\eqref{phi_*: C^-infty_c(M' phi^*E otimes Omega_fiber) -> C^-infty_c(M E)}, and therefore the compositions $\phi_*\phi^*$ and $\phi^*\phi_*$ are defined on smooth/distributional sections.

The space $C^\infty(M';\phi^*E)$ becomes a $C^\infty(M)$-module via the algebra homomorphism $\phi^*:C^\infty(M)\to C^\infty(M')$, and we have
\begin{equation}\label{C^infty(M)-tensor product description of C^pm infty_./c(M' phi^*E)}
C^{\pm\infty}_{{\cdot}/\text{\rm c}}(M';\phi^*E)=C^{\pm\infty}_{{\cdot}/\text{\rm c}}(M')\otimes_{C^\infty(M)}C^\infty(M;E)\;.
\end{equation}
Using~\eqref{C^infty(M)-tensor product description of C^pm infty_./c(M E)} and~\eqref{C^infty(M)-tensor product description of C^pm infty_./c(M' phi^*E)}, we can describe~\eqref{phi^*: C^infty(M E) -> C^infty(M' phi^*E)}--\eqref{phi_*: C^-infty(M E) -> C^-infty(M' phi^*E)} as the $C^\infty(M)$-tensor products of their trivial-line-bundle versions with the identity map on $C^\infty(M;E)$. Thus, from now on, only pull-back and push-forward of distributions will be considered.

\subsection{Differential operators}\label{ss: diff ops}

Let $\Diff(M)$ be the filtered algebra and $C^\infty(M)$-module of differential operators, filtered by the order. Every $\Diff^m(M)$ ($m\in\N_0$) is spanned as $C^\infty(M)$-module by all compositions of up to $m$ elements of $\fX(M)$, considered as the Lie algebra of derivations of $C^\infty_{{\cdot}/\text{\rm c}}(M)$. In particular, $\Diff^0(M)\equiv C^\infty(M)$.

For vector bundles $E$ and $F$ over $M$, the above concepts can be extended by taking the $C^\infty(M)$-tensor product with $C^\infty(M;F\otimes E^*)$, obtaining $\Diff^m(M;E,F)$ ($\Diff^m(M;E)$ \index{$\Diff^m(M;E)$} being obtained if $E=F$); here, redundant notation is simplified like in the case of $C^{\pm\infty}(M;E)$ (\Cref{ss: smooth/distributional sections}). If $E$ is a line bundle, then
\begin{align}
\Diff^m(M;E)
&\equiv\Diff^m(M)\otimes_{C^\infty(M)}C^\infty(M;E\otimes E^*)\notag\\
&\equiv\Diff^m(M)\otimes_{C^\infty(M)}C^\infty(M)\equiv\Diff^m(M)\;.
\label{Diff^m(M E) equiv Diff^m(M)}
\end{align}

Any $A\in\Diff^m(M;E)$ defines a continuous linear endomorphism $A$ of $C^\infty_{{\cdot}/\text{\rm c}}(M;E)$. We get $A^t\in\Diff^m(M;E^*\otimes\Omega)$ using integration by parts. So $A$ has continuous extensions to a continuous endomorphism $A$ of $C^{-\infty}_{{\cdot}/\text{\rm c}}(M;E)$ (\Cref{ss: ops}). A similar map is defined when $A\in\Diff^m(M;E,F)$. 

Other related concepts like symbols and ellipticity can be seen e.g.\ in \cite{Melrose1993}.

\subsection{$L^2$ sections}\label{ss: L^2}

Recall that the Hilbert space $L^2(M;\Omega^{1/2})$ of square-integrable half-densities is the completion of $\Cinftyc(M;\Omega^{1/2})$ with the scalar product $\langle u,v\rangle=\int_Mu\bar v$. The induced norm is denoted by $\|{\cdot}\|$.

If $M$ is compact, the space $L^2(M;E)$ \index{$L^2(M;E)$} of square-integrable sections of $E$ can be described as the $C^\infty(M)$-tensor product of $L^2(M;\Omega^{1/2})$ and $C^\infty(M;\Omega^{-1/2}\otimes E)$. It becomes a Hilbert space with the scalar product $\langle u,v\rangle=\int_M(u,v)\,\omega$ determined by the choice of a Euclidean/Hermitian structure $({\cdot},{\cdot})$ on $E$ and a non-vanishing $\omega\in C^\infty(M;\Omega)$. The equivalence class of its norm $\|{\cdot}\|$ is independent of those choices; in this sense, $L^2(M;E)$ is called a \emph{Hilbertian space} if no norm is distinguished. 

When $M$ is not assumed to be compact, any choice of $({\cdot},{\cdot})$ and $\omega$ can be used to define $L^2(M;E)$ and $\langle{\cdot},{\cdot}\rangle$. Now $L^2(M;E)$ and the equivalence class of $\|{\cdot}\|$ depends on the choices involved. The independence still holds for sections supported in any compact $K\subset M$, obtaining the Hilbertian space $L^2_K(M;E)$. Then the strict LF-space $L^2_{\text{\rm c}}(M;E)$ is defined like in~\eqref{Cinftyc(U)}. On the other hand, let
\begin{equation}\label{L^2_loc(M E)}
L^2_{\text{\rm loc}}(M;E)=\{\,u\in C^{-\infty}(M;E)\mid\Cinftyc(M)\,u\subset L^2_{\text{\rm c}}(M;E)\,\}\;,
\end{equation}
which is a Fr\'echet space with the semi-norms $u\mapsto\|f_ku\|$, for a countable partition of unity $\{f_k\}\subset\Cinftyc(M)$. If $M$ is compact, then $L^2_{\text{\rm loc/c}}(M;E)\equiv L^2(M;E)$ \index{$L^2_{\text{\rm loc/c}}(M;E)$} as TVSs. The spaces $L^2_{\text{\rm loc/c}}(M;E)$ satisfy the obvious version of~\eqref{C^-infty_cdot/c(M;E)}.

Any $A\in\Diff^m(M;E)$ can be considered as a densely defined operator in $L^2(M;E)$. Integration by parts shows that the adjoint $A^*$ is defined by an element $A^*\in\Diff^m(M;E)$ (the \emph{formal adjoint} of $A$).

\subsection{$L^\infty$ sections}\label{ss: L^infty}

A Euclidean/Hermitian structure can be also used to define the Banach space $L^\infty(M;E)$ of its essentially bounded sections, with the norm $\|u\|_{L^\infty}=\esssup_{x\in M}|u(x)|$. There is a continuous injection $L^\infty(M;E)\subset L^2_{\text{\rm loc}}(M;E)$. \index{$L^\infty(M;E)$} If $M$ is compact, then the equivalence class of $\|{\cdot}\|_{L^\infty}$ is independent of $({\cdot},{\cdot})$.

\subsection{Sobolev spaces}\label{ss: Sobolev sps}

\subsubsection{Local and compactly supported versions}\label{sss: Sobolev sps - loc/c}

Recall that the Fourier transform, $f\mapsto\hat f$, defines a TVS-automorphism of the Schwartz space $\SS(\R^n)$, which extends to a TVS-automorphism of the space $\SS(\R^n)'$ of tempered distributions \cite[Section~7.1]{Hormander1983-I}. In turn, for every $s\in\R$, this automorphism of $\SS(\R^n)'$ restricts to unitary isomorphism
\begin{equation}\label{H^s(R^n) cong L^2(R^n (1+|xi|^2)^s d xi)}
H^s(\R^n)\xrightarrow{\cong} L^2(\R^n,(1+|\xi|^2)^s\,d\xi)\;,\quad f\mapsto\hat f\;,
\end{equation}
for some Hilbert space $H^s(\R^n)$, called the Sobolev space of order $s$ of $\R^n$. There is a canonical continuous inclusion $H^s(\R^n)\subset C^{-\infty}(\R^n)$.

For any compact $K\subset \R^n$, we have the Hilbert subspace $H^s_K(\R^n)\subset H^s(\R^n)$ of elements supported in $K$. Then the LCHSs $H^s_{\text{\rm c/loc}}(U)$ are defined like in~\eqref{C^-infty_cdot/c(M;E)} and~\eqref{L^2_loc(M E)}, using the spaces $H^s_K(\R^n)$ for compact subsets $K\subset U$. They are continuously included in $C^{-\infty}_{\text{\rm c}/{\cdot}}(U)$.

For a manifold $M$, the definition of the LCHSs $H^s_{\text{\rm c/loc}}(M)$ \index{$H^s_{\text{\rm c/loc}}(M)$} can be extended in a standard way, by using a locally finite atlas and a partition of unity consisting of compactly supported smooth functions. These are the compactly supported and local versions of the Sobolev space of order $s$ of $M$. They are continuously included in $C^{-\infty}_{\text{\rm c}/{\cdot}}(M)$.

\subsubsection{Case of compact manifolds}\label{sss: Sobolev sps - compact}

Suppose for a while that $M$ is compact. Then $H^s(M):=H_{\text{\rm loc}}^s(M)=H^s_{\text{\rm c}}(M)$ \index{$H^s(M)$} is a Hilbertian space called the \emph{Sobolev space} of order $s$ of $M$. We have
\begin{equation}\label{H^-s(M) = H^s(M  Omega)'}
H^{-s}(M)=H^s(M;\Omega)'\;,
\end{equation}
given by~\eqref{C^-infty_cdot/c(M;E)}. Moreover there are continuous dense inclusions,
\begin{gather}
H^s(M)\subset H^{s'}(M)\;,
\label{H^s(M) subset H^s'(M)}\\
\intertext{for $s'<s$, and}
H^s(M)\subset C^k(M)\subset H^k(M)\;,
\label{H^s(M) subset C^k(M) subset H^k(M)}\\
H^{-s}(M)\supset C^{\prime\,-k}(M)\supset H^{-k}(M)\;,
\label{H^-s(M) supset C^prime -k(M) supset H^-k(M)}
\end{gather}
for $s>k+n/2$. The first inclusion of~\eqref{H^s(M) subset C^k(M) subset H^k(M)} is the Sobolev embedding theorem, and~\eqref{H^-s(M) supset C^prime -k(M) supset H^-k(M)} is the transpose of the version of~\eqref{H^s(M) subset C^k(M) subset H^k(M)} with $\Omega M$. Moreover the inclusions~\eqref{H^s(M) subset H^s'(M)} are compact (Rellich theorem). So the spaces $H^s(M)$ form a compact spectrum with
\begin{equation}\label{C^infty(M) = bigcap_s H^s(M)}
C^\infty(M)=\bigcap_sH^s(M)\;\quad 
C^{-\infty}(M)=\bigcup_sH^s(M)\;.
\end{equation}

Any $A\in\Diff^m(M;E)$ defines a bounded operator $A:H^{s+m}(M;E)\to H^s(M;E)$. It can be considered as a densely defined operator in $H^s(M;E)$, which is closable because, after fixing a scalar product in $H^s(M;E)$, the adjoint of $A$ in $H^s(M;E)$ is densely defined since it is induced by $\bar A^t\in\Diff^m(M;\bar E^*\otimes\Omega)$ via the identity $H^s(M;E)\equiv H^s(M;\bar E)'=H^{-s}(M;\bar E^*\otimes\Omega)$, where the bar stands for the complex conjugate. In the case $s=0$, the adjoint of $A$ is induced by the formal adjoint $A^*\in\Diff^m(M;E)$.

By the elliptic estimate, a scalar product on $H^s(M)$ can be defined by $\langle u,v\rangle_s=\langle(1+P)^su,v\rangle$, for any choice of a nonnegative symmetric elliptic $P\in\Diff^2(M)$,  where $\langle{\cdot},{\cdot}\rangle$ is defined like in \Cref{ss: L^2} and $(1+P)^s$ is given by the spectral theorem for all $s\in\R$. The corresponding norm $\|{\cdot}\|_s$ is independent of the choice of $P$. For a vector bundle $E$, a precise scalar product on $H^s(M;E)$ can be defined as above, using any choice of a Euclidean/Hermitian structure $({\cdot},{\cdot})$ on $E$ and a non-vanishing $\omega\in C^\infty(M;\Omega)$ (\Cref{ss: L^2}), besides a nonnegative symmetric elliptic $P\in\Diff^2(M;E)$. If $E=\Omega^{1/2}M$, then $\langle{\cdot},{\cdot}\rangle_s$ can be defined independently of $({\cdot},{\cdot})$ and $\omega$ (\Cref{ss: L^2}).

If $s\in\N_0$, we can also describe
\begin{align}
H^s(M)&=\{\,u\in C^{-\infty}(M)\mid\Diff^s(M)\,u\subset L^2(M)\,\}\;,\label{H^s(M) = ...}\\
H^{-s}(M)&=\Diff^s(M)\,L^2(M)\;,\label{H^-s(M) = ...}
\end{align}
with the respective projective and injective topologies given by the maps $A:H^s(M)\to L^2(M)$ and $A:L^2(M)\to H^{-s}(M)$ ($A\in\Diff^s(M)$).

\subsubsection{Extension to non-compact manifolds}\label{sss: Sobolev sps - non-compact}

If $M$ is not assumed to be compact, then $H^s(M;E)$ can be defined as the completion of $\Cinftyc(M;E)$ with respect to the scalar product $\langle{\cdot},{\cdot}\rangle_s$ defined by the above choices of $({\cdot},{\cdot})$, $\omega$ and $P$; in this case, $H^s(M;E)$ and the equivalence class of $\|{\cdot}\|_s$ depends on the choices involved. For instance, in~\eqref{H^s(R^n) cong L^2(R^n (1+|xi|^2)^s d xi)}, $H^s(\R^n)$ can be also described with the Laplacian of $\R^n$ and the standard density and Euclidean/Hermitian structure. The version of~\eqref{H^-s(M) = H^s(M  Omega)'} with $E$ can be used to define $H^{-s}(M;E)$. With this generality, the versions of~\eqref{H^s(M) = ...},~\eqref{H^-s(M) = ...}, the right-hand side inclusions of~\eqref{H^s(M) subset C^k(M) subset H^k(M)} and~\eqref{H^-s(M) supset C^prime -k(M) supset H^-k(M)}, and the inclusions ``$\subset$'' of~\eqref{C^infty(M) = bigcap_s H^s(M)} are wrong, but the versions of~\eqref{H^s(M) subset H^s'(M)}, the left-hand side continuous inclusions of~\eqref{H^s(M) subset C^k(M) subset H^k(M)} and~\eqref{H^-s(M) supset C^prime -k(M) supset H^-k(M)}, and the continuous inclusion ``$\supset$'' of~\eqref{C^infty(M) = bigcap_s H^s(M)} are true. Thus the intersection and union of~\eqref{C^infty(M) = bigcap_s H^s(M)} define new LCHSs $H^{\pm\infty}(M)$, which are continuously included in $C^{\pm\infty}(M)$. Any $A\in\Diff^m(M;E)$ defines continuous linear maps $A:H^s_{\text{\rm c/loc}}(M;E)\to H^{s-m}_{\text{\rm c/loc}}(M;F)$.

\subsection{Weighted spaces}\label{ss: weighted sps}

\subsubsection{Case of compact manifolds}\label{sss: weighted sps - compact}

Assume first that $M$ is compact. Take any $h\in C^\infty(M)$ which is positive almost everywhere; for instance, $\{h=0\}$ could be any countable union of submanifolds of positive codimension. Then the \emph{weighted Sobolev space} $hH^s(M;E)$ is a Hilbertian space; a scalar product $\langle{\cdot},{\cdot}\rangle_{hH^s}$ is given by $\langle u,v\rangle_{hH^s}=\langle h^{-1}u,h^{-1}v\rangle_s$, depending on the choice of a scalar product $\langle{\cdot},{\cdot}\rangle_s$ on $H^s(M;E)$ like in \Cref{ss: Sobolev sps}. The corresponding norm is denoted by $\|{\cdot}\|_{hH^s}$. In particular, we get the \emph{weighted $L^2$ space} $hL^2(M;E)$. We have $h>0$ just when $hH^m(M;E)=H^m(M;E)$; in this case, $\langle{\cdot},{\cdot}\rangle_{hH^s}$ can be described like $\langle{\cdot},{\cdot}\rangle_s$ using $h^{-2}\omega$ instead of $\omega$. Thus the notation $hH^m(M;E)$ for $h>0$ is used when changing the density; e.g.,  if it is different from a distinguished choice, say a Riemanian volume.

\subsubsection{Extension to non-compact manifolds}\label{sss: weighted sps - non-compact}

If $M$ is not compact, $hH^s(M;E)$ and $\langle u,v\rangle_{hH^s}$ depend on $h$ and the chosen definitions of $H^s(M;E)$ and $\langle u,v\rangle_s$ (\Cref{ss: Sobolev sps}). We also get the weighted spaces $hH^s_{\text{\rm c/loc}}(M;E)$, \index{$hH^s_{\text{\rm c/loc}}(M;E)$} and the weighted Banach space $hL^\infty(M;E)$ \index{$hL^\infty(M;E)$} with the norm $\|u\|_{hL^\infty}=\|h^{-1}u\|_{L^\infty}$. There is a continuous injection $hL^\infty(M;E)\subset hL^2_{\text{\rm loc}}(M;E)$.

\subsection{Bounded geometry}\label{s: bd geom}

Concerning this topic, we follow \cite{Eichhorn1991,Roe1988I,Shubin1992,Schick1996,Schick2001}; see also \cite{AlvKordyLeichtnam2014} for the way we present it and examples. 

\subsubsection{Manifolds and vector bundles of bounded geometry}\label{sss: mfds and vector bdls of bd geom}

The concepts recalled here become relevant when $M$ is not compact. Equip $M$ with a Riemannian metric $g$, and let $\nabla$ denote its Levi-Civita connection, $R$ its curvature and $\inj_M\ge0$ its injectivity radius (the infimum of the injectivity radius at all points). If $M$ is connected, we have an induced distance function $d$. If $M$ is not connected, we can also define $d$ taking $d(p,q)=\infty$ if $p$ and $q$ belong to different connected components. Observe that $M$ is complete if $\inj_M>0$. For $r>0$ and $p\in M$, let $B(p,r)$ and $\overline B(p,r)$ denote the open and closed $r$-balls centered at $p$.

Recall that $M$ is said to be of \emph{bounded geometry} if $\inj_M>0$ and $\sup|\nabla^mR|<\infty$ for every $m\in\N_0$. This concept has the following chart description. 

\begin{thm}[Eichhorn \cite{Eichhorn1991}; see also \cite{Roe1988I,Schick1996,Schick2001}]\label{t: mfd of bd geom}
$M$ is of bounded geometry if and only if, for some open ball $B\subset\R^n$ centered at $0$, there are normal coordinates at every $p\in M$ defining a diffeomorphism $y_p:V_p\to B$ such that the corresponding Christoffel symbols $\Gamma^i_{jk}$, as a family of functions on $B$ parametrized by $i$, $j$, $k$ and $p$, lie in a bounded set of the Fr\'echet space $C^\infty(B)$. This equivalence holds as well replacing the Cristoffel symbols with the metric coefficients $g_{ij}$.
\end{thm}

From now on in this subsection, assume $M$ is of bounded geometry and consider the charts $y_p:V_p\to B$ given by \Cref{t: mfd of bd geom}. The radius of $B$ is denoted by $r_0$. 

\begin{prop}[{Schick \cite[Theorem~A.22]{Schick1996}, \cite[Proposition~3.3]{Schick2001}}]\label{p: |partial_I(y_q y_p^-1)|}
For every $\alpha\in\N_0^n$, the function $|\partial^\alpha(y_qy_p^{-1})|$ is bounded on $y_p(V_p\cap V_q)$, uniformly on $p,q\in M$.
\end{prop}

\begin{prop}[{Shubin \cite[Appendix~A1.1, Lemma~1.2]{Shubin1992}}]\label{p: p_k}
For any $0<2r\le r_0$, there is a subset $\{p_k\}\subset M$ and some $N\in\N$ such that the balls $B(p_k,r)$ cover $M$, and every intersection of $N+1$ sets $B(p_k,2r)$ is empty.
\end{prop}

A vector bundle $E$ of rank $l$ over $M$ is said to be of \emph{bounded geometry} when it is equipped with a family of local trivializations over the charts $(V_p,y_p)$, for small enough $r_0$, with corresponding defining cocycle $a_{pq}:V_p\cap V_q\to\GL(l,\K)\subset\K^{l\times l}$, such that, for all $\alpha\in\N_0^n$, the function $|\partial^\alpha(a_{pq}y_p^{-1})|$ is bounded on $y_p(V_p\cap V_q)$, uniformly on $p,q\in M$. When referring to local trivializations of a vector bundle of bounded geometry, we always mean that they satisfy this condition. If the corresponding defining cocycle is valued in the orthogonal/unitary group, then $E$ is said to be of \emph{bounded geometry} as a Euclidean/Hermitian vector bundle.

\subsubsection{Uniform spaces}\label{sss: uniform sps}

For every $m\in\N_0$, a function $u\in C^m(M)$ is said to be \emph{$C^m$-uniformy bounded} if there is some $C_m\ge0$ with $|\nabla^{m'}u|\le C_m$ on $M$ for all $m'\le m$. These functions form the \emph{uniform $C^m$ space} $C_{\text{\rm ub}}^m(M)$, \index{$C_{\text{\rm ub}}^m(M)$} which is a Banach space with the norm $\|{\cdot}\|_{C^m_{\text{\rm ub}}}$ defined by the best constant $C_m$. Equivalently, we may take the norm $\|{\cdot}\|'_{C^m_{\text{\rm ub}}}$ defined by the best constant $C'_m\ge0$ such that $|\partial^\alpha(uy_p^{-1})|\le C'_m$ on $B$ for all $p\in M$ and $|\alpha|\le m$; in fact, it is enough to consider any subset of points $p$ so that $\{V_p\}$ covers $M$ \cite[Theorem~A.22]{Schick1996}, \cite[Proposition~3.3]{Schick2001}. The \emph{uniform $C^\infty$ space} is $\Cinftyub(M)=\bigcap_mC_{\text{\rm ub}}^m(M)$. \index{$\Cinftyub(M)$} This is a Fr\'echet space with the semi-norms $\|{\cdot}\|_{C^m_{\text{\rm ub}}}$ or $\|{\cdot}\|'_{C^m_{\text{\rm ub}}}$. It consists of the functions $u\in C^\infty(M)$ such that all functions $uy_p^{-1}$ lie in a bounded set of $C^\infty(B)$, which are said to be \emph{$C^\infty$-uniformy bounded}. 

The same definitions apply to functions with values in $\C^l$. Moreover the definition of uniform spaces with covariant derivative can be also considered for non-complete Riemannian manifolds.

\begin{prop}[{Shubin \cite[Appendix~A1.1, Lemma~1.3]{Shubin1992}; see also \cite[Proposition~3.2]{Schick2001}}]\label{p: f_k}
Given $r$, $\{p_k\}$ and $N$ like in \Cref{p: p_k}, there is a partition of unity $\{f_k\}$ subordinated to the open covering $\{B(p_k,r)\}$, which is bounded in the Fr\'echet space $\Cinftyub(M)$.
\end{prop}

For a Euclidean/Hermitian vector bundle $E$ of bounded geometry over $M$, the \emph{uniform $C^m$ space} $C_{\text{\rm ub}}^m(M;E)$, of \emph{$C^m$-uniformly bounded} sections, can be defined by introducing $\|{\cdot}\|'_{C^m_{\text{\rm ub}}}$ like the case of functions, using local trivializations of $E$ to consider every $uy_p^{-1}$ in $C^m(B,\C^l)$ for all $u\in C^m(M;E)$. Then, as above, we get the \emph{uniform $C^\infty$ space} $\Cinftyub(M;E)$ of \emph{$C^\infty$-uniformly bounded} sections, which are the sections $u\in C^\infty(M;E)$ such that all functions $uy_p^{-1}$ define a bounded set of $\Cinftyub(B;\C^l)$. In particular, $\fXub(M):=\Cinftyub(M;TM)$ \index{$\fXub(M)$} is a $\Cinftyub(M)$-submodule and Lie subalgebra of $\fX(M)$.

\subsubsection{Differential operators of bounded geometry}\label{sss: diff ops of bd geom}

Like in \Cref{ss: diff ops}, by using $\fXub(M)$ and $\Cinftyub(M)$ instead of $\fX(M)$ and $C^\infty(M)$, we get the filtered subalgebra and $\Cinftyub(M)$-submodule $\Diffub(M)\subset\Diff(M)$ of differential operators of \emph{bounded geometry}. Observe that 
\begin{equation}\label{C^m_ub(M)}
C^m_{\text{\rm ub}}(M)=\{\,u\in C^m(M)\mid\Diffub(M)\,u\subset L^\infty(M)\ \forall m'\le m\,\}\;.
\end{equation}
For vector bundles of bounded geometry $E$ and $F$ over $M$, by taking the $\Cinftyub(M)$-tensor product of $\Diffub(M)$ and $\Cinftyub(M;F\otimes E^*)$, we obtain the filtered $\Cinftyub(M)$-submodule $\Diffub(M;E,F)\subset\Diff(M;E,F)$ (or $\Diffub(M;E)$ if $E=F$). \index{$\Diffub(M;E)$} Bounded geometry of differential operators is preserved by compositions and by taking transposes, and by taking formal adjoints in the case of Hermitian vector bundles of bounded geometry; in particular, $\Diffub(M;E)$ is a filtered subalgebra of $\Diff(M;E)$. Like in~\eqref{Diff^m(M E) equiv Diff^m(M)}, if $E$ is a line bundle of bounded geometry, then
\begin{equation}\label{Diffub^m(M E) equiv Diffub^m(M)}
\Diffub^m(M;E)\equiv\Diffub^m(M)\;.
\end{equation}

Every $A\in\Diffub^m(M;E)$ defines continuous linear maps $A:C^{m+k}_{\text{\rm ub}}(M;E)\to C^k_{\text{\rm ub}}(M;E)$ ($k\in\N_0$), which induce a continuous endomorphism $A$ of $\Cinftyub(M;E)$. It is said that $A$ is \emph{uniformly elliptic} if there is some $C\ge1$ such that, for all $p\in M$ and $\xi\in T^*_pM$, its leading symbol $\sigma_m(A)$ satisfies
\[
C^{-1}|\xi|^m\le|\sigma_m(A)(p,\xi)|\le C|\xi|^m\;.
\]
This condition is independent of the choice of the Hermitian metric of bounded geometry on $E$. Any $A\in\Diff^m_{\text{\rm ub}}(M;E,F)$ satisfies the second inequality. The case where $A\in\Diffub^m(M;E,F)$ is similar.

\subsubsection{Sobolev spaces of manifolds of bounded geometry}\label{sss: Sobolev bd geom}

For any Hermitian vector bundle $E$ of bounded geometry over $M$, any nonnegative symmetric uniformly elliptic $P\in\Diff^2_{\text{\rm ub}}(M;E)$ can be used to define the Sobolev space $H^s(M;E)$ ($s\in\R$) with a scalar product $\langle {\cdot},{\cdot}\rangle_s$ (\Cref{ss: Sobolev sps}). Any choice of $P$ defines the same Hilbertian space $H^s(M;E)$, which is a $\Cinftyub(M)$-module. In particular, $L^2(M;E)$ is the $\Cinftyub(M)$-tensor product of $L^2(M;\Omega^{1/2})$ and $\Cinftyub(M;E\otimes\Omega^{1/2})$, and $H^s(M;E)$ is the $\Cinftyub(M)$-tensor product of $H^s(M)$ and $\Cinftyub(M;E)$. For instance, we may take $P=\nabla^*\nabla$ for any unitary connection $\nabla$ of bounded geometry on $E$. For $s\in\N_0$, the Sobolev space $H^s(M)$ can be also described with the scalar product
\[
\langle u,v\rangle'_s=\sum_k\sum_{|\alpha|\le s}\int_Bf_k^2(x)\cdot\partial^\alpha(uy_{p_k}^{-1})(x)\cdot\overline{\partial^\alpha(vy_{p_k}^{-1})(x)}\,dx\;,
\]
using the partition of unity $\{f_k\}$ given by \Cref{p: f_k} \cite[Theorem~A.22]{Schick1996}, \cite[Propositions~3.2 and~3.3]{Schick2001}, \cite[Appendices~A1.2 and~A1.3]{Shubin1992}. A similar scalar product $\langle{\cdot},{\cdot}\rangle'_s$ can be defined for $H^s(M;E)$ with the help of local trivializations defining the bounded geometry of $E$. Every $A\in\Diffub^m(M;E)$ defines bounded operators $A:H^{m+s}(M;E)\to H^s(M;E)$ ($s\in\R$), which induce a continuous endomorphism $A$ of $H^{\pm\infty}(M;E)$. For any almost everywhere positive $h\in C^\infty(M)$, we have $hH^m(M;E)=H^m(M;E)$ if and only if $h>0$ and $h^{\pm1}\in\Cinftyub(M)$.

\begin{prop}[{Roe \cite[Proposition~2.8]{Roe1988I}}]\label{p: Sobolev embedding with bd geometry}
If $m'>m+n/2$, then $H^{m'}(M;E)\subset C_{\text{\rm ub}}^m(M;E)$, continuously. Thus $H^\infty(M;E)\subset \Cinftyub(M;E)$, continuously.
\end{prop}

\section{Symbols}\label{s: symbols}

The canonical coordinates of $\R^n\times\R^l$ ($n,l\in\N_0$) are denoted by $(x,\xi)=(x^1,\dots,x^n,\xi^1,\dots,\xi^l)$, and let $dx=dx^1\wedge\dots\wedge dx^n$ and $d\xi=d\xi^1\wedge\dots\wedge d\xi^l$. Recall that a \emph{symbol} of \emph{order} at most $m\in\R$ on $U\times\R^l$, or simply on $U$, is a function $a\in C^\infty(U\times\R^l)$ such that, for any compact $K\subset U$, and multi-indices $\alpha\in\N_0^n$ and $\beta\in\N_0^l$, 
\begin{equation}\label{| a |_K I J m}
\|a\|_{K,\alpha,\beta,m}:=\sup_{x\in K,\ \xi\in\R^l}\frac{|\partial_x^\alpha \partial_\xi^\beta a(x,\xi)|}{(1+|\xi|)^{m-|\beta|}}<\infty\;.
\end{equation}
The set of symbols of order at most $m$, $S^m(U\times\R^l)$, \index{$S^m(U\times\R^l)$} becomes a Fr\'echet space with the semi-norms $\|{\cdot}\|_{K,\alpha,\beta,m}$ given by~\eqref{| a |_K I J m}. There are continuous inclusions
\begin{equation}\label{S^m(T^*M) subset S^m'(T^*M)}
S^m(U\times\R^l)\subset S^{m'}(U\times\R^l)\quad(m<m')\;,
\end{equation}
giving rise to the LCSs \index{$S^{\pm\infty}(U\times\R^l)$}
\[
S^\infty(U\times\R^l)=\bigcup_mS^m(U\times\R^l)\;,\quad S^{-\infty}(U\times\R^l)=\bigcap_mS^m(U\times\R^l)\;.
\]
The LF-space $S^\infty(U\times\R^l)$ is a filtered algebra and $C^\infty(U)$-module with the pointwise multiplication. The Fr\'echet space $S^{-\infty}(U\times\R^l)$ is a filtered ideal and $C^\infty(U)$-submodule of $S^\infty(U\times\R^l)$. The homogeneous components of the corresponding graded algebra are \index{$S^{(m)}(U\times\R^l)$}
\[
S^{(m)}(U\times\R^l)=S^m(U\times\R^l)/S^{m-1}(U\times\R^l)\;.
\]
When $U=\R^0=\{0\}$, the notation $S^m(\R^l)$, $S^{\pm\infty}(\R^l)$ and $S^{(m)}(\R^l)$ is used, and the subscripts $K$ and $\alpha$ are omitted from the notation of the semi-norms in~\eqref{| a |_K I J m}.

Since $S^\infty(U\times\R^l)$ is an LF-space, we get the following (see \Cref{ss: smooth/distributional sections}).

\begin{prop}\label{p: S^infty(U x R^l) is barreled}
$S^\infty(U\times\R^l)$ is barreled, ultrabornological and webbed.
\end{prop}

There are continuous inclusions (see \Cref{ss: smooth/distributional sections} for the definition of $\Cinftycv(U\times\R^l)$)
\begin{equation}\label{S^infty(U x R^l) subset C^infty(U x R^l)}
\Cinftycv(U\times\R^l)\subset S^{-\infty}(U\times\R^l)\;,\quad S^\infty(U\times\R^l)\subset C^\infty(U\times\R^l)\;;
\end{equation}
in particular, $S^\infty(U\times\R^l)$ is Hausdorff. According to~\eqref{| u |_K C^k} and~\eqref{S^infty(U x R^l) subset C^infty(U x R^l)}, we get continuous semi-norms $\|{\cdot}\|_{Q,C^k}$ on $S^\infty(U\times\R^l)$, for any compact $Q\subset U\times\R^l$ and $k\in\N_0$, given by
\begin{equation}\label{| u |_Q C^k on S^m(U x R^l)}
\|a\|_{Q,C^k}=\sup_{(x,\xi)\in Q,\ |\alpha|+|\beta|\le k}|\partial_x^\alpha \partial_\xi^\beta a(x,\xi)|\;.
\end{equation}
With the notation of~\eqref{| a |_K I J m}, consider also the continuous semi-norms $\|{\cdot}\|'_{K,\alpha,\beta,m}$ on $S^m(U\times\R^l)$ given by
\begin{equation}\label{| a |'_K I J m}
\|a\|'_{K,\alpha,\beta,m}=\sup_{x\in K}\limsup_{|\xi|\to\infty}\frac{\big|\partial_x^\alpha \partial_\xi^\beta a(x,\xi)\big|}{|\xi|^{m-|\beta|}}\;.
\end{equation}
In the case of $S^m(\R^l)$, the subscripts $K$ and $\alpha$ are omitted from the notation of the semi-norms~\eqref{| a |'_K I J m}.

\begin{prop}\label{p: topology of S^m(U x R^l)}
The semi-norms~\eqref{| u |_Q C^k on S^m(U x R^l)} and~\eqref{| a |'_K I J m} together describe the topology of $S^m(U\times\R^l)$.
\end{prop}

\begin{proof}
Let $S^{\prime\,m}(U\times\R^l)$ denote the LCHS defined by endowing the vector space $S^m(U\times\R^l)$ with the topology induced by the semi-norms~\eqref{| u |_Q C^k on S^m(U x R^l)} and~\eqref{| a |'_K I J m} together; in fact, countably many semi-norms of these types are enough to describe its topology (taking exhausting increasing sequences of compact sets), and therefore $S^{\prime\,m}(U\times\R^l)$ is metrizable. Let $\widehat S^{\prime\,m}(U\times\R^l)$ denote its completion, where the stated semi-norms have continuous extensions. There is a continuous inclusion $S^{\prime\,m}(U\times\R^l)\subset C^\infty(U\times\R^l)$, which can be extended to a continuous map $\phi:\widehat S^{\prime\,m}(U\times\R^l)\to C^\infty(U\times\R^l)$ because $C^\infty(U\times\R^l)$ is complete. For any $a\in\widehat S^{\prime\,m}(U\times\R^l)$, and $K$, $\alpha$ and $\beta$ like in~\eqref{| a |'_K I J m}, since $\|\phi(a)\|'_{K,\alpha,\beta,m}=\|a\|'_{K,\alpha,\beta,m}<\infty$, there are $C,R>0$ so that, if $x\in K$ and $|\xi|\ge R$, then
\[
\frac{|\partial_x^\alpha \partial_\xi^\beta \phi(a)(x,\xi)|}{(1+|\xi|)^{m-|\beta|}}\le C\;.
\]
Let $B_R\subset\R^l$ denote the open ball of center $0$ and radius $R$. For $Q=K\times\overline{B_R}\subset U\times\R^l$ and $k=|\alpha|+|\beta|$, since $\|\phi(a)\|_{Q,C^k}=\|a\|_{Q,C^k}<\infty$, there is some $C'>0$ such that $|\partial_x^\alpha \partial_\xi^\beta \phi(a)(x,\xi)|<C'$ for $(x,\xi)\in Q$, yielding 
\[
\frac{|\partial_x^\alpha \partial_\xi^\beta \phi(a)(x,\xi)|}{(1+|\xi|)^{m-|\beta|}}\le
\begin{cases}
C' & \text{if $|\beta|\le m$}\\
C'(1+R)^{|\beta|-m} & \text{if $|\beta|\ge m$}\;.
\end{cases}
\]
This shows that $\|\phi(a)\|_{K,\alpha,\beta,m}<\infty$, obtaining that $a\equiv\phi(a)\in S^m(U\times\R^l)$. Hence $S^{\prime\,m}(U\times\R^l)$ is complete, and therefore it is a Fr\'echet space. Thus the identity map $S^m(U\times\R^l)\to S^{\prime\,m}(U\times\R^l)$ is a continuous linear isomorphism between Fr\'echet spaces, obtaining that it is indeed a homeomorphism by a version of the open mapping theorem \cite[Section~15.12]{Kothe1969-I}, \cite[Theorem~II.2.1]{Schaefer1971}, \cite[Theorem~14.4.6]{NariciBeckenstein2011}.
\end{proof}

\begin{prop}\label{p: | . |_K I J m' = 0 on S^m(U x R^l)}
For $m,m'\in\N_0$, $\alpha\in\N_0^n$, $\beta\in\N_0^l$ and any compact $K\subset U$, if $m<m'$, then $\|{\cdot}\|'_{K,\alpha,\beta,m'}=0$ on $S^m(U\times\R^l)$.
\end{prop}

\begin{proof}
According to~\eqref{| a |'_K I J m}, for all $a\in S^m(U\times\R^l)$,
\[
\|a\|'_{K,\alpha,\beta,m'}=\|a\|'_{K,\alpha,\beta,m}\lim_{|\xi|\to\infty}|\xi|^{m-m'}=0\;.\qedhere
\]
\end{proof}

\begin{cor}\label{c: coincidence of tops on S^m(U x R^l)}
For $m<m'$, the topologies of $S^{m'}(U\times\R^l)$ and $C^\infty(U\times\R^l)$ coincide on $S^m(U\times\R^l)$. Therefore the topologies of $S^\infty(U\times\R^l)$ and $C^\infty(U\times\R^l)$ coincide on $S^m(U\times\R^l)$.
\end{cor}

\begin{proof}
The first assertion is a consequence of \Cref{p: topology of S^m(U x R^l),p: | . |_K I J m' = 0 on S^m(U x R^l)}.

To prove the second assertion, by~\eqref{S^infty(U x R^l) subset C^infty(U x R^l)}, it is enough to show that the topology of $S^\infty(U\times\R^l)$ is finer or equal than the topology of $C^\infty(U\times\R^l)$ on $S^m(U\times\R^l)$. For every open $O\subset S^\infty(U\times\R^l)$ and $m'>m$, since $O\cap S^{m'}(U\times\R^l)$ is open in $S^{m'}(U\times\R^l)$, it follows from the first assertion that there is some open $P\subset C^\infty(U\times\R^l)$ such that $O\cap S^m(U\times\R^l)=P\cap S^m(U\times\R^l)$.
\end{proof}

\begin{cor}\label{c: Cinftyc(U x R^l) is dense in S^infty(U x R^l)}
For $m<m'$, $\Cinftyc(U\times\R^l)$ is dense in $S^m(U\times\R^l)$ with the topology of $S^{m'}(U\times\R^l)$. Therefore $\Cinftyc(U\times\R^l)$ is dense in $S^\infty(U\times\R^l)$.
\end{cor}

\begin{proof}
The first assertion is given by \Cref{c: coincidence of tops on S^m(U x R^l)} and the density of $\Cinftyc(U\times\R^l)$ in $C^\infty(U\times\R^l)$.

To prove the second assertion, take any open $O\ne\emptyset$ in $S^\infty(U\times\R^l)$. We have $O\cap S^m(U\times\R^l)\ne\emptyset$ for some $m$. This intersection is open in $S^m(U\times\R^l)$ with the topology of any $S^{m'}(U\times\R^l)$ for all $m'\ge m$. So $O\cap\Cinftyc(U\times\R^l)\ne\emptyset$ by the first assertion.
\end{proof}

\begin{cor}\label{c: S^infty(U x R^l) is acyclic and Montel}
$S^\infty(U\times\R^l)$ is an acyclic Montel space, and therefore complete, boundedly retractive and reflexive.
\end{cor}

\begin{proof}
\Cref{c: coincidence of tops on S^m(U x R^l)} gives the property of being acyclic, and therefore complete and boundedly retractive (\Cref{ss: TVS}). Since $S^\infty(U\times\R^l)$ is barreled (\Cref{p: S^infty(U x R^l) is barreled}) and every Montel space is reflexive \cite[6.27.2~(1)]{Kothe1969-I}, \cite[Section~8.4.7]{Edwards1965}, \cite[after the examples of~IV.5.8]{Schaefer1971}, it only remains to prove that $S^\infty(U\times\R^l)$ is semi-Montel.

Take any closed bounded subset $B\subset S^\infty(U\times\R^l)$; in particular, $B$ is complete because $S^\infty(U\times\R^l)$ is complete. Since $S^\infty(U\times\R^l)$ is boundedly retractive, $B$ is contained and bounded in some $S^m(U\times\R^l)$, and the topologies of $S^\infty(U\times\R^l)$ and $S^m(U\times\R^l)$ coincide on $B$. By \Cref{c: coincidence of tops on S^m(U x R^l)}, it follows that $B$ is a complete bounded subspace of $C^\infty(U\times\R^l)$, and therefore closed because $C^\infty(U\times\R^l)$ is complete. So $B$ is compact because $C^\infty(U\times\R^l)$ is a Montel space.
\end{proof}

\begin{rem}\label{r: Cinftyc(U x R^l) is dense in S^m(U x R^l)}
Another proof of \Cref{c: Cinftyc(U x R^l) is dense in S^infty(U x R^l)} could be given like in \Cref{p: C^infty_c(M) is dense in x^m L^infty(M)}.
\end{rem}

\begin{rem}\label{r: S^infty(U x R^l) subset C^infty(U x R^l) is not a TVS-embedding}
Despite of \Cref{c: coincidence of tops on S^m(U x R^l)}, the following argument shows that the second inclusion of~\eqref{S^infty(U x R^l) subset C^infty(U x R^l)} is not a TVS-embedding. Let $a_m\in S^\infty(U\times\R^l)$ ($m\in\N_0$) such that $a_m(x,\xi)=0$ if $|\xi^1|\le m$, and $a_m(x,\xi)=(\xi^1-m)^m$ if $|\xi^1|\ge m+1$. Then $a_m\in S^m(U\times\R^l)\setminus S^{m-1}(U\times\R^l)$ and $a_m\to0$ in $C^\infty(U\times\R^l)$ as $m\uparrow\infty$. However $a_m\not\to0$ in $S^\infty(U\times\R^l)$; otherwise, since $S^\infty(U\times\R^l)$ is sequentially retractive (\Cref{c: S^infty(U x R^l) is acyclic and Montel}), all $a_m$ would lie in some step $S^{m_0}(U\times\R^l)$, a contradiction.
\end{rem}

With more generality, a symbol of order $m$ on a vector bundle $E$ over $M$ is a smooth function on $E$ satisfying~\eqref{| a |_K I J m} via charts of $M$ and local trivializations of $E$, with $K$ contained in the domains of charts where $E$ is trivial. As above, they form a Fr\'echet space $S^m(E)$ with the topology described by the semi-norms given by this version of~\eqref{| a |_K I J m}. The version of~\eqref{S^m(T^*M) subset S^m'(T^*M)} in this setting is true, obtaining the corresponding inductive and projective limits $S^{\pm\infty}(E)$, and quotient spaces $S^{(m)}(E)$. We can similarly define the norms~\eqref{| u |_Q C^k on S^m(U x R^l)} and~\eqref{| a |'_K I J m} on $S^m(E)$, and \Cref{p: topology of S^m(U x R^l),p: | . |_K I J m' = 0 on S^m(U x R^l),c: coincidence of tops on S^m(U x R^l),c: Cinftyc(U x R^l) is dense in S^infty(U x R^l),c: S^infty(U x R^l) is acyclic and Montel} can be directly extended to this setting.

Given another vector bundle $F$ over $M$, we can further take the $C^\infty(M)$-tensor product of these spaces with $C^\infty(M;F)$, obtaining the spaces $S^m(E;F)$, $S^{\pm\infty}(E;F)$ and $S^{(m)}(E;F)$, satisfying analogous properties and results. Now~\eqref{S^infty(U x R^l) subset C^infty(U x R^l)} becomes
\[
\Cinftycv(E;\pi^*F)\subset S^{-\infty}(E;F)\;,\quad S^{\infty}(E;F)\subset C^\infty(E;\pi^*F)\;,
\]
where $\pi:E\to M$ is the vector bundle projection.

\section{Conormal distributions}\label{s: conormal distribs}

\subsection{Differential operators tangent to a submanifold}\label{ss: Diff(M L)}

Let $L$ is a regular submanifold of $M$ of codimension $n'$ and dimension $n''$, which is a closed subset. Let $\fX(M,L)\subset\fX(M)$ \index{$\fX(M,L)$} be the Lie subalgebra and $C^\infty(M)$-submodule of vector fields tangent to $L$. Using $\fX(M,L)$ instead of $\fX(M)$, we can define the filtered subalgebra and $C^\infty(M)$-submodule $\Diff(M,L)\subset\Diff(M)$ \index{$\Diff(M,L)$} like in \Cref{ss: diff ops}. We have
\begin{equation}\label{A in Diff(M L) => A^t in Diff(M L Omega)}
A\in\Diff(M,L)\Rightarrow A^t\in\Diff(M,L;\Omega)\;.
\end{equation}
By the conditions on $L$, every $\Diff^m(M,L)$ ($m\in\N_0$) is locally finitely $C^\infty(M)$-generated, and therefore $\Diff(M,L)$ is countably $C^\infty(M)$-generated. The surjective restriction map $\fX(M,L)\to\fX(L)$, $X\mapsto X|_L$, induces a surjective linear restriction map of filtered algebras and $C^\infty(M)$-modules,
\begin{equation}\label{Diff(M L) -> Diff(L)}
\Diff(M,L)\to\Diff(L)\;,\quad A\mapsto A|_L\;.
\end{equation}

Let $(U,x)$ be a chart of $M$ adapted to $L$; i.e., for open subsets $U'\subset\R^{n'}$ and $U''\subset\R^{n''}$,
\begin{gather*}
x=(x^1,\dots,x^n)\equiv(x',x''):U\to U'\times U''\;,\\
x'=(x'^1,\dots,x'^{n'})\;,\quad x''=(x''^{1},\dots,x''^{n''})\;,\quad L_0:=L\cap U=\{x'=0\}\;.
\end{gather*}
If $L$ is of codimension one, then we will use the notation $(x,y)$ instead of $(x',x'')$. For every $m\in\N_0$, $\Diff^m(U,L_0)$ is $C^\infty(U)$-spanned by the operators $x^{\prime\alpha}\partial_{x'}^\beta\partial_{x''}^\gamma$ with $|\beta|+|\gamma|\le m$ and $|\alpha|=|\beta|$; we may use the generators $\partial_{x'}^\beta\partial_{x''}^\gamma x^{\prime\alpha}$ as well, with the same conditions on the multi-indices.

\subsection{Conormal distributions filtered by Sobolev order}\label{ss: conormal - Sobolev order}

\subsubsection{Case of compact manifolds}\label{sss: conormal - Sobolev order - compact}

Suppose first that $M$ is compact. Then the space of \emph{conormal distributions} at $L$ of \emph{Sobolev order} at most $s\in\R$ is the LCS and $C^\infty(M)$-module \index{$I^{(s)}_{{\cdot}/\text{\rm c}}(M,L)$} \index{$I_{{\cdot}/\text{\rm c}}(M,L)$}
\begin{equation}\label{I^(s)(M,L)}
I^{(s)}(M,L)=\{\,u\in C^{-\infty}(M)\mid\Diff(M,L)\,u\subset H^s(M)\,\}\;,
\end{equation}
with the projective topology given by the maps $P:I^{(s)}(M,L)\to H^s(M)$ ($P\in\Diff(M,L)$). 

\begin{prop}\label{p: I^(s)(M L) is a totally reflexive Frechet sp}
$I^{(s)}(M,L)$ is a totally reflexive Fr\'echet space.
\end{prop}

\begin{proof}
For any countable $C^\infty(M)$-spanning set $\{P_j\mid j\in\N_0\}$ of $\Diff(M,L)$, the space $I^{(s)}(M,L)$ has the projective topology given by the maps $P_j:I^{(s)}(M,L)\to H^s(M)$. Let
\[
I_k^{(s)}(M,L)=\{\,u\in C^{-\infty}(M)\mid P_ju\subset H^s(M),\ j=0,\dots,k\,\}\;,
\]
with the projective topology given by the maps $P_j:I^{(s)}(M,L)\to H^s(M)$ ($j=0,\dots,k$). We can assume $P_0=1$, and therefore $I_0^{(s)}(M,L)=H^s(M)$. Every $I_k^{(s)}(M,L)$ is a Hilbert space with the scalar product
\[
\langle u,v\rangle_{s,k}=\sum_{j=0}^k\langle P_ju,P_jv\rangle_s\;,
\]
there are continuous inclusions $I_{k'}^{(s)}(M,L)\subset I_k^{(s)}(M,L)$ ($k<k'$), and $I^{(s)}(M,L)=\bigcap_kI_k^{(s)}(M,L)$. So $I^{(s)}(M,L)$ is a totally reflexive Fr\'echet space \cite[Theorem~4]{Valdivia1989}.
\end{proof}

We have continuous inclusions
\begin{equation}\label{I^(s)(M L) subset I^(s')(M L)}
I^{(s)}(M,L)\subset I^{(s')}(M,L)\quad(s'<s)\;,
\end{equation}
and consider the LCSs and $C^\infty(M)$-modules \index{$I(M,L)$} \index{$I^{(\infty)}(M,L)$}
\[
I(M,L)=\bigcup_sI^{(s)}(M,L)\;,\quad I^{(\infty)}(M,L)=\bigcap_sI^{(s)}(M,L)\;.
\]
$I(M,L)$ is an LF-space, and $I^{(\infty)}(M,L)$ is a Fr\'echet space and submodule of $I(M,L)$. The elements of $I(M,L)$ are called \emph{conormal distributions} of $M$ at $L$ (or of $(M,L)$). The spaces $I^{(s)}(M,L)$ form the \emph{Sobolev order filtration} of $I(M,L)$. From~\eqref{I^(s)(M,L)}, it follows that there are canonical continuous inclusions,
\begin{equation}\label{C^infty(M) subset I^(infty)(M L)}
C^\infty(M)\subset I^{(\infty)}(M,L)\;,\quad I(M,L)\subset C^{-\infty}(M)\;;
\end{equation}
in particular, $I(M,L)$ is Hausdorff.

Since every $I^{(s)}(M,L)$ is a Fr\'echet space (\Cref{p: I^(s)(M L) is a totally reflexive Frechet sp}), the following analog of \Cref{p: S^infty(U x R^l) is barreled} holds true by the same reason.

\begin{cor}\label{c: I(M L) is barreled}
$I(M,L)$ is barreled, ultrabornological and webbed.
\end{cor}

\subsubsection{Extension to non-compact manifolds}\label{sss: conormal - Sobolev order - non-compact}

If $M$ is not assumed to be compact, we can similarly define the LCHS $I^{(s)}_{{\cdot}/\text{\rm c}}(M,L)$ by using $C^{-\infty}_{{\cdot}/\text{\rm c}}(M)$ and $H^s_{\text{\rm loc/c}}(M)$. Every $I^{(s)}(M,L)$ is a Fr\'echet space, as follows like in the proof of \Cref{p: I^(s)(M L) is a totally reflexive Frechet sp}, using the Fr\'echet spaces $H^s_{\text{\rm loc}}(M)$. We can describe $I^{(s)}_{\text{\rm c}}(M,L)=\bigcup_KI^{(s)}_K(M,L)$ like in~\eqref{Cinftyc(U)}, which is a strict LF-space because every $I^{(s)}_K(M,L)$ satisfies an analog of \Cref{p: I^(s)(M L) is a totally reflexive Frechet sp}. Therefore $I_{\text{\rm c}}(M,L)=\bigcup_sI_{\text{\rm c}}^{(s)}(M,L)$ is an LF-space \cite[Exercise~12.108]{NariciBeckenstein2011}; moreover $I_{\text{\rm c}}(M,L)=\bigcup_KI_K(M,L)$. We also have the LCHS $I^{(\infty)}_{\text{\rm c}}(M,L)=\bigcap_sI_{\text{\rm c}}^{(s)}(M,L)$. All of these spaces are modules over $C^\infty(M)$; $I_{\text{\rm c}}(M,L)$ is a filtered module and $I^{(\infty)}_{\text{\rm c}}(M,L)$ a submodule. The extension by zero defines a continuous inclusion $I_{\text{\rm c}}(U,L\cap U)\subset I_{\text{\rm c}}(M,L)$ for any open $U\subset M$. We also define the space $I^{(\infty)}(M,L)$ like in the compact case, as well as the space $\bigcup_sI^{(s)}(M,L)$, which consists of the conormal distributions with a Sobolev order. But now let (cf.\ \cite[Definition~18.2.6]{Hormander1985-III})
\begin{equation}\label{I(M L) - non-compact M}
I(M,L)=\{\,u\in C^{-\infty}(M)\mid\Cinftyc(M)\,u\subset I_{\text{\rm c}}(M,L)\,\}\;,
\end{equation}
which is a LCS with the projective topology given by the (multiplication) maps $f_j:I(M,L)\to I_{\text{\rm c}}(M,L)$, for a countable partition of unity $\{f_j\}\subset\Cinftyc(M)$. We have $I(M,L)=\bigcup_sI^{(s)}(M,L)$ if and only if $L$ is compact; thus the spaces $I^{(s)}(M,L)$ form a filtration of $I(M,L)$ just when $L$ is compact. There is an extension of~\eqref{C^infty(M) subset I^(infty)(M L)} for non-compact $M$, taking arbitrary/compact support; in particular, $I_{{\cdot}/\text{\rm c}}(M,L)$ is Hausdorff.

\subsection{Filtration of $I(M,L)$ by the symbol order}\label{ss: conormal - symbol order}

\subsubsection{Local description of conormal distributions with symbols}\label{sss: conormal - symbol order - local}

Consider the notation of \Cref{ss: Diff(M L)} for a chart $(U,x=(x',x''))$ of $M$ adapted to $L$. We use the identity $U''\times\R^{n'}\equiv N^*U''$, and the symbol spaces $S^m(U''\times\R^{n'})\equiv S^m(N^*U'')$ (\Cref{s: symbols}). Define
\begin{gather}
\Cinftycv(N^*U'')\to C^\infty(U)\;,\quad a\mapsto u\;,\label{a -> u}\\
\Cinftyc(U)\to C^\infty(N^*U'')\;,\quad u\mapsto a\;,\label{u -> a}
\end{gather}
by the following partial inverse Fourier transform and partial Fourier transform:
\begin{align*}
u(x)&=(2\pi)^{-n'}\int_{\R^{n'}}e^{i\langle x',\xi\rangle}a(x'',\xi)\,d\xi\;,\\
a(x'',\xi)&=\int_{\R^{n'}}e^{-i\langle x',\xi\rangle}u(x',x'')\,dx'\;.
\end{align*}

\begin{prop}[{\cite[Theorem~18.2.8]{Hormander1985-III}, \cite[Proposition~6.1.1]{Melrose1996}, \cite[Lemma~9.33]{MelroseUhlmann2008}}]\label{p: symbol local expression of conormal distribs}
If $s<-\bar m-n'/2$, then~\eqref{a -> u} has a continuous extension $S^{\bar m}(N^*U'')\to I^{(s)}(U,L_0)$. If $\bar m>-s-n'/2$, then~\eqref{u -> a} induces a continuous linear map $I_{\text{\rm c}}^{(s)}(U,L_0)\to S^{\bar m}(N^*U'')$.
\end{prop}

\begin{rem}\label{r:  symbol local expression of conormal distribs}
The continuity of the maps of \Cref{p: symbol local expression of conormal distribs} is not stated in \cite[Theorem~18.2.8]{Hormander1985-III}, \cite[Proposition~6.1.1]{Melrose1996}, \cite[Lemma~9.33]{MelroseUhlmann2008}, but it follows easily from their proofs.
\end{rem}

When applying \Cref{p: symbol local expression of conormal distribs} to $M$ via $(U,x)$, it will be convenient to use
\[
a\,|d\xi|\in S^{\bar m}(N^*U'';\Omega N^*U'')\equiv S^{\bar m}(N^*L_0;\Omega N^*L_0)\;.
\]

\subsubsection{Case of compact manifolds}\label{sss: conormal - symbol order - compact}

Assume first that $M$ is compact. Take a finite cover of $L$ by relatively compact charts $(U_j,x_j)$ of $M$ adapted to $L$, and write $L_j=L\cap U_j$. Let $\{h,f_j\}$ be a $C^\infty$ partition of unity of $M$ subordinated to the open covering $\{M\setminus L,U_j\}$. Then $I(M,L)$ consists of the distributions $u\in C^{-\infty}(M)$ such that $hu\in C^\infty(M\setminus L)$ and $f_ju\in I_{\text{\rm c}}(U_j,L_j)$ for all $j$. According to \Cref{p: symbol local expression of conormal distribs}, every $f_ju$ is given by some $a_j\in S^\infty(N^*L_j;\Omega N^*L_j)$. For
\begin{equation}\label{bar m}
\bar m=m+n/4-n'/2\;,
\end{equation}
the condition $a_j\in S^{\bar m}(N^*L_j;\Omega N^*L_j)$ describes the elements $u$ of a $C^\infty(M)$-submodule $I^m(M,L)\subset I(M,L)$, \index{$I^m(M,L)$} which is independent of the choices involved \cite[Proposition~9.33]{MelroseUhlmann2008} (see also \cite[Definition~6.2.19]{Melrose1996} and \cite[Definition~4.3.9]{Simanca1990}). Moreover, applying the versions of semi-norms~\eqref{| u |_K C^k} on $C^\infty(M\setminus L)$ to $hu$ and versions of semi-norms~\eqref{| a |_K I J m} on $S^{\bar m}(N^*L_j;\Omega N^*L_j)$ to every $a_j$, we get semi-norms on $I^m(M,L)$, which becomes a Fr\'echet space \cite[Sections~6.2 and 6.10]{Melrose1996}. In other words, the following map is required to be a TVS-embedding:
\begin{equation}\label{I^m(M L) -> ...}
I^m(M,L)\to C^\infty(M\setminus L)\oplus\prod_jS^{\bar m}(N^*L_j;\Omega N^*L_j)\;,\quad u\mapsto(hu,(a_j))\;.
\end{equation}

The version of~\eqref{S^m(T^*M) subset S^m'(T^*M)} for the spaces $S^{\bar m}(N^*L_j;\Omega N^*L_j)$ gives continuous inclusions
\begin{equation}\label{I^m(M L) subset I^m'(M L)}
I^m(M,L)\subset I^{m'}(M,L)\quad(m<m')\;.
\end{equation}
The element $\sigma_m(u)\in S^{(\bar m)}(N^*L;\Omega N^*L)$ represented by $\sum_ja_j\in S^{\bar m}(N^*L;\Omega N^*L)$ is called the \emph{principal symbol} of $u$. This defines the exact sequence
\[
0\to I^{m-1}(M,L)\hookrightarrow I^m(M,L) \xrightarrow{\sigma_m} S^{(\bar m)}(N^*L;\Omega N^*L)\to0\;.
\]
 From \Cref{p: symbol local expression of conormal distribs} and~\eqref{bar m}, we also get continuous inclusions
\begin{equation}\label{sandwich for I}
I^{(-m-n/4+\epsilon)}(M,L)\subset I^m(M,L)\subset I^{(-m-n/4-\epsilon)}(M,L)\;,
\end{equation}
for all $m\in\R$ and $\epsilon>0$ (cf.\ \cite[Eq.~(6.2.5)]{Melrose1996}, \cite[Eq.~(9.35)]{MelroseUhlmann2008}). So \index{$I^{(\infty)}(M,L)$} \index{$I^{-\infty}(M,L)$}
\[
I(M,L)=\bigcup_mI^m(M,L)\;,\quad I^{(\infty)}(M,L)=I^{-\infty}(M,L):=\bigcap_mI^m(M,L)\;.
\]
The spaces $I^m(M,L)$ form the \emph{symbol order filtration} of $I(M,L)$. The maps~\eqref{I^m(M L) -> ...} induce a TVS-embedding
\begin{equation}\label{I(M L) -> ...}
I(M,L)\to C^\infty(M\setminus L)\oplus\prod_jS^\infty(N^*L_j;\Omega N^*L_j)\;.
\end{equation}

\begin{cor}\label{c: coincidence of tops on I^m(M L)}
For $m<m',m''$, the topologies of $I^{m'}(M,L)$ and $I^{m''}(M,L)$ coincide on $I^m(M,L)$.
\end{cor}

\begin{proof}
Use \Cref{c: coincidence of tops on S^m(U x R^l)} and the TVS-embeddings~\eqref{I^m(M L) -> ...}.
\end{proof}

\begin{cor}[{\cite[Eq.~(6.2.12)]{Melrose1996}}]\label{c: C^infty(M) is dense in I(M L)}
For $m<m'$, $C^\infty(M)$ is dense in $I^m(M,L)$ with the topology of $I^{m'}(M,L)$. Therefore $C^\infty(M)$ is dense in $I(M,L)$.
\end{cor}

\begin{proof}
$C^\infty(M)$ is contained in the stated spaces by~\eqref{C^infty(M) subset I^(infty)(M L)}.

Let us prove the first density, and the second one follows like in \Cref{c: Cinftyc(U x R^l) is dense in S^infty(U x R^l)}. Given $u\in I^m(M,L)$, let $a_j\in S^{\bar m}(N^*L_j;\Omega N^*L_j)$ be the symbol corresponding to $f_ju$ by \Cref{p: symbol local expression of conormal distribs}, like in~\eqref{I^m(M L) -> ...}. By \Cref{c: Cinftyc(U x R^l) is dense in S^infty(U x R^l)}, there is a sequence $b_{j,k}\in\Cinftyc(N^*L_j;\Omega N^*L_j)$ converging to $a_j$ in $S^{\bar m'}(N^*L_j;\Omega N^*L_j)$ ($\bar m'=m'+n/4-n'/2$). Let $v_{j,k}$ be the sequence in $C^\infty(U)$ that corresponds to $b_{j,k}$ via~\eqref{a -> u}; it converges to $f_ju$ in $I^{m'}(U,L)$ as $k\to\infty$ by \Cref{p: symbol local expression of conormal distribs}. Take functions $\tilde f_j\in\Cinftyc(U_j)$ with $\tilde f_j=1$ on $\supp f_j$. Then $\tilde f_jv_{j,k}\to f_ju$ in $I^{m'}_{\text{\rm c}}(U,L)$, and therefore $hu+\sum_j\tilde f_jv_{j,k}\in C^\infty(M)$ is convergent to $u$ in $I^{m'}(M,L)$.
\end{proof}

\begin{cor}\label{c: I(M L) is acyclic and Montel}
$I(M,L)$ is an acyclic Montel space, and therefore complete, boundedly retractive and reflexive.
\end{cor}

\begin{proof}
Like in  \Cref{c: S^infty(U x R^l) is acyclic and Montel}, by \Cref{c: I(M L) is barreled,c: coincidence of tops on I^m(M L)}, it is enough to prove that $I(M,L)$ is semi-Montel. The TVS-embedding~\eqref{I(M L) -> ...} is closed because $I(M,L)$ is complete. Then $I(M,L)$ is semi-Montel because $C^\infty(M\setminus L)$ and $S^\infty(N^*L_j;\Omega N^*L_j)$ are Montel spaces (\Cref{c: S^infty(U x R^l) is acyclic and Montel}), and this property is inherited by closed subspaces and products \cite[Propositions~3.9.3 and~3.9.4]{Horvath1966-I}, \cite[Exercise~12.203~(c)]{NariciBeckenstein2011}.
\end{proof}

\begin{rem}\label{r: I(M L) is reflexive}
The reflexivity of $I(M,L)$ is also a consequence of the reflexivity of $I^{(s)}(M,L)$ (\Cref{p: I^(s)(M L) is a totally reflexive Frechet sp}) and the regularity of $I(M,L)$ (\Cref{c: I(M L) is acyclic and Montel}) \cite{Kucera2004}.
\end{rem}

\subsubsection{Extension to non-compact manifolds}\label{sss: conormal - symbol order - non-compact}

When $M$ is not assumed to be compact, the definition of $I^m(M,L)$ can be immediately extended assuming $\{U_j\}$ is a locally finite cover of $L$, obtaining an analog of~\eqref{I^m(M L) -> ...}. We can similarly define $I^m_K(M,L)$ for all compact $K\subset M$, and take $I^m_{\text{\rm c}}(M,L)=\bigcup_KI^m_K(M,L)$ like in~\eqref{Cinftyc(U)}. The space of conormal distributions with a symbol order is $\bigcup_mI^m(M,L)$, and let $I^{-\infty}_{{\cdot}/\text{\rm c}}(M,L)=\bigcap_mI^m_{{\cdot}/\text{\rm c}}(M,L)$. There are extensions of~\eqref{I^m(M L) -> ...}--\eqref{I(M L) -> ...} and \Cref{c: coincidence of tops on I^m(M L),c: C^infty(M) is dense in I(M L)}, with arbitrary/compact support (using direct sums instead of products in the case of compact support). So $\bigcup_mI^m(M,L)=\bigcup_sI^{(s)}(M,L)$, $I_{\text{\rm c}}(M,L)=\bigcup_mI^m_{\text{\rm c}}(M,L)$ and $I^{(\infty)}_{{\cdot}/\text{\rm c}}(M,L)=I^{-\infty}_{{\cdot}/\text{\rm c}}(M,L)$. \Cref{c: I(M L) is acyclic and Montel} has extensions for $\bigcup_mI^m(M,L)$ and $I_{{\cdot}/\text{\rm c}}(M,L)$, except acyclicity in the case of $I(M,L)$.

\subsection{Dirac sections at submanifolds}\label{ss: Dirac sections}

Let $NL$ and $N^*L$ denote the normal and conormal bundles of $L$. We have $\Omega NL\otimes\Omega L\equiv\Omega_LM$. The transpose of the restriction map $C^\infty_{\text{\rm c}/{\cdot}}(M;E^*\otimes\Omega M)\to C^\infty_{\text{\rm c}/{\cdot}}(L;E^*\otimes\Omega_LM)$ is a continuous inclusion \index{$\delta_L^u$}
\begin{gather}
C^{-\infty}_{{\cdot}/\text{\rm c}}(L;E\otimes\Omega^{-1}NL)\subset C^{-\infty}_{{\cdot}/\text{\rm c}}(M;E)\;,
\label{u in C^-infty_c mapsto delta_L^u}\\
u\mapsto\delta_L^u\;,\quad\langle\delta_L^u,v\rangle=\langle u,v|_L\rangle\;,\quad 
v\in C^\infty_{\text{\rm c}/{\cdot}}(M;E^*\otimes\Omega)\;.\notag
\end{gather}
By restriction of~\eqref{u in C^-infty_c mapsto delta_L^u}, we get a continuous inclusion \cite[p.~310]{GuilleminSternberg1977},
\begin{equation}\label{u in C^infty_c mapsto delta_L^u}
C^\infty_{{\cdot}/\text{\rm c}}(L;E\otimes\Omega^{-1}NL)\subset C^{-\infty}_{{\cdot}/\text{\rm c}}(M;E)\;;
\end{equation}
in this case, we can write $\langle\delta_L^u,v\rangle=\int_Lu\,v|_L$. This is the subspace of \emph{$\delta$-sections} or \emph{Dirac sections} at $L$. Actually, the following sharpening of~\eqref{u in C^infty_c mapsto delta_L^u} is true.

\begin{prop}\label{p: C^infty_./c(L ...) subset H^s_./c(M E)}
The inclusion~\eqref{u in C^infty_c mapsto delta_L^u} induces a continuous injection
\[
C^\infty_{{\cdot}/\text{\rm c}}(L;E\otimes\Omega^{-1}NL)\subset H^s_{\text{\rm loc/c}}(M;E)\quad(s<-n'/2)
\]
with
\[
C^\infty_{{\cdot}/\text{\rm c}}(L;E\otimes\Omega^{-1}NL)\cap H^{-n'/2}_{\text{\rm loc/c}}(M;E)=0\;.
\]
\end{prop}

\begin{proof}
First, take $M=\R^n$, $L=\R^{n''}\times\{0\}\equiv\R^{n''}$ and $E=M\times\C$ (the trivial line bundle). Let $\delta_0$ be the Dirac mass at $0$ in $\R^{n'}$. For any $\phi\in\SS(\R^{n''})$, consider the tensor product distribution $\phi\otimes\delta_0\in\SS(\R^n)'$ \cite[Section~5.1]{Hormander1983-I}. Its Fourier transform is $\hat\phi\otimes\hat\delta_0=\hat\phi\otimes1$. If $\phi\ne0$, then $\hat\phi\otimes1\in L^2(\R^n,(1+|\xi|^2)^s\,d\xi)$ if and only if $1\in L^2(\R^{n'},(1+|\xi|^2)^s\,d\xi)$, which holds just when $s<-n'/2$. Moreover the map
\[
\SS(\R^{n''})\to L^2(\R^n,(1+|\xi|^2)^s\,d\xi)\;,\quad\phi\mapsto\hat\phi\otimes1\;,
\]
is continuous if $s<-n'/2$. Thus~\eqref{H^s(R^n) cong L^2(R^n (1+|xi|^2)^s d xi)} yields versions of the stated properties using $\SS(\R^{n''})$ and $H^s(R^n)$. 

For arbitrary $M$, $L$ and $E$, the result follows from the previous case by using a locally finite atlas, a subordinated partition of unity, and diffeomorphisms of triviality of $E$.
\end{proof}

For instance, for any $p\in M$ and $u\in E_p\otimes\Omega_p^{-1}M$, we get $\delta_p^u\in H^s_{\text{\rm c}}(M;E)$ if $s<-n/2$, with $\langle\delta_p^u,v\rangle=u\cdot v(p)$ for $v\in C^\infty(M;E^*\otimes\Omega)$, obtaining a continuous map
\[
M\times C^\infty(M;E\otimes\Omega^{-1})\to H^s_{\text{\rm c}}(M;E)\;,\quad(p,u)\mapsto\delta_p^{u(p)}\;.
\]
As a particular case, the Dirac mass at any $p\in\R^n$ is $\delta_p=\delta_p^{1\otimes|dx|^{-1}}\in H^s_{\text{\rm c}}(\R^n)$.

\subsection{Differential operators on conormal distributional sections}\label{ss: diff opers on conormal distribs}

Any $A\in\Diff^k(M;E)$ induces continuous linear maps \cite[Lemma~6.1.1]{Melrose1996}
\begin{equation}\label{A: I^(s)(M L E) -> I^[s-k](M L E)}
A:I^{(s)}_{{\cdot}/\text{\rm c}}(M,L;E)\to I^{(s-k)}_{{\cdot}/\text{\rm c}}(M,L;E)\;,
\end{equation}
which induce a continuous endomorphism $A$ of $I_{{\cdot}/\text{\rm c}}(M,L;E)$. If $A\in\Diff(M,L;E)$, then it clearly induces a continuous endomorphism $A$ of every $I^{(s)}_{{\cdot}/\text{\rm c}}(M,L;E)$.

According to~\eqref{u in C^-infty_c mapsto delta_L^u}, for $A\in\Diff(M,L;E)$ and $u\in C^\infty_{{\cdot}/\text{\rm c}}(L;E\otimes\Omega^{-1}NL)$,
\begin{equation}\label{A delta_L^u}
A\delta_L^u=\delta_L^{A'u}\;,\quad A'=((A^t)|_L)^t\in\Diff(L;E\otimes\Omega^{-1}NL)\;,
\end{equation}
where $A^t\in\Diff(M,L;E^*\otimes\Omega)$ and $(A^t)|_L\in\Diff(L,E^*\otimes\Omega_LM)$ using the vector bundle versions of~\eqref{A in Diff(M L) => A^t in Diff(M L Omega)} and~\eqref{Diff(M L) -> Diff(L)}. In fact, for $v\in C^\infty_{\text{\rm c}/{\cdot}}(M;E^*\otimes\Omega)$,
\[
\langle A\delta_L^u,v\rangle=\langle\delta_L^u,A^tv\rangle=\langle u,(A^tv)|_L\rangle
=\langle u,(A^t)|_L(v|_L)\rangle
=\langle A'u,v|_L\rangle=\langle\delta_L^{A'u},v\rangle\;.
\]
By~\eqref{A delta_L^u}, $\Diff(M,L;E)$ preserves the subspace of Dirac sections given by~\eqref{u in C^infty_c mapsto delta_L^u}. Thus the continuous inclusion of \Cref{p: C^infty_./c(L ...) subset H^s_./c(M E)} induces a continuous inclusion
\begin{equation}\label{u mapsto delta_L^u conormal}
C^\infty_{{\cdot}/\text{\rm c}}(L;E\otimes\Omega^{-1}NL)\subset I^{(s)}_{{\cdot}/\text{\rm c}}(M,L;E)\quad(s<-n'/2)\;.
\end{equation}

\subsection{Pull-back of conormal distributions}\label{ss: pull-back of conormal distribs}

If a smooth map $\phi:M'\to M$ is transverse to a regular submanifold $L\subset M$, then $L':=\phi^{-1}(L)\subset M'$ is a regular submanifold and (the trivial-line-bundle version of)~\eqref{phi^*: C^infty(M E) -> C^infty(M' phi^*E)} has continuous extensions 
\begin{equation}\label{phi^*: I^m(M L) -> I^m+k/4(M' L')}
\phi^*:I^m(M,L)\to I^{m+k/4}(M',L')\quad(m\in\R)\;,
\end{equation}
where $k=\dim M-\dim M'$  \cite[Theorem~5.3.8]{Simanca1990}, \cite[Proposition~6.6.1]{Melrose1996}. Taking inductive limits and using~\eqref{sandwich for I}, we get a continuous linear map
\begin{equation}\label{phi^*: I(M L) -> I(M' L')}
\phi^*:I(M,L)\to I(M',L')\;.
\end{equation}
If $\phi$ is a submersion, then~\eqref{phi^*: I(M L) -> I(M' L')} is a restriction of~\eqref{phi_*: C^-infty(M E) -> C^-infty(M' phi^*E)}. If $\phi$ is a local diffeomorphism, then~\eqref{phi^*: I(M L) -> I(M' L')} is compatible with the Sobolev and symbol order filtrations in the sense that it restricts to continuous maps between the spaces defining those filtrations.

A more general pull-back of distributional sections can be defined under conditions on the wave front set \cite[Theorem~8.2.4]{Hormander1983-I}, but we will not use it.

\subsection{Push-forward of conormal distributional sections}\label{ss: push-forward of conormal distribs}

Now let $\phi:M'\to M$ be a smooth submersion, and let $L\subset M$ and $L'\subset M'$ be regular submanifolds such that $\phi(L')\subset L$ and the restriction $\phi:L'\to L$ is also a smooth submersion. Then~\eqref{phi_*: C^infty_c(M' phi^*E otimes Omega_fiber) -> C^infty_c(M E)} has continuous extensions
\begin{equation}\label{phi_*: I^m_c(M' L' Omega_fiber) -> I^m+l/2-k/4_c(M L)}
\phi_*:I^m_{\text{\rm c}}(M',L';\Omega_{\text{\rm fiber}})\to I^{m+l/2-k/4}_{\text{\rm c}}(M,L)\quad(m\in\R)\;,
\end{equation}
where $k=\dim M'-\dim M$ and $l=\dim L'-\dim L$ \cite[Theorem~5.3.6]{Simanca1990}, \cite[Proposition~6.7.2]{Melrose1996}. Taking inductive limits, we get a continuous linear map
\begin{equation}\label{phi_*: I_c(M' L' Omega_fiber) -> I_c(M L)}
\phi_*:I_{\text{\rm c}}(M',L';\Omega_{\text{\rm fiber}})\to I_{\text{\rm c}}(M,L)\;,
\end{equation}
which is a restriction of~\eqref{phi_*: C^-infty_c(M' phi^*E otimes Omega_fiber) -> C^-infty_c(M E)}. If $\phi$ is a local diffeomorphism, then~\eqref{phi_*: I_c(M' L' Omega_fiber) -> I_c(M L)} is compatible with the Sobolev and symbol order filtrations.

\subsection{Pseudodifferential operators}\label{ss: pseudodiff ops}

This type of operators is the main application of conormal distributions (see e.g.\ \cite{Taylor1981,Hormander1985-III,Melrose2006,Simanca1990}).

\subsubsection{Case of compact manifolds}\label{sss: pseudodiff ops - compact}

Suppose first that $M$ is compact. The filtered algebra and $C^\infty(M^2)$-module of pseudodifferential operators, $\Psi(M)$, consists of the continuous endomorphisms $A$ of $C^\infty(M)$ with Schwartz kernel $K_A\in I(M^2,\Delta)$, where $\Delta$ is the diagonal. In fact, by the Schwartz kernel theorem, we may consider $\Psi(M)\equiv I(M^2,\Delta)$. It is filtered by the symbol order, $\Psi^m(M)\equiv I^m(M^2,\Delta)$ ($m\in\R$), and $\Psi^{-\infty}(M)\equiv I^{-\infty}(M^2,\Delta)$ consists of the smoothing operators. The analogs of~\eqref{H^s(M) = ...} and~\eqref{H^-s(M) = ...} hold true using $\Psi^s(M)$ instead of $\Diff^s(M)$ for any $s\in\R$. In this way, $\Psi(M)$ also becomes a LCHS satisfying the properties indicated in \Cref{sss: conormal - Sobolev order - compact,sss: conormal - symbol order - compact}.

Taking the $C^\infty(M^2)$-tensor product of $\Psi(M)$ with $C^\infty(M;F\boxtimes E^*)$, we get $\Psi(M;E,F)$ (or $\Psi(M;E)$ \index{$\Psi(M;E)$} if $E=F$) as in \Cref{ss: diff ops}, satisfying the analog of~\eqref{Diff^m(M E) equiv Diff^m(M)}. In this case, we have $\bar m=m$ in~\eqref{bar m}, and the \emph{symbol} of any $A\in\Psi^m(M;E,F)$ can be given by $\sigma_m(A)\equiv\sigma_m(K_A)$; this symbol is used to extend the concept of \emph{ellipticity} to pseudodifferential operators (see e.g.\ \cite{Melrose1993}). 

$\Psi(M;E)$ is preserved by taking transposes, and therefore any $A\in\Psi(M;E)$ defines a continuous endomorphism $A$ of $C^{-\infty}(M;E)$ (\Cref{ss: ops}), and $\sing\supp Au\subset\sing\supp u$ for all $u\in C^{-\infty}(M;E)$ (\emph{pseudolocality}).

If $A\in\Psi^m(M;E)$, it defines a bounded operator $A:H^{s+m}(M;E)\to H^s(M;E)$. This can be considered as a closable densely defined operator in $H^s(M;E)$, like in the case of differential operators (\Cref{ss: Sobolev sps}). In the case $s=0$, the adjoint of $A$ is induced by the formal adjoint $A^*\in\Psi^m(M;E)$. The symbol map on $\Psi(M;E)$ is multiplicative and compatible with transposition and taking adjoints.

The class of pseudodifferential operators is preserved by transposition. So any $A\in\Psi^m(M)$ defines a continuous endomorphism $A$ of $C^{-\infty}(M)$ (\Cref{ss: ops}), and $\sing\supp Au\subset\sing\supp u$ for all $u\in C^{-\infty}_{\text{\rm c}}(M)$ (\emph{pseudolocality}).

\subsubsection{Extension to non-compact manifolds}\label{sss: pseudodiff ops - non-compact}

If $M$ is not assumed to be compact, $\Psi(M)$ is similarly defined with the change that any $A\in\Psi^m(M)$ defines continuous linear maps $A:C^{\pm\infty}_{\text{\rm c}}(M)\to C^{\pm\infty}(M)$ and $A:H^{s+m}_{\text{\rm c}}(M)\to H^s_{\text{\rm loc}}(M)$. Thus $\Psi(M)$ is not an algebra in this case. However, if $A\in\Psi^m(M)$ is properly supported (both factor projections $M^2\to M$ restrict to proper maps $\supp K_A\to M$), then it defines a continuous endomorphism $A$ of $C^{-\infty}_{\text{\rm c}}(M;E)$; in this sense, properly supported pseudodifferential operators can be composed. Pseudodifferential operators are properly supported modulo $\Psi^{-\infty}(M)$. Like in the compact case, $\Psi(M)\equiv I(M^2,\Delta)$ becomes a filtered $C^\infty(M^2)$-module and LCHS satisfying the properties indicated in \Cref{sss: conormal - Sobolev order - non-compact,sss: conormal - symbol order - non-compact}.

In the setting of bounded geometry (\Cref{sss: diff ops of bd geom}), properly supported pseudodifferential operators with uniformly bounded symbols, and their uniform ellipticity, were studied in \cite{Kordyukov1991,Kordyukov2000}.

\section{Dual-conormal distributions}\label{s: dual-conormal distribs}

\subsection{Dual-conormal distributions}\label{ss: dual-conormal distribs}

Consider the notation of \Cref{ss: conormal - Sobolev order,ss: conormal - symbol order}. 

\subsubsection{Case of compact manifolds}\label{sss: dual-conormal distribs - compact}

Assume first that $M$ is compact. The space of \emph{dual-conormal distributions} of $M$ at $L$ (or of $(M,L)$) is \cite[Chapter~6]{Melrose1996} \index{$I'(M,L)$}
\begin{equation}\label{I'(M L) = I(M L Omega)'}
I'(M,L)=I(M,L;\Omega)'\;.
\end{equation}

\begin{cor}\label{c: I'(M L) is complete and Montel}
$I'(M,L)$ is a complete Montel space.
\end{cor}

\begin{proof}
Since $I(M,L;\Omega)$ is bornological (the version \Cref{c: I(M L) is barreled} with $\Omega M$), $I'(M,L)$ is complete \cite[IV.6.1]{Schaefer1971}, \cite[Corollary~6.1.18]{PerezCarrerasBonet1987}, \cite[Theorem~13.2.13]{NariciBeckenstein2011}.

Since $I(M,L;\Omega)$ is a Montel space (the version \Cref{c: I(M L) is acyclic and Montel} with $\Omega M$), $I'(M,L)$ is a Montel space \cite[Proposition~3.9.9]{Horvath1966-I}, \cite[6.27.2~(2)]{Kothe1969-I}, \cite[IV.5.9]{Schaefer1971}.
\end{proof}

Let also \index{$I^{\prime\,(s)}(M,L)$} \index{$I^{\prime\,m}(M,L)$}
\begin{equation}\label{I^prime (s)(M L) = I^(-s)(M L Omega)'}
I^{\prime\,(s)}(M,L)=I^{(-s)}(M,L;\Omega)'\;,\quad I^{\prime\,m}(M,L)=I^{-m}(M,L;\Omega)'\;.
\end{equation}

\begin{cor}\label{c: I'^(s)(M L) is bornological}
$I^{\prime\,(s)}(M,L)$ is bornological and barreled.
\end{cor}

\begin{proof}
Since $I^{(-s)}(M,L;\Omega)$ is a reflexive Fr\'echet space (the version of \Cref{p: I^(s)(M L) is a totally reflexive Frechet sp} with $\Omega M$), $I^{\prime\,(s)}(M,L)$ is bornological \cite[Corollary~1 of~IV.6.6]{Schaefer1971}, and therefore barreled \cite[IV.6.6]{Schaefer1971}. 
\end{proof}

Transposing the versions of~\eqref{I^(s)(M L) subset I^(s')(M L)} and~\eqref{I^m(M L) subset I^m'(M L)} with $\Omega M$, we get continuous restrictions, for $s'<s$ and $m<m'$,
\[
I^{\prime\,(s)}(M,L)\to I^{\prime\,(s')}(M,L)\;,\quad I^{\prime\,m}(M,L)\to I^{\prime\,m'}(M,L)\;.
\]
These maps form projective systems, giving rise to $\varprojlim I^{\prime\,(s)}(M,L)$ as $s\uparrow\infty$ and $\varprojlim I^{\prime\,m}(M,L)$ as $m\downarrow-\infty$. Transposing the versions of~\eqref{C^infty(M) subset I^(infty)(M L)} and~\eqref{sandwich for I} with $\Omega M$, we get continuous inclusions
\begin{equation}\label{C^-infty(M) supset I'(M L) supset C^infty(M)}
C^{-\infty}(M)\supset I'(M,L)\supset C^\infty(M)\;,
\end{equation}
and, for all $m\in\R$ and $\epsilon>0$, continuous restrictions
\begin{equation}\label{sandwich for I'}
I^{\prime\,(-m+n/4-\epsilon)}(M,L)\leftarrow I^{\prime\,m}(M,L)
\leftarrow I^{\prime\,(-m+n/4+\epsilon)}(M,L)\;.
\end{equation}
Thus 
\begin{equation}\label{varprojlim I'^(s)(M L) equiv varprojlim I'^m(M L)}
\varprojlim I^{\prime\,(s)}(M,L)\equiv\varprojlim I^{\prime\,m}(M,L)\;.
\end{equation}

\begin{cor}\label{c: I'(M L) equiv varprojlim I'^(s)(M L)}
$I'(M,L)\equiv\varprojlim I^{\prime\,(s)}(M,L)$.
\end{cor}

\begin{proof}
This holds because $I(M,L)$ is regular (\Cref{c: I(M L) is acyclic and Montel}) \cite[Lemma~1]{Kucera2004}.

Alternatively, the following argument can be used. $I'(M,L)$ is a Montel space (\Cref{c: I'(M L) is complete and Montel}); in particular, it is barreled, and therefore a Mackey space \cite[IV.3.4]{Schaefer1971}. On the other hand, every $I^{\prime\,(s)}(M,L)$ is bornological  (\Cref{c: I'^(s)(M L) is bornological}), and therefore a Mackey space \cite[Proposition~3.7.2]{Horvath1966-I}, \cite[IV.3.4]{Schaefer1971}, \cite[Theorem~13.2.10]{NariciBeckenstein2011}. So the result follows applying \cite[Remark of~IV.4.5]{Schaefer1971}.
\end{proof}

\subsubsection{Extension to non-compact manifolds}\label{sss: dual-conormal distribs - non-compact}

If $M$ is not supposed to be compact, we can similarly define the space $I'_K(M,L)$ of dual-conormal distributions supported in any compact $K\subset M$. Then define the LCHSs, $I'_{\text{\rm c}}(M,L)=\bigcup_KI'_K(M,L)$ like in~\eqref{Cinftyc(U)}, and $I'(M,L)$ like in~\eqref{I(M L) - non-compact M} using $I'_{\text{\rm c}}(M,L)$ instead of $I_{\text{\rm c}}(M,L)$. These spaces satisfy a version of~\eqref{I'(M L) = I(M L Omega)'}, interchanging arbitrary/compact support like in~\eqref{C^-infty_cdot/c(M;E)}. Given a smooth partition of unity $\{f_j\}$ so that every $K_j:=\supp f_j$ is compact, the multiplication by the functions $f_j$ defines closed TVS-embeddings
\begin{equation}\label{I'(M L) -> ...}
I'(M,L)\to\prod_jI'_{K_j}(M,L)\;,\quad I'_{\text{\rm c}}(M,L)\to\bigoplus_jI'_{K_j}(M,L)\;.
\end{equation}
By the extension of \Cref{c: I(M L) is acyclic and Montel} for $I_{\text{\rm c}}(M,L;\Omega)$, the obvious extension of \Cref{c: I'^(s)(M L) is bornological} for every $I'_{K_j}(M,L)$, the indicated extension of~\eqref{I'(M L) = I(M L Omega)'} and the properties of~\eqref{I'(M L) -> ...}, we get an extension of \Cref{c: I'(M L) is complete and Montel}. 

Similarly, we can define the spaces $I^{\prime\,(s)}_{{\cdot}/\text{\rm c}}(M,L)$ and $I^{\prime\,m}_{{\cdot}/\text{\rm c}}(M,L)$. \index{$I^{\prime\,m}_{{\cdot}/\text{\rm c}}(M,L)$} They satisfy~\eqref{I^prime (s)(M L) = I^(-s)(M L Omega)'} interchanging the support condition, and also obvious versions of~\eqref{C^-infty(M) supset I'(M L) supset C^infty(M)}--\eqref{varprojlim I'^(s)(M L) equiv varprojlim I'^m(M L)}. Since $I_{\text{\rm c}}(M,L)$ is an acyclic Montel space (\Cref{sss: conormal - symbol order - non-compact}), there is an extension of \Cref{c: I'(M L) equiv varprojlim I'^(s)(M L)} for $I'(M,L)$.

\subsection{Differential operators on dual-conormal distributional sections}
\label{ss: diff opers on dual-conormal distribs}

For any $A\in\Diff(M;E)$, consider $A^t\in\Diff(M;E^*\otimes\Omega)$. The transpose of $A^t$ on $I_{\text{\rm c}/{\cdot}}(M,L;E^*\otimes\Omega)$ is a continuous endomorphism $A$ of $I'_{{\cdot}/\text{\rm c}}(M,L;E)$, which is a restriction of the map $A$ on $C^{-\infty}(M;E)$ (\Cref{ss: diff ops}). By~\eqref{A: I^(s)(M L E) -> I^[s-k](M L E)}, if $A\in\Diff^m(M;E)$, we get induced continuous linear maps
\begin{equation}\label{A: I^prime [s](M L E) -> I^prime (s-m)(M L E)}
A:I^{\prime\,(s)}_{{\cdot}/\text{\rm c}}(M,L;E)\to I^{\prime\,(s-m)}_{{\cdot}/\text{\rm c}}(M,L;E)\;,
\end{equation}
If $A\in\Diff(M,L;E)$, the transpose of $A^t$ of $I^{(-s)}_{\text{\rm c}/{\cdot}}(M,L;E^*\otimes\Omega)$ is a continuous endomorphism $A$ of $I^{\prime\,(s)}_{{\cdot}/\text{\rm c}}(M,L;E)$.

\subsection{Pull-back of dual-conormal distributions}
\label{ss: pull-back of dual-conormal distributions}

If the conditions of \Cref{ss: push-forward of conormal distribs} hold, transposing the versions of~\eqref{phi_*: I^m_c(M' L' Omega_fiber) -> I^m+l/2-k/4_c(M L)} and~\eqref{phi_*: I_c(M' L' Omega_fiber) -> I_c(M L)} with $\Omega M$ and $-m$, we get continuous linear pull-back maps 
\begin{gather}
\phi^*:I^{\prime\,m}(M,L)\to I^{\prime\,m+l/2-k/4}(M',L')\quad(m\in\R)\;,
\label{phi^*: I^prime m(M L) -> I^prime m+l/2-k/4(M' L')}\\
\phi^*:I'(M,L)\to I'(M',L')\;.\label{phi^*: I'(M L) -> I'(M' L')}
\end{gather}
The map~\eqref{phi^*: I'(M L) -> I'(M' L')} is an extension of~\eqref{phi^*: C^infty(M E) -> C^infty(M' phi^*E)}, a restriction of~\eqref{phi_*: C^-infty(M E) -> C^-infty(M' phi^*E)} and the projective limit of the maps~\eqref{phi^*: I^prime m(M L) -> I^prime m+l/2-k/4(M' L')}. If $\phi$ is a local diffeomorphism, then~\eqref{phi^*: I'(M L) -> I'(M' L')} is compatible with the Sobolev and symbol order filtrations.

\subsection{Push-forward of dual-conormal distributions}
\label{ss: push-forward of dual-conormal distributions}

With the notation of \Cref{ss: pull-back of conormal distribs}, if $\phi$ is a submersion, transposing the versions of~\eqref{phi^*: I^m(M L) -> I^m+k/4(M' L')} and~\eqref{phi^*: I(M L) -> I(M' L')} with $\Omega M$ and $-m$, we get continuous linear push-forward maps
\begin{gather}
\phi_*:I^{\prime\,m}_{\text{\rm c}}(M',L'\otimes\Omega_{\text{\rm fiber}})\to I^{\prime\,m-k/4}_{\text{\rm c}}(M,L)\quad(m\in\R)\;, \label{phi_*: I^prime m_c(M' L' Omega_fiber) -> I^prime m-k/4_c(M L)}\\
\phi_*:I'_{\text{\rm c}}(M',L';\Omega_{\text{\rm fiber}})\to I'_{\text{\rm c}}(M,L)\;.
\label{phi_*: I'_c(M' L' Omega_fiber) -> I'_c(M L)}
\end{gather}
The map~\eqref{phi_*: I'_c(M' L' Omega_fiber) -> I'_c(M L)} is an extension of~\eqref{phi_*: C^infty_c(M' phi^*E otimes Omega_fiber) -> C^infty_c(M E)}, a restriction of~\eqref{phi_*: C^-infty_c(M' phi^*E otimes Omega_fiber) -> C^-infty_c(M E)} and the projective limit of the maps~\eqref{phi_*: I^prime m_c(M' L' Omega_fiber) -> I^prime m-k/4_c(M L)}. If $\phi$ is a local diffeomorphism, then~\eqref{phi_*: I'_c(M' L' Omega_fiber) -> I'_c(M L)} is compatible with the Sobolev and symbol order filtrations.

\section{Conormal distributions at the boundary}\label{s: conormal distribs at the boundary}

For the sake of simplicity, in this section and in \Cref{s: conormal seq,s: dual-conormal seq}, we only consider the case of compact manifolds unless otherwise stated. But the concepts, notation and some of the results, can be extended to the non-compact case like in \Cref{sss: conormal - Sobolev order - non-compact,sss: conormal - symbol order - non-compact,sss: dual-conormal distribs - non-compact}, using arbitrary/compact support conditions. Such extensions to non-compact manifolds may be used without further comment.

\subsection{Some notions of b-geometry}\label{ss: b-geometry}

R.~Melrose introduced b-calculus, a way to extend calculus to manifolds with boundary \cite{Melrose1993,Melrose1996}. We will only use a part of it called small b-calculus. Let $M$ be a compact (smooth) $n$-manifold with boundary; its interior is denoted by $\mathring M$. \index{$\mathring M$} There exists a function $x\in C^\infty(M)$ so that $x\ge0$, $\partial M=\{x=0\}$ (i.e., $x^{-1}(0)$) and $dx\ne0$ on $\partial M$, which is called a \emph{boundary defining function}. Let ${}_+N\partial M\subset N\partial M$ be the inward-pointing subbundle of the normal bundle to the boundary. There is a unique trivialization $\nu\in C^\infty(\partial M;{}_+N\partial M)$ of ${}_+N\partial M$ so that $dx(\nu)=1$. Take a collar neighborhood $T\equiv[0,\epsilon_0)_x\times\partial M$ of $\partial M$, whose projection $\varpi:T\to\partial M$ is the second factor projection. (In a product expression, every factor projection may be indicated as subscript of the corresponding factor.) Given coordinates $y=(y^1,\dots,y^{n-1})$ on some open $V\subset\partial M$, we get via $\varpi$ coordinates $(x,y)=(x,y^1,\dots,y^{n-1})$ adapted (to $\partial M$) on the open subset $U\equiv[0,\epsilon_0)\times V\subset M$. There are vector bundles over $M$, $\bT M$ and $\bT^*M$, \index{$\bT M$} \index{$\bT^*M$} called \emph{b-tangent} and \emph{b-cotangent} bundles, which have the same restrictions as $TM$ and $T^*M$ to $\mathring M$, and such that $x\partial_x,\partial_{y^1},\dots,\partial_{y^{n-1}}$ and $x^{-1}dx,dy^1,\dots,dy^{n-1}$ extend to smooth local frames around boundary points. This gives rise to versions of induced vector bundles, like $\bOmega^sM:=\Omega^s(\bT M)$ ($s\in\R$) and $\bOmega M:=\bOmega^1M$. Clearly,
\begin{equation}\label{C^infty(M Omega^s) equiv x^s C^infty(M bOmega^s)}
C^\infty(M;\Omega^s)\equiv x^sC^\infty(M;\bOmega^s)\;.
\end{equation}
Thus the integration operator $\int_M$ is defined on $xC^\infty(M;\bOmega)$, and induces a pairing between $C^\infty(M)$ and $xC^\infty(M;\bOmega)$.

At the points of $\partial M$, the local section $x\partial_x$ is independent of the choice of adapted local coordinates, spanning a trivial line subbundle ${}^{\text{\rm b}}\!N\partial M\subset\bT_{\partial M}M$ with $T\partial M=\bT_{\partial M}M/{}^{\text{\rm b}}\!N\partial M$. So $\bOmega^s_{\partial M}M\equiv\Omega^s\partial M\otimes\Omega^s({}^{\text{\rm b}}\!N\partial M)$, and a restriction map $C^\infty(M;\bOmega^s)\to C^\infty(\partial M;\Omega^s)$ is locally given by
\[
u=a(x,y)\,\Big|\frac{dx}{x}dy\Big|^s\mapsto u|_{\partial M}=a(0,y)\,|dy|^s\;.
\]

A Euclidean structure $g$ on $\bT M$ is called a \emph{b-metric}. Locally,
\[
g=a_0\Big(\frac{dx}{x}\Big)^2+2\sum_{j=1}^{n-1}a_{0j}\,\frac{dx}{x}\,dy^j
+\sum_{j,k=1}^{n-1}a_{jk}\,dy^j\,dy^k\;,
\]
where $a_0$, $a_{0j}$ and $a_{jk}$ are $C^\infty$ functions, on condition that $g$ is positive definite. If moreover $a_0=1+O(x^2)$ and $a_{0j}=O(x)$ as $x\downarrow0$, then $g$ is called \emph{exact}. In this case, the restriction of $g$ to $\mathring T\equiv(0,\epsilon_0)\times\partial M$ is asymptotically cylindrical, and therefore $g|_{\mathring M}$ is complete. This restriction is of bounded geometry if it is cylindrical around the boundary; i.e., $g=(\frac{dx}{x})^2+h$ on $\mathring T$ for (the pull-back via $\varpi$ of) some Riemannian metric $h$ on $\partial M$, taking $\epsilon_0$ small enough; i.e., $a_0=1$ and $a_{0j}=0$ using adapted local coordinates.

\subsection{Supported and extendible functions}\label{ss: supported and extendible smooth funcs}

Let $\breve M$ \index{$\breve M$} be any closed manifold of dimension $n$ which contains $M$ as submanifold (for instance, $\breve M$ could be the double of $M$), and let $M'=\breve M\setminus\mathring M$, which is another compact $n$-submanifold with boundary of $\breve M$, with dimension $n$ and $\partial M'=M\cap M'=\partial M$. 

The concepts, notation and conventions of \Cref{ss: smooth/distributional sections} have straightforward extensions to manifolds with boundary, like the Fr\'echet space $C^\infty(M)$. Its elements are called \emph{extendible functions} because the continuous linear restriction map
\begin{equation}\label{R: C^infty(widetilde M) -> C^infty(M)}
R:C^\infty(\breve M)\to C^\infty(M)
\end{equation}
is surjective; in fact, there is a continuous linear extension map $E:C^\infty(M)\to C^\infty(\breve M)$ \cite{Seeley1964}. Since $C^\infty(\breve M)$ and $C^\infty(M)$ are Fr\'echet spaces, the map~\eqref{R: C^infty(widetilde M) -> C^infty(M)} is open by the open mapping theorem, and therefore it is a surjective topological homomorphism. Its null space is $C^\infty_{M'}(\breve M)$.

The Fr\'echet space of \emph{supported} functions is the closed subspace of the smooth functions on $M$ that vanish to all orders at the points of $\partial M$, \index{$\dot C^\infty(M)$}
\begin{equation}\label{dot C^infty(M) = bigcap_m ge 0 x^m C^infty(M) subset C^infty(M)}
\dot C^\infty(M)=\bigcap_{m\ge0}x^mC^\infty(M)\subset C^\infty(M)\;.
\end{equation}
The extension by zero realizes $\dot C^\infty(M)$ as the closed subspace of functions on $\breve M$ supported in $M$,
\begin{equation}\label{dot C^infty(M) subset C^infty(breve M)}
\dot C^\infty(M)\equiv C^\infty_M(\breve M)\subset C^\infty(\breve M)\;.
\end{equation} 
By~\eqref{dot C^infty(M) = bigcap_m ge 0 x^m C^infty(M) subset C^infty(M)},
\begin{equation}\label{x^m dot C^infty(M) = dot C^infty(M)}
x^m\dot C^\infty(M)=\dot C^\infty(M)\quad(m\in\R)\;,
\end{equation}
and therefore, by~\eqref{C^infty(M Omega^s) equiv x^s C^infty(M bOmega^s)},
\begin{equation}\label{dot C^infty(M bOmega^s) equiv dot C^infty(M Omega^s)}
\dot C^\infty(M;\bOmega^s)\equiv\dot C^\infty(M;\Omega^s)\quad(s\in\R)\;.
\end{equation}

We can similarly define Banach spaces $C^k(M)$ and $\dot C^k(M)$ ($k\in\N_0$) satisfying the analogs of~\eqref{R: C^infty(widetilde M) -> C^infty(M)}--\eqref{dot C^infty(M) subset C^infty(breve M)}, which in turn yield analogs of the first inclusions of~\eqref{C^prime -k'(M E) supset C^prime -k(M E)}, obtaining $C^\infty(M)=\bigcap_kC^k(M)$ and $\dot C^\infty(M)=\bigcap_k\dot C^k(M)$.

\subsection{Supported and extendible distributions}\label{ss: supported and extendible distribs}

The spaces of \emph{supported} and \emph{extendible} distributions on $M$ are \index{$\dot C^{-\infty}(M)$}
\[
\dot C^{-\infty}(M)=C^\infty(M;\Omega)'\;,\quad C^{-\infty}(M)=\dot C^\infty(M;\Omega)'\;.
\]
Transposing the version of~\eqref{R: C^infty(widetilde M) -> C^infty(M)} with $\Omega M$, we get \cite[Proposition~3.2.1]{Melrose1996}
\begin{equation}\label{dot C^-infty(M) subset C^-infty(breve M)}
\dot C^{-\infty}(M)\equiv C^{-\infty}_M(\breve M)\subset C^{-\infty}(\breve M)\;.
\end{equation}
Similarly,~\eqref{dot C^infty(M) subset C^infty(breve M)} and~\eqref{dot C^infty(M) = bigcap_m ge 0 x^m C^infty(M) subset C^infty(M)} give rise to continuous linear restriction maps
\begin{gather}
R:C^{-\infty}(\breve M)\to C^{-\infty}(M)\;,\label{R: C^-infty(breve M) -> C^-infty(M)}\\
R:\dot C^{-\infty}(M)\to C^{-\infty}(M)\;,\label{R: dot C^-infty(M) -> C^-infty(M)}
\end{gather}
which are surjective by the Hahn-Banach theorem. Their null spaces are $C^{-\infty}_{M'}(\breve M)=\dot C^{-\infty}(M')$ and $\dot C^{-\infty}_{\partial M}(M)$ \cite[Proposition~3.3.1]{Melrose1996}, respectively. According to~\eqref{dot C^-infty(M) subset C^-infty(breve M)}, the map~\eqref{R: dot C^-infty(M) -> C^-infty(M)} is a restriction of~\eqref{R: C^-infty(breve M) -> C^-infty(M)}. As a consequence of~\eqref{dot C^-infty(M) subset C^-infty(breve M)}, there are continuous dense inclusions \cite[Lemma~3.2.1]{Melrose1996}
\begin{equation}\label{Cinftyc(mathring M) subset dot C^infty(M) subset C^infty(M) subset dot C^-infty(M)}
\Cinftyc(\mathring M)\subset\dot C^\infty(M)\subset C^\infty(M)\subset\dot C^{-\infty}(M)\;,
\end{equation}
the last one given by the integration pairing between $C^\infty(M)$ and $C^\infty(M;\Omega)$. The restriction of this pairing to $\dot C^\infty(M;\Omega)$ induces a continuous dense inclusion
\begin{equation}\label{C^infty(M) subset C^-infty(M)}
C^\infty(M)\subset C^{-\infty}(M)\;.
\end{equation}
Moreover~\eqref{R: dot C^-infty(M) -> C^-infty(M)} is the identity map on $C^\infty(M)$. 

As before, from~\eqref{x^m dot C^infty(M) = dot C^infty(M)} and~\eqref{dot C^infty(M bOmega^s) equiv dot C^infty(M Omega^s)}, we get
\begin{align}
x^mC^{-\infty}(M)&=C^{-\infty}(M)\quad(m\in\R)\;,\label{x^m C^-infty(M) = C^-infty(M)}\\
C^{-\infty}(M;\bOmega^s)&\equiv C^{-\infty}(M;\Omega^s)\quad(s\in\R)\;.\label{C^-infty(M bOmega^s) equiv C^-infty(M Omega^s)}
\end{align} 

The Banach spaces $C^{\prime-k}(M)$ and $\dot C^{\prime-k}(M)$ ($k\in\N_0$) \index{$\dot C^{\prime-k}(M)$} are similarly defined and satisfy the analogs of~\eqref{dot C^-infty(M) subset C^-infty(breve M)}--\eqref{C^-infty(M bOmega^s) equiv C^-infty(M Omega^s)}. These spaces satisfy the analogs of the second inclusions of~\eqref{C^prime -k'(M E) supset C^prime -k(M E)}, obtaining $\bigcup_kC^{\prime-k}(M)=C^{-\infty}(M)$ and $\bigcup_k\dot C^{\prime-k}(M)=\dot C^{-\infty}(M)$.

\subsection{Supported and extendible Sobolev spaces}\label{ss: supported and extendible Sobolev sps} 

The \emph{supported} Sobolev space of order $s\in\R$ is the closed subspace of the elements supported in $M$, \index{$\dot H^s(M)$}
\begin{equation}\label{dot H^s(M) subset H^s(breve M)}
\dot H^s(M)=H^s_M(\breve M)\subset H^s(\breve M)\;.
\end{equation}
On the other hand, using the map~\eqref{R: dot C^-infty(M) -> C^-infty(M)}, the \emph{extendible} Sobolev space of order $s$ is $H^s(M)=R(H^s(\breve M))$ with the inductive topology given by 
\begin{equation}\label{R: H^s(breve M) -> H^s(M)}
R:H^s(\breve M)\to H^s(M)\;;
\end{equation}
i.e., this is a surjective topological homomorphism. Its null space is $H^s_{M'}(\breve M)$.
The analogs of~\eqref{H^s(M) subset H^s'(M)}--\eqref{C^infty(M) = bigcap_s H^s(M)} hold true in this setting using $\dot C^{\pm\infty}(M)$ and $C^{\pm\infty}(M)$.  Furthermore the analogs of~\eqref{H^s(M) subset H^s'(M)} are also compact operators because~\eqref{dot H^s(M) subset H^s(breve M)} is a closed embedding and~\eqref{R: H^s(breve M) -> H^s(M)} a surjective topological homomorphism.

The following properties are satisfied \cite[Proposition~3.5.1]{Melrose1996}. $C^\infty(M)$ is dense in $H^s(M)$, we have
\begin{equation}\label{dot H^s(M) equiv H^-s(M Omega)'}
\dot H^s(M)\equiv H^{-s}(M;\Omega)'\;,\quad H^s(M)\equiv\dot H^{-s}(M;\Omega)'\;,
\end{equation}
and the map~\eqref{R: dot C^-infty(M) -> C^-infty(M)} has a continuous restriction
\begin{equation}\label{R: dot H^s(M) -> H^s(M)}
R:\dot H^s(M)\to H^s(M)\;,
\end{equation}
which is surjective if $s\le1/2$, and injective if $s\ge-1/2$. In particular, $\dot H^0(M)\equiv H^0(M)\equiv L^2(M)$. The null space of~\eqref{R: dot H^s(M) -> H^s(M)} is $\dot H^s_{\partial M}(M)$.

Since $\dot H^s(M)$ and $H^s(M)$ form compact spectra of Hilbertian spaces, we get the following result.

\begin{prop}\label{p: dot C^-infty(M) and C^-infty(M) are barreled ...}
$\dot C^{-\infty}(M)$ and $C^{-\infty}(M)$ are barreled, ultrabornological, webbed, acyclic DF Montel spaces, and therefore complete, boundedly retractive and reflexive.
\end{prop}

\begin{prop}\label{p: R: dot C^-infty(M) -> C^-infty(M) is a top hom}
The maps~\eqref{R: C^-infty(breve M) -> C^-infty(M)} and~\eqref{R: dot C^-infty(M) -> C^-infty(M)} are surjective topological homomorphisms.
\end{prop}

\begin{proof}
We already know that these maps are linear, continuous and surjective. Since $C^{-\infty}(\breve M)$ is webbed, and $\dot C^{-\infty}(M)$ and $C^{-\infty}(M)$ are webbed and ultrabornological (\Cref{p: dot C^-infty(M) and C^-infty(M) are barreled ...}), the stated maps are also open by the open mapping theorem \cite[7.35.3~(1)]{Kothe1979-II}, \cite[Exercise~14.202~(a)]{NariciBeckenstein2011}, \cite[Section~IV.5]{DeWilde1978}, \cite{Bourles2014}.
\end{proof}

\subsection{The space $\dot C^{-\infty}_{\partial M}(M)$}\label{ss: dot C^-infty_partial M(M)} 

Consider the sequences
\begin{gather}
0\to\dot C^{-\infty}(M') \xrightarrow{\iota} C^{-\infty}(\breve M) \xrightarrow{R} C^{-\infty}(M)\to0\;,
\label{0 -> dot C^-infty(M') -> C^-infty(breve M) -> C^-infty(M) -> 0}\\
0\to\dot C^{-\infty}_{\partial M}(M) \xrightarrow{\iota} \dot C^{-\infty}(M) \xrightarrow{R} C^{-\infty}(M)\to0\;.
\label{0 -> dot C^-infty_partial M(M) -> dot C^-infty(M) -> C^-infty(M) -> 0}
\end{gather}
\Cref{p: R: dot C^-infty(M) -> C^-infty(M) is a top hom} has the following direct consequence.

\begin{cor}\label{c: 0 -> dot C^-infty_partial M(M) -> dot C^-infty(M) -> C^-infty(M) -> 0 is exact}
The sequences~\eqref{0 -> dot C^-infty(M') -> C^-infty(breve M) -> C^-infty(M) -> 0} and~\eqref{0 -> dot C^-infty_partial M(M) -> dot C^-infty(M) -> C^-infty(M) -> 0} are exact sequences in the category of continuous linear maps between LCSs.
\end{cor}

From~\eqref{dot C^-infty(M) subset C^-infty(breve M)}, we get \index{$\dot C^{-\infty}_{\partial M}(M)$}
\begin{equation}\label{dot C^-infty_partial M(M) equiv C^-infty_partial M(breve M)}
\dot C^{-\infty}_{\partial M}(M)\equiv C^{-\infty}_{\partial M}(\breve M)\subset C^{-\infty}(\breve M)\;.
\end{equation}

The analogs of the second inclusion of~\eqref{C^prime -k'(M E) supset C^prime -k(M E)},~\eqref{H^s(M) subset H^s'(M)} and~\eqref{H^-s(M) supset C^prime -k(M) supset H^-k(M)} for the spaces $\dot C^{\prime\,-k}(M)$ and $\dot H^s(M)$ yield corresponding analogs for the spaces $\dot C^{\prime\,-k}_{\partial M}(M)$ and $\dot H^s_{\partial M}(M)$. Thus the spaces $\dot C^{\prime\,-k}_{\partial M}(M)$ \index{$\dot C^{\prime\,-k}_{\partial M}(M)$} and $\dot H^s_{\partial M}(M)$ \index{$\dot H^s_{\partial M}(M)$} form spectra with the same union; the spectrum of spaces $\dot H^s_{\partial M}(M)$ is compact.

\begin{prop}\label{p: dot C^-infty_partial M(M) = bigcup_s H^s_partial M(M)}
$\dot C^{-\infty}_{\partial M}(M)$ is a limit subspace of the LF-space $\dot C^{-\infty}(M)$.
\end{prop}

\begin{proof}
By \Cref{p: R: dot C^-infty(M) -> C^-infty(M) is a top hom,p: dot C^-infty(M) and C^-infty(M) are barreled ...}, $\dot C^{-\infty}(M)/\dot C^{-\infty}_{\partial M}(M)\equiv C^{-\infty}(M)$ is acyclic, which is equivalent to the statement.
\end{proof}

The following analog of \Cref{p: dot C^-infty(M) and C^-infty(M) are barreled ...} hold true with the same arguments, applying \Cref{p: dot C^-infty_partial M(M) = bigcup_s H^s_partial M(M)} and using that the Hilbertian spaces $\dot H^s_{\partial M}(M)$ form a compact spectrum.

\begin{cor}\label{c: dot C^-infty_partial M(M) is barreled ...}
$\dot C^{-\infty}_{\partial M}(M)$ is barreled, ultrabornological, webbed acyclic DF Montel space, and therefore complete, boundedly retractive and reflexive.
\end{cor}

A description of $\dot C^{-\infty}_{\partial M}(M)$ will be indicated in \Cref{r: description of dot C^infty_partial M(M)}.

\subsection{Differential operators acting on $C^{-\infty}(M)$ and $\dot C^{-\infty}(M)$}
\label{ss: diff ops on supp/ext distribs}

The notions of \Cref{ss: diff ops} also have straightforward extensions to manifolds with boundary. The action of any $A\in\Diff(M)$ on $C^\infty(M)$ preserves $\dot C^\infty(M)$. Taking the version of this property with $\Omega M$, we get that $A^t$ acts on $\dot C^\infty(M;\Omega)$ and $C^\infty(M;\Omega)$. Using the transpose again, we get extended continuous actions of $A$ on $C^{-\infty}(M)$ and $\dot C^{-\infty}(M)$. They fit into commutative diagrams
\begin{equation}\label{AR = RA}
\begin{CD}
\dot C^{-\infty}(M) @>A>> \dot C^{-\infty}(M) \\
@VRVV @VVRV \\
C^{-\infty}(M) @>A>> C^{-\infty}(M)
\end{CD}
\qquad
\begin{CD}
C^{-\infty}(M) @>A>> C^{-\infty}(M) \\
@A{\iota}AA @AA{\iota}A \\
C^\infty(M) @>A>> C^\infty(M)\;.\hspace{-.2cm}
\end{CD}
\end{equation}
However the analogous diagram
\begin{equation}\label{A inclusion ne inclusion A}
\begin{CD}
\dot C^{-\infty}(M) @>A>> \dot C^{-\infty}(M) \\
@A{\iota}AA @AA{\iota}A \\
C^\infty(M) @>A>> C^\infty(M)
\end{CD}
\end{equation}
may not be commutative. Let us use the notation $u\mapsto u_{\text{\rm c}}$ for the injection $C^\infty(M)\subset\dot C^{-\infty}(M)$ (see~\eqref{Cinftyc(mathring M) subset dot C^infty(M) subset C^infty(M) subset dot C^-infty(M)}). (Following Melrose, the subscript ``c'' stands for ``cutoff at the boundary.'') We have $A(u_{\text{\rm c}})-(Au)_{\text{\rm c}}\in C^{-\infty}_{\partial M}(M)$ for all $u\in C^\infty(M)$ \cite[Eq.~(3.4.8)]{Melrose1996}. For instance, if $M=[x_0,x_1]$, where $x_0<x_1$ in $\R$, and $A=\partial_x$, integration by parts gives
\[
\partial_x(u_{\text{\rm c}})-(\partial_xu)_{\text{\rm c}}=u(x_1)\,\delta_{x_1}-u(x_0)\,\delta_{x_0}
\]
for all $u\in C^\infty([x_0,x_1])$, using the Dirac mass at $x_j$ ($j=0,1$).

Using~\eqref{R: C^infty(widetilde M) -> C^infty(M)} and its version for vector fields, we get a surjective restriction map
\begin{equation}\label{restriction map Diff(breve M) -> Diff(M)}
\Diff(\breve M)\to\Diff(M)\;,\quad\breve A\mapsto\breve A|_M\;.
\end{equation}
For any $\breve A\in\Diff(\breve M)$ with $\breve A|_M=A$, we have the commutative diagrams
\begin{equation}\label{AR = RA with breve M}
\begin{CD}
C^{-\infty}(\breve M) @>{\breve A}>> C^{-\infty}(\breve M) \\
@VRVV @VVRV \\
C^{-\infty}(M) @>A>> C^{-\infty}(M)\;,\hspace{-.2cm}
\end{CD}
\qquad
\begin{CD}
C^{-\infty}(\breve M) @>{\breve A}>> C^{-\infty}(\breve M) \\
@A{\iota}AA @AA{\iota}A \\
\dot C^{-\infty}(M) @>A>> \dot C^{-\infty}(M)\;,\hspace{-.2cm}
\end{CD}
\end{equation}
where the left-hand side square extends the left-hand side square of~\eqref{AR = RA}.

If $A\in\Diff^m(M)$ ($m\in\N_0$), its actions on $\dot C^{-\infty}(M)$ and $C^{-\infty}(M)$ define continuous linear maps,
\begin{equation}\label{A:H^s(M) -> H^s-m(M)}
A:\dot H^s(M)\to\dot H^{s-m}(M)\;,\quad A:H^s(M)\to H^{s-m}(M)\;.
\end{equation}
The maps~\eqref{R: dot H^s(M) -> H^s(M)} and~\eqref{A:H^s(M) -> H^s-m(M)} fit into a commutative diagram given by the left-hand side square of~\eqref{AR = RA}.

\subsection{Differential operators tangent to the boundary}\label{ss: Diffb(M)}

The concepts of \Cref{s: conormal distribs} can be generalized to the case with boundary when $L=\partial M$ \cite[Chapter~6]{Melrose1996} (see also \cite[Section~4.9]{Melrose1993}), giving rise to the Lie subalgebra and $C^\infty(M)$-submodule $\fXb(M)\subset\fX(M)$ \index{$\fXb(M)$} of vector fields tangent to $\partial M$, called \emph{b-vector fields}. There is a canonical identity $\fXb(M)\equiv C^\infty(M;\bT M)$. Using $\fXb(M)$  like in \Cref{ss: diff ops}, we get the filtered $C^\infty(M)$-submodule and filtered subalgebra $\Diffb(M)\subset\Diff(M)$ \index{$\Diffb(M)$} of \emph{b-differential operators}. It consists of the operators $A\in\Diff(M)$ such that~\eqref{A inclusion ne inclusion A} is commutative \cite[Exercise~3.4.20]{Melrose1996}. The extension of $\Diffb(M)$ to arbitrary vector bundles is closed by taking transposes and formal adjoints. 

Clearly, the restriction map~\eqref{restriction map Diff(breve M) -> Diff(M)} satisfies
\begin{equation}\label{Diff(breve M,partial M) = Diffb(M)}
\Diff(\breve M,\partial M)|_M=\Diffb(M)\;.
\end{equation}
For all $a\in\R$ and $k\in\Z$, we have \cite[Eqs.~(4.2.7) and~(4.2.8)]{Melrose1996}
\begin{equation}\label{Diffb^k(M) x^a = x^a Diffb^k(M)}
\Diffb^k(M)\,x^a=x^a\Diffb^k(M)\;.
\end{equation}
Since $\Diff(M)$ is spanned by $\partial_x$ and $\Diffb(M)$ as algebra, it follows that
\begin{equation}\label{Diff^k(M) x^a subset x^a-k Diff^k(M)}
\Diff^k(M)\,x^a\subset x^{a-k}\Diff^k(M)\;.
\end{equation}

\subsection{Conormal distributions at the boundary}
\label{ss: conormality at the boundary - Sobolev order}

The spaces of \emph{supported} and \emph{extendible} conormal distributions at the boundary of Sobolev order $s\in\R$ are the $C^\infty(M)$-modules and LCSs \index{$\dot\AA^{(s)}(M)$} \index{$\AA^{(s)}(M)$}
\begin{align*}
\dot\AA^{(s)}(M)&=\{\,u\in\dot C^{-\infty}(M)\mid\Diffb(M)\,u\subset\dot H^s(M)\,\}\;,\\
\AA^{(s)}(M)&=\{\,u\in C^{-\infty}(M)\mid\Diffb(M)\,u\subset H^s(M)\,\}\;,
\end{align*}
with the projective topologies given by the maps $P:\dot\AA^{(s)}(M)\to H^s(M)$ and $P:\AA^{(s)}(M)\to\dot H^s(M)$ ($P\in\Diffb(M)$). They satisfy the analogs of the continuous inclusions~\eqref{I^(s)(M L) subset I^(s')(M L)}, giving rise to the filtered $C^\infty(M)$-modules and LF-spaces of \emph{supported} and \emph{extendible} conormal distributions at the boundary, \index{$\dot\AA(M)$} \index{$\AA(M)$}
\begin{equation}\label{dot AA(M) - AA(M)}
\dot\AA(M)=\bigcup_s\dot\AA^{(s)}(M)\;,\quad\AA(M)=\bigcup_s\AA^{(s)}(M)\;.
\end{equation}
By definition, there are continuous inclusions
\begin{equation}\label{dot AA(M) subset dot C^-infty(M)}
\dot\AA(M)\subset\dot C^{-\infty}(M)\;,\quad\AA(M)\subset C^{-\infty}(M)\;.
\end{equation}
Thus $\dot\AA(M)$ and $\AA(M)$ are Hausdorff.

The following analogs of \Cref{p: I^(s)(M L) is a totally reflexive Frechet sp,p: S^infty(U x R^l) is barreled} hold true with formally the same proofs. 

\begin{prop}\label{p: dot AA^(s)(M) and dot AA^(s)(M) are totally reflexive Frechet sps}
$\dot\AA^{(s)}(M)$ and $\AA^{(s)}(M)$ are totally reflexive Fr\'echet spaces.
\end{prop}

\begin{cor}\label{c: dot AA(M) and AA(M) are barreled}
$\dot\AA(M)$ and $\AA(M)$ are barreled, ultrabornological and webbed.
\end{cor}

We have
\begin{equation}\label{bigcap_s AA^(s)(M) = C^infty(M)}
\bigcap_s\dot\AA^{(s)}(M)=\dot C^\infty(M)\;,\quad\bigcap_s\AA^{(s)}(M)=C^\infty(M)\;,
\end{equation}
obtaining continuous dense inclusions \cite[Proposition~4.1.1 and Lemma~4.6.1]{Melrose1996}
\begin{equation}\label{C^infty(M) subset AA(M)}
\dot C^\infty(M)\subset\dot\AA(M)\;,\quad C^\infty(M)\subset\AA(M)\;.
\end{equation}
By elliptic regularity, we also get continuous inclusions \cite[Eq.~(4.1.4)]{Melrose1996}
\begin{equation}\label{dot AA(M)|_mathring M AA(M) subset C^infty(mathring M)}
\dot\AA(M)|_{\mathring M},\AA(M)\subset C^\infty(\mathring M)\;.
\end{equation}
Using~\eqref{R: dot C^-infty(M) -> C^-infty(M)},~\eqref{R: H^s(breve M) -> H^s(M)} and the commutativity of the left-hand side square of~\eqref{AR = RA}, we get the continuous linear restriction maps
\begin{equation}\label{R: dot AA^(s)(M) -> AA^(s)(M)}
R:\dot\AA^{(s)}(M)\to\AA^{(s)}(M)\;,
\end{equation}
which are surjective for $s\le1/2$ and injective for $s\ge-1/2$ because this is true for the maps~\eqref{R: dot H^s(M) -> H^s(M)}.
From~\eqref{dot C^infty(M) = bigcap_m ge 0 x^m C^infty(M) subset C^infty(M)},~\eqref{C^infty(M) subset AA(M)} and~\eqref{R: dot AA^(s)(M) -> AA^(s)(M)} for $s=0$, we get a continuous inclusion $C^\infty(M)\subset\dot\AA^{(0)}(M)$, yielding a continuous inclusion  \cite[Proposition~4.1.1]{Melrose1996}
\begin{equation}\label{C^infty(M) subset dot AA(M)}
C^\infty(M)\subset\dot\AA(M)\;.
\end{equation}
The maps~\eqref{R: dot AA^(s)(M) -> AA^(s)(M)} induce a surjective continuous linear restriction map
\begin{equation}\label{R: dot AA(M) -> AA(M)}
R:\dot\AA(M)\to\AA(M)\;,
\end{equation}
with $R=1$ on $C^\infty(M)$ \cite[Proposition~4.1.1]{Melrose1996}. The surjectivity of~\eqref{R: dot AA(M) -> AA(M)} is a consequence of the existence of enough partial extension maps \cite[Section~4.4]{Melrose1996}; the precise statement is recalled in \Cref{p: E_m} for later use. The maps~\eqref{R: dot AA^(s)(M) -> AA^(s)(M)} and~\eqref{R: dot AA(M) -> AA(M)} are restrictions of~\eqref{R: dot C^-infty(M) -> C^-infty(M)}.

The following analog of \Cref{p: R: dot C^-infty(M) -> C^-infty(M) is a top hom} holds true with formally the same proof, using that $\dot\AA(M)$ is webbed and $\AA(M)$ ultrabornological (\Cref{c: dot AA(M) and AA(M) are barreled}).

\begin{prop}\label{p: R: dot AA(M) -> AA(M) is a surj top hom}
The map~\eqref{R: dot AA(M) -> AA(M)} is a surjective topological homomorphism.
\end{prop}

\begin{cor}\label{c: C^infty(M) subset dot AA(M) is dense}
The inclusion~\eqref{C^infty(M) subset dot AA(M)} is dense.
\end{cor}

\subsection{The spaces $x^mL^\infty(M)$}\label{ss: x^m L^infty(M)}

For $m\in\R$, consider the weighted space $x^mL^\infty(M)$ (\Cref{ss: weighted sps}). From~\eqref{x^m C^-infty(M) = C^-infty(M)} and since $L^\infty(M)\subset C^{-\infty}(M)$, it follows that there is a continuous inclusion
\[
x^mL^\infty(M)\subset C^{-\infty}(M)\;. 
\]
Moreover, for $m'<m$, from $x^{m-m'}\in L^\infty(M)$, we easily get a continuous inclusion
\begin{equation}\label{x^m L^infty(M) subset x^m' L^infty(M)}
x^mL^\infty(M)\subset x^{m'}L^\infty(M)\;.
\end{equation}

\begin{prop}\label{p: C^infty_c(M) is dense in x^m L^infty(M)}
For $m'<m$, $\Cinftyc(\mathring M)$ is dense in $x^mL^\infty(M)$ with the topology of $x^{m'}L^\infty(M)$.
\end{prop}

\begin{proof}
Given $u\in x^mL^\infty(M)$ and $\epsilon>0$, let $B$ be the ball in $x^{m'}L^\infty(M)$ of center $u$ and radius $\epsilon$. Let $S=\sup_Mx^{m-m'}>0$. Since $C^\infty(M)$ is dense in $L^\infty(M)$, there is some $f\in C^\infty(M)$ so that $|f-x^{-m}u|<\min\{\epsilon/2,\epsilon/S\}$ (Lebesgue-) almost everywhere. (Recall that the sets of Lebesgue measure zero are well-defined in any $C^1$ manifold \cite[Lemma~3.1.1]{Hirsch1976}.) There is some $0<\delta<1$ so that $\delta x^{-m}|u|<\epsilon/4$ almost everywhere. Take some $\lambda\in C^\infty(\R)$ so that $\lambda\ge0$, $\lambda(r)\le r^m$ if $r>0$, $\lambda(r)=0$ if $r^{m-m'}\le\delta/2$, and $\lambda(r)=r^m$ if $r^{m-m'}\ge\delta$. Let $h=\lambda(x)f\in\Cinftyc(\mathring M)$. If $x^{m-m'}\le\delta$, then, almost everywhere,
\begin{align*}
x^{-m'}|h-u|&\le\delta x^{-m}(|h|+|u|)\le\delta(|f|+x^{-m}|u|)\\
&\le\delta(|f-x^{-m}u|+2x^{-m}|u|)<\delta\Big(\frac\epsilon2+\frac\epsilon{2\delta}\Big)<\epsilon\;.
\end{align*}
If $x^{m-m'}\ge\delta$, then, almost everywhere,
\[
x^{-m'}|h-u|=x^{m-m'}|x^{-m}\lambda(x)f-x^{-m}u|\le S|f-x^{-m}u|<\epsilon\;.
\]
Thus $h\in B\cap\Cinftyc(\mathring M)$.
\end{proof}

\subsection{Filtration of $\AA(M)$ by bounds}
\label{ss: description of AA(M) by bounds}

For every $m\in\R$, let \index{$\AA^m(M)$}
\[
\AA^m(M)=\{\,u\in C^{-\infty}(M)\mid\Diffb(M)\,u\subset x^mL^\infty(M)\,\}\;.
\]
This is another $C^\infty(M)$-module and LCS, with the projective topology given by the maps $P:\AA^m(M)\to x^mL^\infty(M)$ ($P\in\Diffb(M)$).

\begin{ex}[{\cite[Exercises~4.2.23 and~4.2.24]{Melrose1996}}]
Via the injection of  $\R^l$ into its stereographic compactification $\S^l_+=\{\,x\in\S^l\mid x^{l+1}\ge0\,\}$, the space $\AA^{-m}(\S^l_+)$ corresponds to the symbol space $S^m(\R^l)$ (\Cref{s: symbols}).
\end{ex} 

Note that~\eqref{x^m L^infty(M) subset x^m' L^infty(M)} yields a continuous inclusion
\begin{equation}\label{AA^m(M) subset AA^m'(M)}
\AA^m(M)\subset\AA^{m'}(M)\quad(m'<m)\;.
\end{equation}
Moreover there are continuous inclusions \cite[Proof of Proposition~4.2.1]{Melrose1996}
\begin{equation}\label{sandwich for AA}
\AA^{(s)}(M)\subset\AA^m(M)\subset\AA^{(\min\{m,0\})}(M)\quad(m<s-n/2-1)\;.
\end{equation}
Hence
\begin{equation}\label{AA(M) = bigcup_m AA^m(M)}
\AA(M)=\bigcup_m\AA^m(M)\;.
\end{equation}
Despite of defining the same LF-space, the filtrations of $\AA(M)$ defined by the spaces $\AA^{(s)}(M)$ and $\AA^m(M)$ are not equivalent because, in contrast with~\eqref{bigcap_s AA^(s)(M) = C^infty(M)},
\[
\dot C^\infty(M)=\bigcap_m\AA^m(M)\;.
\]

Let $\{\,P_j\mid j\in\N_0\,\}$ be a countable $C^\infty(M)$-spanning set of $\Diffb(M)$. The topology of $\AA^m(M)$ can be described by the semi-norms $\|{\cdot}\|_{k,m}$ ($k\in\N_0$) given by
\begin{equation}\label{| u |_k m}
\|u\|_{k,m}=\|P_ku\|_{x^mL^\infty}=\underset{M}{\esssup}\,\big|x^{-m}P_ku\big|=\sup_{\mathring M}\big|x^{-m}P_ku\big|\;,
\end{equation}
using~\eqref{dot AA(M)|_mathring M AA(M) subset C^infty(mathring M)} in the last expression. From~\eqref{| u |_K C^k} and~\eqref{dot AA(M)|_mathring M AA(M) subset C^infty(mathring M)}, we also get the continuous semi-norms $\|{\cdot}\|_{K,k,m}$ (for any compact $K\subset\mathring M$ and $k\in\N_0$) on $\AA^m(M)$ given by
\begin{equation}\label{| u |_K k m}
\|u\|_{K,k,m}=\sup_K|P_ku|\;.
\end{equation}
Other continuous semi-norms $\|{\cdot}\|'_{k,m}$ ($k\in\N_0$) on $\AA^m(M)$ are defined by
\begin{equation}\label{| u |'_k m}
\|u\|'_{k,m}=\lim_{\epsilon\downarrow0}\sup_{\{0<x<\epsilon\}}\big|x^{-m}P_ku\big|\;.
\end{equation}
The proofs of the following results are similar to the proofs of \Cref{p: topology of S^m(U x R^l),p: | . |_K I J m' = 0 on S^m(U x R^l),c: coincidence of tops on S^m(U x R^l),c: Cinftyc(U x R^l) is dense in S^infty(U x R^l),c: S^infty(U x R^l) is acyclic and Montel}, using~\eqref{dot AA(M)|_mathring M AA(M) subset C^infty(mathring M)}.

\begin{prop}\label{p: topology of AA^m(M)}
The semi-norms~\eqref{| u |_K k m} and~\eqref{| u |'_k m} together describe the topology of $\AA^m(M)$.
\end{prop}

\begin{prop}\label{p: | . |_k m' = 0 on AA^m(M)}
For $m,m',k\in\N_0$, if $m'<m$, then $\|{\cdot}\|'_{k,m'}=0$ on $\AA^m(M)$.
\end{prop}

\begin{cor}\label{c: coincidence of tops on AA^m(M)}
If $m'<m$, then the topologies of $\AA^{m'}(M)$ and $C^\infty(\mathring M)$ coincide on $\AA^m(M)$. Therefore the topologies of $\AA(M)$ and $C^\infty(\mathring M)$ coincide on $\AA^m(M)$.
\end{cor}

\begin{cor}\label{c: Cinftyc(mathring M) is dense in AA(M)}
For $m'<m$, $\Cinftyc(\mathring M)$ is dense in $\AA^m(M)$ with the topology of $\AA^{m'}(M)$. Therefore $\Cinftyc(\mathring M)$ is dense in $\AA(M)$.
\end{cor}

\begin{cor}\label{c: AA(M) is acyclic and Montel}
$\AA(M)$ is an acyclic Montel space, and therefore complete, boundedly retractive and reflexive.
\end{cor}

\begin{rem}\label{r: Cinftyc(mathring M) is dense in AA(M)}
\Cref{p: C^infty_c(M) is dense in x^m L^infty(M)} provides an alternative direct proof of \Cref{c: Cinftyc(mathring M) is dense in AA(M)}. Actually, it will be shown that $\Cinftyc(\mathring M)$ is dense in every $\AA^m(M)$ with its own topology (\Cref{c: Cinftyc(mathring M) is dense in AA^m(M),r: Cinftyc(mathring M) is dense in AA^m(M)}), reconfirming~\Cref{c: C^infty(M) subset dot AA(M) is dense}.
\end{rem}

The obvious analog of \Cref{r: S^infty(U x R^l) subset C^infty(U x R^l) is not a TVS-embedding} makes sense for~\eqref{dot AA(M)|_mathring M AA(M) subset C^infty(mathring M)} and \Cref{c: coincidence of tops on AA^m(M)}.

\subsection{$\dot\AA(M)$ and $\AA(M)$ vs $I(\breve M,\partial M)$}
\label{ss: dot AA(M) and AA(M) vs I(breve M partial M)}

Using~\eqref{R: C^-infty(breve M) -> C^-infty(M)},~\eqref{R: H^s(breve M) -> H^s(M)},~\eqref{Diff(breve M,partial M) = Diffb(M)} and the commutativity of the left-hand side square of~\eqref{AR = RA with breve M}, we get continuous linear restriction maps
\[
R:I^{(s)}(\breve M,\partial M)\to\AA^{(s)}(M)\;,
\]
which induce a continuous linear restriction map
\begin{equation}\label{R: I(breve M partial M) -> AA(M)}
R:I(\breve M,\partial M)\to\AA(M)\;.
\end{equation}

By~\eqref{dot H^s(M) subset H^s(breve M)},~\eqref{Diff(breve M,partial M) = Diffb(M)} and the commutativity of the right-hand side square of~\eqref{AR = RA with breve M}, we get the TVS-identities
\begin{equation}\label{dot AA^(s)(M) equiv I^(s)_M(breve M partial M)}
\dot\AA^{(s)}(M)\equiv I^{(s)}_M(\breve M,\partial M)\;,
\end{equation}
inducing a continuous linear isomorphism
\begin{equation}\label{dot AA(M) cong I_M(breve M partial M)}
\dot\AA(M) \xrightarrow{\cong} I_M(\breve M,\partial M)\;.
\end{equation}

By~\eqref{dot AA(M) cong I_M(breve M partial M)} and \Cref{p: R: dot AA(M) -> AA(M) is a surj top hom}, the map~\eqref{R: I(breve M partial M) -> AA(M)} is also surjective. Then the following analog of \Cref{p: R: dot C^-infty(M) -> C^-infty(M) is a top hom} follows with formally the same proof, using that $I(\breve M,\partial M)$ is webbed (\Cref{c: I(M L) is barreled}) and $\AA(M)$ ultrabornological (\Cref{c: dot AA(M) and AA(M) are barreled}).

\begin{prop}\label{p: R: I(breve M partial M) -> AA(M) is a surj top hom}
The map~\eqref{R: I(breve M partial M) -> AA(M)} is a surjective topological homomorphism.
\end{prop}

The null space of~\eqref{R: I(breve M partial M) -> AA(M)} is $I_{M'}(\breve M,\partial M)$. The following analog of \Cref{p: dot C^-infty_partial M(M) = bigcup_s H^s_partial M(M)} follows with formally the same proof, using \Cref{p: R: I(breve M partial M) -> AA(M) is a surj top hom,c: AA(M) is acyclic and Montel}.

\begin{prop}\label{p: I_M(breve M partial M) = bigcup_s I^(s)_M(breve M partial M)}
$I_M(\breve M,\partial M)$ is a limit subspace of the LF-space $I(\breve M,\partial M)$.
\end{prop}

\begin{cor}\label{c: dot AA(M) cong I_M(breve M partial M)}
The map~\eqref{dot AA(M) cong I_M(breve M partial M)} is a TVS-isomorphism.
\end{cor}

\begin{proof}
Apply~\eqref{dot AA(M) - AA(M)},~\eqref{dot AA^(s)(M) equiv I^(s)_M(breve M partial M)} and \Cref{p: I_M(breve M partial M) = bigcup_s I^(s)_M(breve M partial M)}.
\end{proof}

\subsection{Filtration of $\dot\AA(M)$ by the symbol order}
\label{ss: conormality at the boundary - symbol order}

Inspired by~\eqref{dot AA^(s)(M) equiv I^(s)_M(breve M partial M)}, let \index{$\dot\AA^m(M)$}
\begin{equation}\label{dot AA^m(M) = I^m_M(breve M partial M)}
\dot\AA^m(M)=I^m_M(\breve M,\partial M)\subset I^m(\breve M,\partial M)\quad(m\in\R)\;,
\end{equation}
which are closed subspaces satisfying the analogs of~\eqref{I^m(M L) subset I^m'(M L)} and~\eqref{sandwich for I}. Thus
\[
\dot\AA(M)=\bigcup_m\dot\AA^m(M)\;,\quad\dot C^\infty(M)=\bigcap_m\dot\AA^m(M)\;,
\]
and the TVS-isomorphism~\eqref{dot AA(M) cong I_M(breve M partial M)} is also compatible with the symbol filtration.

The following is a consequence of \Cref{c: coincidence of tops on I^m(M L)} applied to $(\breve M,\partial M)$.

\begin{cor}\label{c: coincidence of tops on dot AA^m(M)}
For $m<m',m''$, the topologies of $\dot\AA^{m'}(M)$ and $\dot\AA^{m''}(M)$ coincide on $\dot\AA^m(M)$.
\end{cor}

The following result follows like \Cref{c: S^infty(U x R^l) is acyclic and Montel}, applying \Cref{c: coincidence of tops on dot AA^m(M)} and using that $\dot\AA(M)$ is barreled (\Cref{c: dot AA(M) and AA(M) are barreled}) and a closed subspace of the Montel space $I(\breve M,\partial M)$ (\Cref{c: I(M L) is acyclic and Montel}).

\begin{cor}\label{c: dot AA(M) is acyclic and Montel}
$\dot\AA(M)$ is an acyclic Montel space, and therefore complete, boundedly retractive and reflexive.
\end{cor}

\subsection{The space $\KK(M)$}\label{ss: KK(M)}

Using the condition of being supported in $\partial M$, define the LCHSs and $C^\infty(M)$-modules \index{$\KK^{(s)}(M)$} \index{$\KK^m(M)$} \index{$\KK(M)$}
\[
\KK^{(s)}(M)=\dot\AA^{(s)}_{\partial M}(M)\;,\quad\KK^m(M)=\dot\AA^m_{\partial M}(M)\;,\quad
\KK(M)=\dot\AA_{\partial M}(M)\;.
\]
These are closed subspaces of $\dot\AA^{(s)}(M)$, $\dot\AA^m(M)$ and $\dot\AA(M)$, respectively; more precisely, they are the null spaces of the corresponding restrictions of the map~\eqref{R: dot AA(M) -> AA(M)}. They satisfy the analogs of~\eqref{I^(s)(M L) subset I^(s')(M L)},~\eqref{I^m(M L) subset I^m'(M L)} and~\eqref{sandwich for I}. So
\[
\bigcup_s\KK^{(s)}(M)=\bigcup_m\KK^m(M)\;.
\]

The term \emph{conormal sequence at the boundary} of $M$ will be used for
\begin{equation}\label{0 -> KK(M) -> dot AA(M) -> AA(M) -> 0}
0\to\KK(M) \xrightarrow{\iota} \dot\AA(M) \xrightarrow{R} \AA(M)\to0\;.
\end{equation}
\Cref{p: R: dot AA(M) -> AA(M) is a surj top hom} has the following direct consequence.

\begin{cor}\label{c: 0 -> KK(M) -> dot AA(M) -> AA(M) -> 0 is exact}
The conormal sequence at the boundary of $M$ is exact in the category of continuous linear maps between LCSs.
\end{cor}

The following analog of \Cref{p: dot C^-infty_partial M(M) = bigcup_s H^s_partial M(M)} holds true with formally the same proof, using \Cref{p: R: dot AA(M) -> AA(M) is a surj top hom,c: AA(M) is acyclic and Montel}.

\begin{prop}\label{p: KK(M) = bigcup_s KK^(s)(M)}
$\KK(M)$ is a limit subspace of the LF-space $\dot\AA(M)$.
\end{prop}

From the definition of $\dot\AA^{(s)}(M)$ (\Cref{ss: conormality at the boundary - Sobolev order}), we get
\[
\KK^{(s)}(M)=\{\,u\in\dot C^{-\infty}_{\partial M}(M)\mid\Diffb(M)\,u\subset\dot H^s_{\partial M}(M)\,\}\;,
\]
with the projective topology given by the maps $P:\KK^{(s)}(M)\to\dot H^s_{\partial M}(M)$ ($P\in\Diffb(M)$). Hence the following analogs of \Cref{p: I^(s)(M L) is a totally reflexive Frechet sp,p: S^infty(U x R^l) is barreled} hold true with formally the same proofs.

\begin{prop}\label{p: KK^(s)(M) is a totally reflexive Frechet sp}
$\KK^{(s)}(M)$ is a totally reflexive Fr\'echet space.
\end{prop}

\begin{cor}\label{c: KK(M) is barreled}
$\KK(M)$ is barreled, ultrabornological and webbed.
\end{cor}

Now the following analogs of \Cref{c: coincidence of tops on dot AA^m(M),c: dot AA(M) is acyclic and Montel} hold true with formally the same proofs, using \Cref{c: coincidence of tops on dot AA^m(M),c: dot AA(M) is acyclic and Montel,c: KK(M) is barreled}.

\begin{cor}\label{c: coincidence of tops on KK^m(M)}
For $m<m',m''$, the topologies of $\KK^{m'}(M)$ and $\KK^{m''}(M)$ coincide on $\KK^m(M)$.
\end{cor}

\begin{cor}\label{c: KK(M) is acyclic and Montel}
$\KK(M)$ is an acyclic Montel space, and therefore complete, boundedly retractive and reflexive.
\end{cor}

By \Cref{c: dot AA(M) cong I_M(breve M partial M)},
\begin{equation}\label{KK(M) cong I_partial M(breve M partial M)}
\KK(M)\equiv I_{\partial M}(\breve M,\partial M)\;,
\end{equation}
which restricts to identities between the spaces defining the Sobolev and symbol order filtrations, according to~\eqref{dot AA^(s)(M) equiv I^(s)_M(breve M partial M)} and~\eqref{dot AA^m(M) = I^m_M(breve M partial M)}.

A description of $\KK^{(s)}(M)$ and $\KK(M)$ will be indicated in \Cref{r: description of KK(M)}.

\subsection{Action of $\Diff(M)$ on $\dot\AA(M)$, $\AA(M)$ and $\KK(M)$}
\label{ss: diff ops on dot AA(M) and AA(M)}

According to \Cref{ss: diff opers on conormal distribs}, and using~\eqref{Diff(breve M,partial M) = Diffb(M)},~\eqref{dot AA^(s)(M) equiv I^(s)_M(breve M partial M)}, \Cref{p: R: I(breve M partial M) -> AA(M) is a surj top hom} and locality, any $A\in\Diff(M)$ defines continuous endomorphisms $A$ of $\dot\AA(M)$, $\AA(M)$ and $\KK(M)$. If $A\in\Diff^k(M)$, these maps also satisfy the analogs of~\eqref{A: I^(s)(M L E) -> I^[s-k](M L E)}. If $A\in\Diffb(M)$, it clearly defines continuous endomorphisms of $\dot\AA^{(s)}(M)$, $\AA^{(s)}(M)$, $\AA^m(M)$ and $\KK^{(s)}(M)$.

According to \Cref{ss: diff ops on supp/ext distribs},~\eqref{dot AA(M) subset dot C^-infty(M)},~\eqref{C^infty(M) subset AA(M)} and~\eqref{dot AA(M)|_mathring M AA(M) subset C^infty(mathring M)}, the maps of this subsection are restrictions of the endomorphisms $A$ of $\dot C^{-\infty}(M)$, $C^{-\infty}(M)$ and $C^\infty(\mathring M)$, and extensions of the endomorphisms $A$ of $\dot C^\infty(M)$ and $C^\infty(M)$.

\subsection{Partial extension maps}\label{ss: partial extension maps}

Given linear subspaces, $X\subset\AA(M)$ and $Y\subset\dot\AA(M$), a map $E:X\to Y$ is called a \emph{partial extension map} if $R(Y)\subset X$ and $RE=1$ on $X$.

\begin{prop}[Cf.\ {\cite[Section~4.4]{Melrose1996}}]\label{p: E_m}
For all $m\in\R$, there is a continuous linear partial extension map $E_m:\AA^m(M)\to\dot\AA^{(s)}(M)$, \index{$E_m$} where $s=0$ if $m\ge0$, and $m>s\in\Z^-$ if $m<0$. For $m\ge0$, $E_m:\AA^m(M)\to\dot\AA^{(0)}(M)$ is a continuous inclusion map.
\end{prop}

\begin{rem}\label{r: A: AA^m(M) -> AA^m'-k(M)}
By~\eqref{sandwich for AA} and \Cref{p: E_m}, for any $A\in\Diff^k(M)$, the endomorphism $A$ of $\AA(M)$ (\Cref{ss: diff ops on dot AA(M) and AA(M)}) is induced by the continuous linear compositions
\[
\AA^m(M) \xrightarrow{E_m} \dot\AA^{(s)}(M) \xrightarrow{A} \dot\AA^{(s-k)}(M) \xrightarrow{R} \AA^{(s-k)}(M)
\subset \AA^{m'-k}(M)\;,
\]
where $m'=s-n/2-1$ for $m$ and $s$ satisfying the conditions of \Cref{p: E_m}.
\end{rem}

\subsection{$L^2$ half-b-densities}\label{ss: L^2 half-b-densities}

By~\eqref{C^infty(M Omega^s) equiv x^s C^infty(M bOmega^s)},
\begin{align*}
C^\infty(M;\Omega^{-\frac12}\otimes\bOmega^{\frac12})
&\equiv C^\infty(M;\Omega^{-\frac12})\otimes_{C^\infty(M)}C^\infty(M;\bOmega^{\frac12})\\
&\equiv C^\infty(M;\Omega^{-\frac12})\otimes_{C^\infty(M)}x^{-\frac12}C^\infty(M;\Omega^{\frac12})\\
&\equiv x^{-\frac12}C^\infty(M;\Omega^{-\frac12}\otimes\Omega^{\frac12})\equiv x^{-\frac12}C^\infty(M)\;.
\end{align*}
So
\begin{align}
L^2(M;\bOmega^{\frac12})
&\equiv L^2(M;\Omega^{\frac12})\otimes_{C^\infty(M)}C^\infty(M;\Omega^{-\frac12}\otimes\bOmega^{\frac12})\notag\\
&\equiv L^2(M;\Omega^{\frac12})\otimes_{C^\infty(M)}x^{-\frac12}C^\infty(M)
\equiv x^{-\frac12}L^2(M;\Omega^{\frac12})\;.
\label{L^2(M bOmega^1/2) equiv x^-1/2 L^2(M Omega^1/2)}
\end{align}
This is an identity of Hilbert spaces, using the weighted $L^2$ space structure of $x^{-1/2}L^2(M;\Omega^{1/2})$ (\Cref{ss: weighted sps}) and the Hilbert space structure on $L^2(M;\bOmega^{1/2})$ induced by the canonical identity
\begin{equation}\label{L^2(mathring M Omega^1/2) equiv L^2(M bOmega^1/2)}
L^2(\mathring M;\Omega^{\frac12})\equiv L^2(M;\bOmega^{\frac12})\;.
\end{equation}

\subsection{$L^\infty$ half-b-densities}\label{ss: L^infty half-b-densities}

Like in~\eqref{L^2(M bOmega^1/2) equiv x^-1/2 L^2(M Omega^1/2)}, we get
\begin{equation}\label{L^infty(M bOmega^1/2) equiv x^-1/2 L^infty(M Omega^1/2)}
L^\infty(M;\bOmega^{\frac12})\equiv x^{-\frac12}L^\infty(M;\Omega^{\frac12})\;,
\end{equation}
as LCSs endowed with a family of equivalent Banach space norms.

Equip $M$ with a b-metric $g$ (\Cref{ss: b-geometry}), and endow $\mathring M$ with the restriction of $g$, also denoted by $g$. With the corresponding Euclidean/Hermitean structures on $\Omega^{1/2}\mathring M$ and $\bOmega^{1/2}M$, we have the identity of Banach spaces
\begin{equation}\label{L^infty(mathring M Omega^1/2) equiv L^infty(M bOmega^1/2)}
L^\infty(\mathring M;\Omega^{\frac12})\equiv L^\infty(M;\bOmega^{\frac12})\;.
\end{equation}

\subsection{b-Sobolev spaces}\label{ss: b-Sobolev}

For $m\in\N_0$, the \emph{b-Sobolev spaces} of \emph{order} $\pm m$ are defined by the following analogs of~\eqref{H^-s(M) = H^s(M  Omega)'},~\eqref{H^s(M) = ...} and~\eqref{H^-s(M) = ...}: \index{$\Hb^m(M;\bOmega^{\frac12})$}
\begin{gather*}
\Hb^m(M;\bOmega^{\frac12})=\{\,u\in L^2(M;\bOmega^{\frac12})\mid
\Diffb^m(M;\bOmega^{\frac12})\,u\subset L^2(M;\bOmega^{\frac12})\,\}\;,\\
\Diffb^m(M;\bOmega^{\frac12})\,L^2(M;\bOmega^{\frac12})=\Hb^{-m}(M;\bOmega^{\frac12})
= \Hb^m(M;\bOmega^{\frac12})'\;.
\end{gather*}
These are $C^\infty(M)$-modules and Hilbertian spaces with no canonical choice of a scalar product in general; we can use any finite set of $C^\infty(M)$-generators of $\Diffb^m(M;\bOmega^{1/2})$ to define a scalar product on $\Hb^{\pm m}(M;\bOmega^{1/2})$. The intersections and unions of the b-Sobolev spaces are denoted by $\Hb^{\pm\infty}(M;\bOmega^{1/2})$. In particular, $\Hb^\infty(M;\bOmega^{1/2})=\AA^{(0)}(M;\bOmega^{1/2})$.

\subsection{Weighted b-Sobolev spaces}\label{ss: weighted b-Sobolev}

We will also use the \emph{weighted b-Sobolev space} $x^a\Hb^m(M;\bOmega^{1/2})$ ($a\in\R$), \index{$x^a\Hb^m(M;\bOmega^{1/2})$} which is another Hilbertian space with no canonical choice of a scalar product; given a scalar product on $\Hb^m(M;\bOmega^{1/2})$ with norm $\|{\cdot}\|_{\Hb^m}$, we get a scalar product on $x^a\Hb^m(M;\bOmega^{1/2})$ with norm $\|{\cdot}\|_{x^a\Hb^m}$, like in \Cref{ss: weighted sps}. Observe that
\[
\bigcap_{a,m}x^aH^m_{\text{\rm b}}(M;\bOmega^{\frac12})=\dot C^\infty(M;\bOmega^{\frac12})\;.
\]

\subsection{Action of $\Diffb^m(M)$ on weighted b-Sobolev spaces}
\label{ss: (pseudo) diff ops on weighted b-Sobolev}

Like in~\eqref{Diff^m(M E) equiv Diff^m(M)},
\begin{equation}\label{Diffb^m(M bOmega^1/2) equiv Diffb^m(M) equiv Diffb^m(M Omega^1/2)}
\Diffb^m(M;\bOmega^{\frac12})\equiv\Diffb^m(M)\equiv\Diffb^m(M;\Omega^{\frac12})\;.
\end{equation}

By~\eqref{Diffb^k(M) x^a = x^a Diffb^k(M)}, for all $k\in\N_0$, $m\in \Z$ and $a\in \R$, any $A\in\Diffb^k(M;\bOmega^{1/2})$ defines a continuous linear map \cite[Lemma 5.14]{Melrose1993}
\[
A:x^aH^m_{\text{\rm b}}(M;\bOmega^{\frac12})\to x^aH^{m-k}_{\text{\rm b}}(M;\bOmega^{\frac12})\;.
\]
Thus it induces a continuous endomorphism $A$ of $x^aH^{\pm\infty}_{\text{\rm b}}(M;\bOmega^{1/2})$.

\subsection{A description of $\AA(M)$}\label{ss: a description of AA(M)}

In this subsection, unless the contrary is indicated, assume the following properties:
\begin{enumerate}[{\rm(A)}]

\item\label{i: g is of bounded geometry} $\mathring M$ is of bounded geometry with $g$.

\item\label{i: A'}  The collar neighborhood $T$ of $\partial M$ can be chosen so that:
\begin{enumerate}[(a)]

\item\label{i: extension A'} every $A\in\fX(\partial M)$ has an extension $A'\in\fXb(T)$ such that $A'$ is $\varpi$-projectable to $A$, and $A'|_{\mathring T}$ is orthogonal to the $\varpi$-fibers; and

\item\label{i: fX_ub(mathring M)|_mathring T is generated by x partial_x and the vector fields A'} $\fXub(\mathring M)|_{\mathring T}$ is $\Cinftyub(\mathring M)|_{\mathring T}$-generated by $x\partial_x$ and the restrictions $A'|_{\mathring T}$ of the vector fields $A'$ of~\ref{i: extension A'}, for $A\in\fX(\partial M)$.

\end{enumerate}
\end{enumerate}
For instance, these properties hold if $\mathring T$ is cylindrical with $g$ (\Cref{ss: b-geometry}).

\begin{lem}\label{l: fX_b(M)|_T is generated by x partial_x and the vector fields A'}
$\fXb(M)|_T$ is $C^\infty(M)|_T$-generated by $x\partial_x$ and the vector fields $A'$ of~\ref{i: A'}, for $A\in\fX(\partial M)$.
\end{lem}

\begin{proof}
For every $A\in\fX(\partial M)$, there is a unique $A''\in\fX(T)$ such that $A''$ is $\varpi$-projectable to $A$ and $dx(A'')=0$. Since $A'-A''$ is tangent to the $\varpi$-fibers and vanishes on $\partial M$, we have $A'-A''=fx\partial_x$ for some $f\in C^\infty(T)$. Then the result follows because $\fXb(M)|_T$ is $C^\infty(M)|_T$-spanned by $x\partial_x$ and the vector fields $A''$.
\end{proof}

Consider the notation of \Cref{sss: uniform sps,sss: diff ops of bd geom,sss: Sobolev bd geom} for $\mathring M$ with $g$.

\begin{cor}\label{c: injection of C^infty(M)|_T into Cinftyub(mathring M)|_mathring T}
The restriction to $\mathring M$ defines a continuous injection $C^\infty(M)\subset\Cinftyub(\mathring M)$; in particular, $\Cinftyub(\mathring M)$ becomes a $C^\infty(M)$-module.
\end{cor}

\begin{proof}
It is enough to work on a collar neighborhood $T$ of the boundary satisfying~\ref{i: g is of bounded geometry} and~\ref{i: A'}. But, by~\eqref{C^m_ub(M)},~\ref{i: fX_ub(mathring M)|_mathring T is generated by x partial_x and the vector fields A'} and \Cref{l: fX_b(M)|_T is generated by x partial_x and the vector fields A'}, the restriction to $\mathring T$ defines an injection of $C^\infty(M)|_T$ into $\Cinftyub(\mathring M)|_{\mathring T}$.
\end{proof}

\begin{prop}\label{p: Diffb^m(M) generates Diff_ub^m(mathring M)}
There is a canonical identity of $\Cinftyub(\mathring M)$-modules,
\[
\Diffub^m(\mathring M)\equiv\Diffb^m(M)\otimes_{C^\infty(M)}\Cinftyub(\mathring M)\;.
\]
\end{prop}

\begin{proof}
We have to prove that $\Diffub^m(\mathring M)$ is $\Cinftyub(\mathring M)$-spanned by $\Diffb^m(M)$. It is enough to consider the case $m=1$ because the filtered algebra $\Diffub(\mathring M)$ (respectively, $\Diffb(M)$) is spanned by $\Diffub^1(\mathring M)$ (respectively, $\Diffb^1(M)$). Moreover it is clearly enough to work on a collar neighborhood $T$ of the boundary satisfying~\ref{i: g is of bounded geometry} and~\ref{i: A'}. By~\ref{i: fX_ub(mathring M)|_mathring T is generated by x partial_x and the vector fields A'}, \Cref{l: fX_b(M)|_T is generated by x partial_x and the vector fields A'} and Corollary~\ref{c: injection of C^infty(M)|_T into Cinftyub(mathring M)|_mathring T}, the restriction to $\mathring T$ defines an injection of $\Diffb^1(M)|_T$ as a $\Cinftyub(\mathring M)|_{\mathring T}$-spanning subset of $\Diffub^1(\mathring M)|_{\mathring T}$.
\end{proof}

\begin{cor}\label{c: Diffb^m(M bOmega^1/2) generates Diff_ub^m(mathring M Omega^1/2)}
There is a canonical identity of $\Cinftyub(\mathring M)$-modules,
\[
\Diffub^m(\mathring M;\Omega^{\frac12})\equiv\Diffb^m(M;\bOmega^{\frac12})\otimes_{C^\infty(M)}\Cinftyub(\mathring M)\;.
\]
\end{cor}

\begin{proof}
This follows from~\eqref{Diffub^m(M E) equiv Diffub^m(M)} for $\mathring M$,~\eqref{Diffb^m(M bOmega^1/2) equiv Diffb^m(M) equiv Diffb^m(M Omega^1/2)} and \Cref{p: Diffb^m(M) generates Diff_ub^m(mathring M)}.
\end{proof}

\begin{cor}\label{c: H^infty(mathring M Omega^1/2) equiv Hb^infty(M bOmega^1/2)}
$H^m(\mathring M;\Omega^{1/2})\equiv\Hb^m(M;\bOmega^{1/2})$ $(m\in\Z)$ as $C^\infty(M)$-modules and Hilbertian spaces, and therefore $H^{\pm\infty}(\mathring M;\Omega^{1/2})\equiv\Hb^{\pm\infty}(M;\bOmega^{1/2})$.
\end{cor}

\begin{proof}
We show the case where $m\ge0$, and the case where $m<0$ follows by taking dual spaces. For any $m\in\N_0$, let $\{P_k\}$ be a finite $C^\infty(M)$-spanning set of $\Diffb^m(M;\bOmega^{1/2})$, which is also a $\Cinftyub(\mathring M)$-spanning set of $\Diffub^m(\mathring M;\Omega^{1/2})$ by Corollary~\ref{c: Diffb^m(M bOmega^1/2) generates Diff_ub^m(mathring M Omega^1/2)}. Then, by~\eqref{L^2(mathring M Omega^1/2) equiv L^2(M bOmega^1/2)},
\begin{align*}
H^m(\mathring M;\Omega^{\frac12})&=\{\,u\in L^2(\mathring M;\Omega^{\frac12})\mid 
P_ku\in L^2(\mathring M;\Omega^{\frac12})\ \forall k\,\}\\
&\equiv\{\,u\in L^2(M;\bOmega^{\frac12})\mid P_ku\in L^2(M;\bOmega^{\frac12})\ \forall k\,\}
=\Hb^m(M;\bOmega^{\frac12})\;.
\end{align*}
Moreover $\{P_k\}$ can be used to define scalar products on both $H^m(\mathring M;\Omega^{1/2})$ and $\Hb^m(M;\bOmega^{1/2})$, obtaining that the above identity is unitary.
\end{proof}

\begin{prop}
\label{p: x^m+1/2 Hb^infty(M bOmega^1/2) subset AA^m(M Omega^1/2) subset x^m Hb^infty(M bOmega^1/2)}
$\AA^m(M;\Omega^{1/2})\equiv x^{m+1/2}\Hb^\infty(M;\bOmega^{1/2})$ $(m\in\R)$.
\end{prop}

\begin{proof}
By~\eqref{L^infty(M bOmega^1/2) equiv x^-1/2 L^infty(M Omega^1/2)},~\eqref{L^infty(mathring M Omega^1/2) equiv L^infty(M bOmega^1/2)}, \Cref{p: Sobolev embedding with bd geometry,c: H^infty(mathring M Omega^1/2) equiv Hb^infty(M bOmega^1/2)}, we get the following identities and continuous inclusions:
\begin{align*}
\Hb^\infty(M;\bOmega^{\frac12})&\equiv H^\infty(\mathring M;\Omega^{\frac12})
\subset \Cinftyub(\mathring M;\Omega^{\frac12})\notag\\
&\subset L^\infty(\mathring M;\Omega^{\frac12})\equiv L^\infty(M;\bOmega^{\frac12})
\equiv x^{-\frac12}L^\infty(M;\Omega^{\frac12})\;.
\label{Hb^infty(M bOmega^1/2) subset ... equiv L^infty(M bOmega^1/2)}
\end{align*}
So, according to~\Cref{ss: (pseudo) diff ops on weighted b-Sobolev}, every $A\in\Diffb(M;\bOmega^{1/2})$ induces a continuous linear map
\[
x^{m+\frac12}\Hb^\infty(M;\bOmega^{1/2}) \xrightarrow{A}
x^{m+\frac12}\Hb^\infty(M;\bOmega^{\frac12})
\subset x^mL^\infty(M;\Omega^{\frac12})\;.
\]
Hence there is a continuous inclusion
\[
x^{m+\frac12}\Hb^\infty(M;\bOmega^{\frac12})\subset\AA^m(M;\Omega^{\frac12})\;.
\]

On the other hand, by~\eqref{C^-infty(M bOmega^s) equiv C^-infty(M Omega^s)} and the version of~\eqref{x^m C^-infty(M) = C^-infty(M)} with $\Omega^{1/2}M$, for all $a\in\R$,
\[
x^a\AA^m(M;\Omega^{\frac12})\subset x^aC^{-\infty}(M;\Omega^{\frac12})
=C^{-\infty}(M;\Omega^{\frac12})
\equiv C^{-\infty}(M;\bOmega^{\frac12})\;.
\]
Moreover, by~\eqref{L^2(M bOmega^1/2) equiv x^-1/2 L^2(M Omega^1/2)} and~\eqref{Diffb^m(M bOmega^1/2) equiv Diffb^m(M) equiv Diffb^m(M Omega^1/2)}, every $A\in\Diffb(M;\bOmega^{1/2})$ induces a continuous linear map
\[
\AA^m(M;\Omega^{\frac12}) \xrightarrow{A}
x^mL^\infty(M;\Omega^{\frac12})
\subset x^mL^2(M;\Omega^{\frac12})
\equiv x^{m+\frac12}L^2(M;\bOmega^{\frac12})\;.
\]
Hence, by~\eqref{Diffb^k(M) x^a = x^a Diffb^k(M)} and~\eqref{Diffb^m(M bOmega^1/2) equiv Diffb^m(M) equiv Diffb^m(M Omega^1/2)}, $A$ induces a continuous linear map
\[
A:x^{-m-\frac12}\AA^m(M;\Omega^{\frac12})\to L^2(M;\bOmega^{\frac12})\;.
\]
It follows that there is a continuous inclusion
\[
x^{-m-\frac12}\AA^m(M;\Omega^{\frac12})\subset\Hb^\infty(M;\bOmega^{\frac12})\;,
\]
or, equivalently, there is a continuous inclusion
\[
\AA^m(M;\Omega^{\frac12})\subset x^{m+\frac12}\Hb^\infty(M;\bOmega^{\frac12})\;.\qedhere
\]
\end{proof}

\begin{cor}\label{c: H^infty(mathring M) = x^-1/2 Hb^infty(M)}
$H^m(\mathring M)=x^{-1/2}\Hb^m(M)$ $(m\in\Z)$ as $C^\infty(M)$-modules and Hilbertian spaces, and therefore $H^{\pm\infty}(\mathring M)=x^{-1/2}\Hb^{\pm\infty}(M)$.
\end{cor}

\begin{proof}
By~\eqref{C^infty(M Omega^s) equiv x^s C^infty(M bOmega^s)} and Corollary~\ref{c: H^infty(mathring M Omega^1/2) equiv Hb^infty(M bOmega^1/2)},
\begin{align*}
H^m(\mathring M)&\equiv H^m(\mathring M;\Omega^{\frac12})\otimes_{\Cinftyub(\mathring M)}
\Cinftyub(\mathring M;\Omega^{-\frac12})\\
&\equiv H^m(\mathring M;\Omega^{\frac12})\otimes_{\Cinftyub(\mathring M)}
\big(C^\infty(M;\Omega^{-\frac12})\otimes_{C^\infty(M)}\Cinftyub(\mathring M)\big)\\
&\equiv H^m(\mathring M;\Omega^{\frac12})\otimes_{C^\infty(M)}C^\infty(M;\Omega^{-\frac12})\\
&\equiv\Hb^m(M;\bOmega^{\frac12})\otimes_{C^\infty(M)}x^{-\frac12}C^\infty(M;\bOmega^{-\frac12})
\equiv x^{-\frac12}\Hb^m(M)\;.\qedhere
\end{align*}
\end{proof}

\begin{cor}\label{c: AA^m(M) equiv x^m Hb^infty(M)}
$\AA^m(M)\equiv x^m\Hb^\infty(M)\equiv x^{m+1/2}H^\infty(\mathring M)$ $(m\in\R)$.
\end{cor}

\begin{proof}
The second identity is given by \Cref{c: H^infty(mathring M) = x^-1/2 Hb^infty(M)}. By \Cref{p: x^m+1/2 Hb^infty(M bOmega^1/2) subset AA^m(M Omega^1/2) subset x^m Hb^infty(M bOmega^1/2)} and~\eqref{C^infty(M Omega^s) equiv x^s C^infty(M bOmega^s)},
\begin{align*}
x^m\Hb^\infty(M)&\equiv x^{m+\frac12}\Hb^\infty(M;\bOmega^{\frac12})\otimes_{C^\infty(M)}
x^{-\frac12}C^\infty(M;\bOmega^{-\frac12})\\
&\equiv\AA^m(M;\Omega^{\frac12})\otimes_{C^\infty(M)}C^\infty(M;\Omega^{-\frac12})
\equiv\AA^m(M)\;.\qedhere
\end{align*}
\end{proof}

By~\eqref{sandwich for AA} and~\eqref{AA(M) = bigcup_m AA^m(M)}, we get the following consequences of \Cref{c: AA^m(M) equiv x^m Hb^infty(M)}.

\begin{cor}\label{c: Cinftyc(mathring M) is dense in AA^m(M)}
$\Cinftyc(\mathring M)$ is dense in every $\AA^m(M)$ and $\AA^{(s)}(M)$.
\end{cor}

\begin{cor}\label{c: AA(M) equiv bigcup_m x^m Hb^infty(M)}
$\AA(M)\equiv\bigcup_mx^m\Hb^\infty(M)=\bigcup_mx^mH^\infty(\mathring M)$.
\end{cor}

\begin{rem}\label{r: Cinftyc(mathring M) is dense in AA^m(M)}
\Cref{c: Cinftyc(mathring M) is dense in AA^m(M)} and the first identities of \Cref{c: AA^m(M) equiv x^m Hb^infty(M),c: AA(M) equiv bigcup_m x^m Hb^infty(M)} are independent of $g$. So they hold true without the assumptions~\ref{i: g is of bounded geometry} and~\ref{i: A'}. Observe that \Cref{c: Cinftyc(mathring M) is dense in AA^m(M)} is stronger than \Cref{c: Cinftyc(mathring M) is dense in AA(M)}.
\end{rem}

\subsection{Dual-conormal distributions at the boundary}\label{ss: AA'(M)}

Consider the LCHSs \cite[Section~18.3]{Hormander1985-III}, \cite[Chapter~4]{Melrose1996} \index{$\KK'(M)$} \index{$\AA'(M)$} \index{$\dot\AA'(M)$}
\[
\KK'(M)=\KK(M;\Omega)'\;,\quad\AA'(M)=\dot\AA(M;\Omega)'\;,\quad\dot\AA'(M)=\AA(M;\Omega)'\;.
\]
The elements of $\AA'(M)$ (respectively, $\dot\AA'(M)$) will be called \emph{extendible} (respectively, \emph{supported}) \emph{dual-conormal distributions} at the boundary. The following analog of \Cref{c: I'(M L) is complete and Montel} holds true with formally the same proof, using the versions with $\Omega M$ of \Cref{c: dot AA(M) and AA(M) are barreled,c: dot AA(M) is acyclic and Montel,c: AA(M) is acyclic and Montel,c: KK(M) is barreled,c: KK(M) is acyclic and Montel}

\begin{prop}\label{p: KK'(M) AA'(M) dot AA'(M) are complete and Montel}
$\KK'(M)$, $\AA'(M)$ and $\dot\AA'(M)$ are complete Montel spaces.
\end{prop}

We also define the LCHSs \index{$\KK^{\prime\,(s)}(M)$} \index{$\AA^{\prime\,(s)}(M)$} \index{$\dot\AA^{\prime\,(s)}(M)$}
\begin{alignat*}{2}
\KK^{\prime\,(s)}(M)&=\KK^{(-s)}(M;\Omega)'\;,&\quad\KK^{\prime\,m}(M)&=\KK^{-m}(M;\Omega)'\;,\\
\AA^{\prime\,(s)}(M)&=\dot\AA^{(-s)}(M;\Omega)'\;,&\quad\AA^{\prime\,m}(M)&=\dot\AA^{-m}(M;\Omega)'\;,\\
\dot\AA^{\prime\,(s)}(M)&=\AA^{(-s)}(M;\Omega)'\;,&\quad\dot\AA^{\prime\,m}(M)&=\AA^{-m}(M;\Omega)'\;.
\end{alignat*}
Transposing the analogs of~\eqref{I^(s)(M L) subset I^(s')(M L)} and~\eqref{I^m(M L) subset I^m'(M L)} for the spaces $\KK^{(s)}(M;\Omega)$, $\KK^m(M;\Omega)$, $\dot\AA^{(s)}(M;\Omega)$ and $\dot\AA^m(M;\Omega)$, we get continuous linear restriction maps
\begin{gather*}
\KK^{\prime\,(s')}(M)\to\KK^{\prime\,(s)}(M)\;,\quad\KK^{\prime\,m}(M)\to\KK^{\prime\,m'}(M)\;,\\
\AA^{\prime\,(s')}(M)\to\AA^{\prime\,(s)}(M)\;,\quad\AA^{\prime\,m}(M)\to\AA^{\prime\,m'}(M)\;,
\end{gather*}
for $s<s'$ and $m<m'$. These maps form projective spectra, giving rise to projective limits. The spaces $\KK^{\prime\,(s)}(M)$, $\KK^{\prime\,m}(M)$, $\AA^{\prime\,(s)}(M)$ and $\AA^{\prime\,m}(M)$ satisfy the analogs of~\eqref{sandwich for I'}. So the corresponding projective limits satisfy the analogs of~\eqref{varprojlim I'^(s)(M L) equiv varprojlim I'^m(M L)}.

Similarly, transposing the analog of~\eqref{I^(s)(M L) subset I^(s')(M L)} for the spaces $\AA^{(s)}(M;\Omega)$ and the version of~\eqref{AA^m(M) subset AA^m'(M)} with $\Omega M$, by \Cref{c: Cinftyc(mathring M) is dense in AA^m(M),r: Cinftyc(mathring M) is dense in AA^m(M)}, we get continuous inclusions
\[
\dot\AA^{\prime\,(s')}(M)\subset\dot\AA^{\prime\,(s)}(M)\;,\quad\dot\AA^{\prime\,m'}(M)\subset\dot\AA^{\prime\,m}(M)\;,
\]
for $s<s'$ and $m<m'$. The version of~\eqref{sandwich for AA} with $\Omega M$ yields continuous inclusions
\begin{equation}\label{sandwich for ZZ}
\dot\AA^{\prime\,(s)}(M)\supset\dot\AA^{\prime\,m}(M)\supset\dot\AA^{\prime\,(\max\{m,0\})}(M)\quad(m>s+n/2+1)\;.
\end{equation}
Therefore
\begin{equation}\label{bigcap_s ZZ^(s)(M) = bigcap_m ZZ^m(M)}
\bigcap_s\dot\AA^{\prime\,(s)}(M)=\bigcap_m\dot\AA^{\prime\,m}(M)\;.
\end{equation}

The following analogs of \Cref{c: I'^(s)(M L) is bornological,c: I'(M L) equiv varprojlim I'^(s)(M L)} hold true with formally the same proofs, using the versions with $\Omega M$ of \Cref{p: dot AA^(s)(M) and dot AA^(s)(M) are totally reflexive Frechet sps,p: KK^(s)(M) is a totally reflexive Frechet sp,c: dot AA(M) is acyclic and Montel,c: AA(M) is acyclic and Montel,c: KK(M) is acyclic and Montel}; or use \Cref{p: KK'(M) AA'(M) dot AA'(M) are complete and Montel} for an alternative proof.

\begin{cor}\label{c: KK^prime [s](M) AA^prime [s](M) dot AA^prime [s](M) are bornological}
$\KK^{\prime\,(s)}(M)$, $\AA^{\prime\,(s)}(M)$ and $\dot\AA^{\prime\,(s)}(M)$ are bornological and barreled.
\end{cor}

\begin{cor}\label{c: AA'(M) equiv varprojlim AA^prime [s](M)}
We have
\[
\KK'(M)\equiv\varprojlim\KK^{\prime\,(s)}(M)\;,\quad\AA'(M)\equiv\varprojlim\AA^{\prime\,(s)}(M)\;,\quad
\dot\AA'(M)=\bigcap_s\dot\AA^{\prime\,(s)}(M)\;.
\]
\end{cor}

Transposing the versions of~\eqref{dot C^infty(M) = bigcap_m ge 0 x^m C^infty(M) subset C^infty(M)},~\eqref{dot AA(M) subset dot C^-infty(M)},~\eqref{C^infty(M) subset AA(M)},~\eqref{C^infty(M) subset dot AA(M)}  and~\Cref{c: C^infty(M) subset dot AA(M) is dense} with $\Omega M$, we get continuous inclusions \cite[Section~4.6]{Melrose1996}
\begin{align}
C^\infty(M)\subset\AA'(M)\subset C^{-\infty}(M),\dot C^{-\infty}(M)\;,
\label{C^infty(M) subset AA'(M) subset C^-infty(M)}\\
\dot C^\infty(M)\subset\dot\AA'(M)\subset\dot C^{-\infty}(M),C^{-\infty}(M)\;,
\label{dot C^infty(M) subset dot AA'(M) subset dot C^-infty(M)}
\end{align}
and $R:\dot C^{-\infty}(M)\to C^{-\infty}(M)$ restricts to the identity map on $\AA'(M)$ and $\dot\AA'(M)$.

\subsection{Dual-conormal sequence at the boundary}\label{ss: dual conomal seq at the boundary}

Transposing maps in the version of~\eqref{0 -> KK(M) -> dot AA(M) -> AA(M) -> 0} with $\Omega M$, we get the sequence
\[
0\leftarrow\KK'(M) \xleftarrow{\iota^t} \AA'(M) \xleftarrow{R^t} \dot\AA'(M)\leftarrow0\;,
\]
which will called the \emph{dual-conormal sequence at the boundary} of $M$.

\begin{prop}\label{p: dual-conormal seq is exact}
The dual-conormal sequence at the boundary of $M$ is exact in the category of continuous linear maps between LCSs.
\end{prop}

\begin{proof}
By \Cref{p: R: dot AA(M) -> AA(M) is a surj top hom} and \cite[Lemma~7.6]{Wengenroth2003}, it is enough to prove that the map~\eqref{R: dot AA(M) -> AA(M)} satisfies the following condition of ``topological lifting of bounded sets.''

\begin{claim}\label{cl: topological lifting of bounded sets}
For all bounded subset $A\subset\AA(M)$, there is some bounded subset $B\subset\dot\AA(M)$ such that, for all $0$-neighborhood $U\subset\dot\AA(M)$, there is a $0$-neighborhood $V\subset\AA(M)$ so that $A\cap V\subset R(B\cap U)$.
\end{claim}

Since $\AA(M)$ is boundedly retractive (\Cref{c: AA(M) is acyclic and Montel}), $A$ is contained and bounded in some step $\AA^m(M)$. For any $m'>m$, let $E_{m'}:\AA^m(M)\to\dot\AA^{(s)}(M)$ be the partial extension map given by \Cref{p: E_m}. Then $B:=E_{m'}(A)$ is bounded in $\dot\AA^{(s)}(M)$, and therefore in $\dot\AA(M)$. Moreover, given any $0$-neighborhood $U\subset\dot\AA(M)$, there is some $0$-neighborhood $W\subset\AA^{m'}(M)$ so that $E_{m'}(W)\subset U\cap\dot\AA^{(s)}(M)$. By \Cref{c: coincidence of tops on AA^m(M)}, there is some $0$-neighborhood $V\subset\AA(M)$ such that $V\cap\AA^m(M)=W\cap\AA^m(M)$.  Hence $E_{m'}(V\cap\AA^m(M))\subset U\cap\dot\AA^{(s)}(M)$, yielding
\[
A\cap V=R(E_{m'}(A\cap V))\subset R(E_{m'}(A)\cap E_{m'}(V\cap\AA^m(M)))\subset R(B\cap U)\;.\qedhere
\]
\end{proof}

\begin{rem}
\Cref{p: dual-conormal seq is exact} does not agree with \cite[Proposition~4.6.2]{Melrose1996}, which seems to be a minor error of that book project.
\end{rem}

\subsection{$\dot\AA(M)$ and $\AA(M)$ vs $\AA'(M)$}\label{ss: ABC}

Using~\eqref{dot AA(M) subset dot C^-infty(M)},~\eqref{C^infty(M) subset AA(M)} and~\eqref{C^infty(M) subset AA'(M) subset C^-infty(M)}, we have \cite[Proposition~18.3.24]{Hormander1985-III}, \cite[Theorem~4.6.1]{Melrose1996}
\begin{equation}\label{ABC}
\dot\AA(M)\cap\AA'(M)=C^\infty(M)\;. 
\end{equation} 
(Actually, the a priori weaker equality $\AA(M)\cap\AA'(M)=C^\infty(M)$ is proved in \cite[Theorem~4.6.1]{Melrose1996}, but it is equivalent to~\eqref{ABC} because $R=1$ on $\AA'(M)$.)

\subsection{A description of $\dot\AA'(M)$}\label{ss: description of dot AA'(M)}

In this subsection, assume again~\ref{i: g is of bounded geometry} and~\ref{i: A'}.

\begin{cor}\label{c: dot AA^prime m(M) equiv x^m Hb^-infty(M)}
$\dot\AA^{\prime\,m}(M)\equiv x^m\Hb^{-\infty}(M)=x^{m-\frac12}H^{-\infty}(\mathring M)$ $(m\in\R)$.
\end{cor}

\begin{proof}
Apply the version of \Cref{c: AA^m(M) equiv x^m Hb^infty(M)} with $\Omega M$.
\end{proof}

\begin{cor}\label{c: dot AA'(M) equiv bigcap_m x^m Hb^-infty(M)}
$\dot\AA'(M)\equiv\bigcap_mx^m\Hb^{-\infty}(M)=\bigcap_mx^mH^{-\infty}(\mathring M)$.
\end{cor}

\begin{proof}
Apply~\eqref{bigcap_s ZZ^(s)(M) = bigcap_m ZZ^m(M)} and \Cref{c: AA'(M) equiv varprojlim AA^prime [s](M),c: dot AA^prime m(M) equiv x^m Hb^-infty(M)}.
\end{proof}

\begin{cor}\label{c: Cinftyc(mathring M) is dense in dot AA'(M)}
$\Cinftyc(\mathring M)$ is dense in every $\dot\AA^{\prime\,m}(M)$ and in $\dot\AA'(M)$. Therefore the first inclusion of~\eqref{dot C^infty(M) subset dot AA'(M) subset dot C^-infty(M)} is also dense.
\end{cor}

\begin{proof}
Since $\Cinftyc(\mathring M)$ is dense in $H^{-\infty}(\mathring M)$, we get that $\Cinftyc(\mathring M)=x^m\Cinftyc(\mathring M)$ is dense in every $x^mH^{-\infty}(\mathring M)\equiv\dot\AA^{\prime\,m}(M)$ (\Cref{c: dot AA^prime m(M) equiv x^m Hb^-infty(M)}), and therefore in $\dot\AA'(M)$ (\Cref{c: dot AA'(M) equiv bigcap_m x^m Hb^-infty(M)}).
\end{proof}

\begin{rem}\label{r: Cinftyc(mathring M) is dense in dot AA'(M)}
Like in \Cref{r: Cinftyc(mathring M) is dense in AA^m(M)}, \Cref{c: Cinftyc(mathring M) is dense in dot AA'(M)} and the first identities of \Cref{c: dot AA'(M) equiv bigcap_m x^m Hb^-infty(M),c: dot AA'(M) equiv bigcap_m x^m Hb^-infty(M)} are independent of $g$, and hold true without the assumptions~\ref{i: g is of bounded geometry} and~\ref{i: A'}.
\end{rem}

\subsection{Action of $\Diff(M)$ on $\AA'(M)$, $\dot\AA'(M)$ and $\KK'(M)$}\label{ss: Diff(M) on AA'(M)}

According to \Cref{ss: ops,ss: diff ops on dot AA(M) and AA(M)}, any $A\in\Diff(M)$ induces continuous linear endomorphisms $A$ of $\AA'(M)$, $\dot\AA'(M)$ and $\KK'(M)$ \cite[Proposition~4.6.1]{Melrose1996}, which are the transposes of $A^t$ on $\dot\AA(M;\Omega)$, $\AA(M;\Omega)$ and $\KK(M;\Omega)$. If $A\in\Diff^k(M)$, these maps satisfy the analogs of~\eqref{A: I^prime [s](M L E) -> I^prime (s-m)(M L E)}. If $A\in\Diffb(M)$, it induces continuous endomorphisms of $\AA^{\prime\,(s)}(M)$, $\AA^{\prime\,m}(M)$, $\dot\AA^{\prime\,(s)}(M)$ and $\KK^{\prime\,(s)}(M)$.

\section{Conormal sequence}\label{s: conormal seq}

\subsection{Cutting along a submanifold}\label{ss: cutting}

Let $M$ be a closed connected manifold, and $L\subset M$ be a regular closed submanifold of codimension one. $L$ may not be connected, and therefore $M\setminus L$ may have several connected components. First assume also that $L$ is transversely oriented. Then, like in the boundary case \Cref{ss: b-geometry}, there is some real-valued smooth function $x$ on some tubular neighborhood $T$ of $L$ in $M$, with projection $\varpi:T\to L$, so that $L=\{x=0\}$ and $dx\ne0$ on $L$. Any function $x$ satisfying these conditions is called a \emph{defining function} of $L$ on $T$. We can suppose $T\equiv(-\epsilon,\epsilon)_x\times L$, for some $\epsilon>0$, so that $\varpi:T\to L$ is the second factor projection. For any atlas $\{V_j,y_j\}$ of $L$, we get an atlas of $T$ of the form $\{U_j\equiv(-\epsilon,\epsilon)_x\times V_j,(x,y)\}$, whose charts are adapted to $L$. The corresponding local vector fields $\partial_x\in\fX(U_j)$ can be combined to define a vector field $\partial_x\in\fX(T)$; we can consider $\partial_x$ as the derivative operator on $C^\infty(T)\equiv C^\infty((-\epsilon,\epsilon),C^\infty(L))$. For every $j$, $\Diff(U_j,L\cap U_j)$ is spanned by $x\partial_x,\partial_j^1,\dots,\partial_j^{n-1}$ using the operations of $C^\infty(U_j)$-module and algebra, where $\partial_j^\alpha=\partial/\partial y_j^\alpha$. Using $T\equiv(-\epsilon,\epsilon)_x\times L$, any $A\in\Diff(L)$ induces an operator $1\otimes A\in\Diff(T,L)$, such that $(1\otimes A)(u(x)v(y))=u(x)\,(Av)(y)$ for $u\in C^\infty(-\epsilon,\epsilon)$ and $v\in C^\infty(L)$. This defines a canonical injection $\Diff(L)\equiv1\otimes\Diff(L)\subset\Diff(T,L)$ so that $(1\otimes A)|_L=A$. (This also shows the surjectivity of~\eqref{Diff(M L) -> Diff(L)} in this case.) Moreover $\Diff(T)$ (respectively, $\Diff(T,L)$) is spanned by $\partial_x$ (respectively, $x\partial_x$) and $1\otimes\Diff(L)$ using the operations of $C^\infty(T)$-module and algebra. Clearly,
\begin{equation}\label{[partial_x 1 otimes Diff(L)] = 0}
[\partial_x,1\otimes\Diff(L)]=0\;,\quad[\partial_x, x\partial_x]=\partial_x\;,
\end{equation}
yielding
\begin{equation}\label{[partial_x Diff^k(T L)] subset ...}
[\partial_x,\Diff^k(T,L)]\subset\Diff^k(T,L)+\Diff^{k-1}(T,L)\,\partial_x\;.
\end{equation}
$\Diff^k(T,L)$ and $\Diff^k(T)$ satisfy the obvious versions of~\eqref{Diffb^k(M) x^a = x^a Diffb^k(M)} and~\eqref{Diff^k(M) x^a subset x^a-k Diff^k(M)}.

For a vector bundle $E$ over $M$, there is an identity $E_T\equiv(-\epsilon,\epsilon)\times E_L$ over $T\equiv(-\epsilon,\epsilon)\times L$, which can be used to define $\partial_x\in\Diff^1(T;E)$ using the above charts. With this interpretation of $\partial_x$ and using tensor products like in~\eqref{C^infty(M)-tensor product description of C^pm infty_./c(M E)}, the vector bundle versions of the properties and spaces of distributions of this section are straightforward.

Let $\bfM$ \index{$\bfM$} be the smooth manifold with boundary defined by ``cutting'' $M$ along $L$; i.e., modifying $M$ only on the tubular neighborhood $T\equiv(-\epsilon,\epsilon)\times L$, which is replaced with $\bfT=((-\epsilon,0]\sqcup[0,\epsilon))\times L$ in the obvious way. ($\bfM$ is the blowing-up $[M,L]$ of $M$ along $L$ \cite[Chapter~5]{Melrose1996}.) Thus $\partial\bfM\equiv L\sqcup L$ because $L$ is transversely oriented, and $\mathring\bfM\equiv M\setminus L$. A canonical projection $\bfpi:\bfM\to M$ \index{$\bfpi$} is defined as the combination of the identity map $\mathring\bfM\to M\setminus L$ and the map $\bfT\to T$ given by the product of the canonical projection $(-\epsilon,0]\sqcup[0,\epsilon)\to(-\epsilon,\epsilon)$ and $\id_L$. This projection realizes $M$ as a quotient space of $\bfM$ by the equivalence relation defined by the homeomorphism $h\equiv h_0\times\id$ of $\partial\bfM\equiv\partial\bfT=(\{0\}\sqcup\{0\})\times L$, where $h_0$ switches the two points of $\{0\}\sqcup\{0\}$. Moreover $\bfpi:\bfM\to M$ is a local embedding of a compact manifold with boundary to a closed manifold of the same dimension. 

Like in \Cref{ss: pull-back and push-forward of distrib sections}, we have the continuous linear pull-back map
\begin{equation}\label{bfpi^*: C^infty(M) -> C^infty(bfM)}
\bfpi^*:C^\infty(M)\to C^\infty(\bfM)\;,
\end{equation}
which is clearly injective. Then the transpose of the version of~\eqref{bfpi^*: C^infty(M) -> C^infty(bfM)} with $\Omega M$ and $\Omega\bfM\equiv\bfpi^*\Omega M$ is the continuous linear push-forward map
\begin{equation}\label{bfpi_*: dot C^-infty(bfM) -> C^-infty(M)}
\bfpi_*:\dot C^{-\infty}(\bfM)\to C^{-\infty}(M)\;,
\end{equation}
which is surjective by a consequence of the Hahn-Banach theorem \cite[Theorem~II.4.2]{Schaefer1971}.

After distinguishing a connected component $L_0$ of $L$, let $\widetilde M$ \index{$\widetilde M$} and $\widetilde L$ be the quotients of $\bfM\sqcup\bfM\equiv\bfM\times\Z_2$ and $\partial\bfM\sqcup\partial\bfM\equiv\partial\bfM\times\Z_2$ by the equivalence relation generated by $(p,a)\sim(h(p),a)$ if $\bfpi(p)\in L\setminus L_0$ and $(p,a)\sim(h(p),a+1)$ if $\bfpi(p)\in L_0$ ($p\in\bfpi^{-1}(L)=\partial\bfM$ in both cases). Let us remark that $\widetilde M$ may not be homeomorphic to the double of $\bfM$, which is the quotient of $\bfM\times\Z_2$ by the equivalence relation generated by $(p,0)\sim(p,1)$, for $p\in\partial\bfM$. Note that $\widetilde M$ is a closed connected manifold and $\widetilde L$ is a closed regular submanifold. Moreover the quotient $\widetilde T$ of $\bfT\sqcup\bfT$ becomes a tubular neighborhood of $\widetilde L$ in $\widetilde M$. The combination $\bfpi\sqcup\bfpi:\bfM\sqcup\bfM\to M$ induces a two-fold covering map $\tilde\pi:\widetilde M\to M$, whose restrictions to $\widetilde L$ and $\widetilde T$ are trivial two-fold coverings of $L$ and $T$, respectively; i.e., $\widetilde L\equiv L\sqcup L$ and $\widetilde T\equiv T\sqcup T$. The group of deck transformations of $\tilde\pi:\widetilde M\to M$ is $\{\id,\sigma\}$, where $\sigma:\widetilde M\to\widetilde M$ is induced by the map $\sigma_0:\bfM\times\Z_2\to\bfM\times\Z_2$ defined by switching the elements of $\Z_2$. The composition of the injection $\bfM\to\bfM\times\Z_2$, $p\mapsto(p,0)$, with the quotient map $\bfM\sqcup\bfM\to\widetilde M$ is a smooth embedding $\bfM\to\widetilde M$. This will be considered as an inclusion map of a regular submanifold with boundary, obtaining $\partial\bfM\equiv\widetilde L$.

Since $\tilde\pi$ is a two-fold covering map, we have continuous linear maps (\Cref{ss: pull-back and push-forward of distrib sections})
\begin{gather}
\tilde\pi_*:C^\infty(\widetilde M) \to C^\infty(M)\;,\quad
\tilde\pi^*:C^\infty(M) \to C^\infty(\widetilde M)\;,\notag \\
\tilde\pi^*:C^{-\infty}(M) \to C^{-\infty}(\widetilde M)\;,\quad
\tilde\pi_*:C^{-\infty}(\widetilde M) \to C^{-\infty}(M)\;,\label{tilde pi^*: C^-infty(M) -> C^-infty(widetilde M)}
\end{gather}
both pairs of maps satisfying
\begin{equation}\label{tilde pi_* tilde pi^* = 2}
\tilde\pi_*\tilde\pi^*=2\;,\quad\tilde\pi^*\tilde\pi_*=A_\sigma\;,
\end{equation}
where $A_\sigma:C^{\pm\infty}(\widetilde M)\to C^{\pm\infty}(\widetilde M)$ is given by $A_\sigma u=u+\sigma_*u$. Using the continuous linear restriction and inclusion maps given by~\eqref{R: C^infty(widetilde M) -> C^infty(M)} and~\eqref{dot C^-infty(M) subset C^-infty(breve M)}, we get the commutative diagrams
\begin{equation}\label{CD: tilde pi_* iota = bfpi_*}
\begin{CD}
C^\infty(\widetilde M) @>R>> C^\infty(\bfM) \\
@A{\tilde\pi^*}AA @AA{\bfpi^*}A \\
C^\infty(M) @= C^\infty(M)\;,\hspace{-.2cm}
\end{CD}
\qquad
\begin{CD}
\dot C^{-\infty}(\bfM) @>{\iota}>> C^{-\infty}(\widetilde M) \\
@V{\bfpi_*}VV @VV{\tilde\pi_*}V \\
C^{-\infty}(M) @= C^{-\infty}(M)\;,\hspace{-.2cm}
\end{CD}
\end{equation}
the second one being the transpose of the density-bundles version of the first one.

\subsection{Lift of differential operators from $M$ to $\widetilde M$}\label{ss: lift of diff ops}

For any $A\in\Diff(M)$, let $\widetilde A\in\Diff(\widetilde M)$ denote its lift via the covering map $\tilde\pi:\widetilde M\to M$.  The action of $\widetilde A$ on $C^{\pm\infty}(\widetilde M)$ corresponds to the action of $A$ on $C^{\pm\infty}(M)$ via $\tilde\pi^*:C^{\pm\infty}(M)\to C^{\pm\infty}(\widetilde M)$ and $\tilde\pi_*:C^{\pm\infty}(\widetilde M)\to C^{\pm\infty}(M)$. According to~\eqref{restriction map Diff(breve M) -> Diff(M)}, $\widetilde A|_{\bfM}\in\Diff(\bfM)$ is the lift of $A$ via the local embedding $\bfpi:\bfM\to M$, sometimes also denoted by $\widetilde A$. \index{$\widetilde A$} The action of $\widetilde A$ on $C^\infty(\bfM)$ (respectively, $C^{-\infty}(\bfM)$) corresponds to the action of $A$ on $C^\infty(M)$ (respectively, $C^{-\infty}(M)$) via $\bfpi^*:C^\infty(M)\to C^\infty(\bfM)$ (respectively, $\bfpi_*:C^{-\infty}(\bfM)\to C^{-\infty}(M)$). If $A\in\Diff(M,L)$, then $\widetilde A\in\Diff(\widetilde M,\widetilde L)$ and $\widetilde A|_{\bfM}\in\Diffb(\bfM)$ by~\eqref{Diff(breve M,partial M) = Diffb(M)}.

\subsection{The spaces $C^{\pm\infty}(M,L)$}\label{ss: C^-infty(M L)}

Consider the closed subspaces, \index{$C^\infty(M,L)$} \index{$C^k(M,L)$}
\begin{equation}\label{C^infty(M L) subset C^infty(M)}
C^\infty(M,L)\subset C^\infty(M)\;,\quad C^k(M,L)\subset C^k(M)\quad(k\in\N_0)\;,
\end{equation}
consisting of functions that vanish to all orders at the points of $L$ in the first case, and that vanish up to order $k$ at the points of $L$ in the second case. Then let \index{$C^{-\infty}(M,L)$} \index{$C^{\prime\,-k}(M,L)$}
\[
C^{-\infty}(M,L)=C^\infty(M,L;\Omega)'\;,\quad C^{\prime\,-k}(M,L)=C^k(M,L;\Omega)'\;.
\]
Note that~\eqref{bfpi^*: C^infty(M) -> C^infty(bfM)} restricts to TVS-isomorphisms
\begin{equation}\label{bfpi^*: C^infty(M L) cong dot C^infty(bfM)}
\bfpi^*:C^\infty(M,L)\xrightarrow{\cong}\dot C^\infty(\bfM)\;,\quad
\bfpi^*:C^k(M,L)\xrightarrow{\cong}\dot C^k(\bfM)\;.
\end{equation}
Taking the transposes of its versions with density bundles, it follows that~\eqref{bfpi_*: dot C^-infty(bfM) -> C^-infty(M)} restricts to TVS-isomorphisms
\begin{equation}\label{bfpi_*: C^-infty(bfM) cong C^-infty(M L)}
\bfpi_*:C^{-\infty}(\bfM)\xrightarrow{\cong}C^{-\infty}(M,L)\;,\quad
\bfpi_*:C^{\prime\,-k}(\bfM)\xrightarrow{\cong}C^{\prime\,-k}(M,L)\;.
\end{equation}
So the spaces $C^\infty(M,L)$, $C^k(M,L)$, $C^{-\infty}(M,L)$ and $C^{\prime\,-k}(M,L)$ satisfy the analogs of~\eqref{C^prime -k'(M E) supset C^prime -k(M E)} and~\eqref{bigcap_k C^k_./c(M) = C^infty_./c(M)}.

On the other hand, there are Hilbertian spaces $H^r(M,L)$ ($r>n/2$) \index{$H^r(M,L)$} and $H^{\prime\,s}(M,L)$ ($s\in\R$), \index{$H^{\prime\,s}(M,L)$} continuously included in $C^0(M,L)$ and $C^{-\infty}(M,L)$, respectively, such that the second map of~\eqref{bfpi^*: C^infty(M L) cong dot C^infty(bfM)} for $k=0$ and the first map of~\eqref{bfpi_*: C^-infty(bfM) cong C^-infty(M L)} restrict to a TVS-isomorphisms
\begin{equation}\label{bfpi^*: H^r(M L) cong dot H^r(bfM)}
\bfpi^*:H^r(M,L)\xrightarrow{\cong}\dot H^r(\bfM)\;,\quad
\bfpi_*:H^s(\bfM)\xrightarrow{\cong}H^{\prime\,s}(M,L)\;.
\end{equation}
By~\eqref{dot H^s(M) equiv H^-s(M Omega)'},
\begin{equation}\label{H^prime -r(M,L) equiv H^r(M L Omega)'}
H^{\prime\,-r}(M,L)\equiv H^r(M,L;\Omega)'\;,\quad H^r(M,L)\equiv H^{\prime\,-r}(M,L;\Omega)'\;.
\end{equation}
Now, the second identity of~\eqref{H^prime -r(M,L) equiv H^r(M L Omega)'} can be used to extend the definition of $H^r(M,L)$ for all $r\in\R$.

Alternatively, we may also use trace theorems \cite[Theorem~7.53 and~7.58]{Adams1975} to define $H^m(M,L)$ for $m\in\Z^+$, and then use the first identity of~\eqref{H^prime -r(M,L) equiv H^r(M L Omega)'} to define $H^{\prime\,-m}(M,L)$.

From~\eqref{bfpi^*: C^infty(M) -> C^infty(bfM)},~\eqref{bfpi_*: dot C^-infty(bfM) -> C^-infty(M)},~\eqref{bfpi^*: H^r(M L) cong dot H^r(bfM)} and the analogs of~\eqref{H^s(M) subset C^k(M) subset H^k(M)}--\eqref{C^infty(M) = bigcap_s H^s(M)} mentioned in \Cref{ss: supported and extendible Sobolev sps}, we get
\begin{alignat}{2}
C^\infty(M,L)&=\bigcap_rH^r(M,L)\;,&\quad C^{-\infty}(M,L)&=\bigcup_sH^{\prime\,s}(M,L)\;,
\label{C^infty(M L) = bigcap_r H^r(M L)}\\
\intertext{as well as a continuous inclusion and a continuous linear surjection,}
C^\infty(M)&\subset\bigcap_sH^{\prime\,s}(M,L)\;,&\quad C^{-\infty}(M)&\leftarrow\bigcup_rH^r(M,L)\;.
\label{C^infty(M) subset bigcap_sH^prime s(M L)}
\end{alignat}
By~\eqref{H^prime -r(M,L) equiv H^r(M L Omega)'} and~\eqref{C^infty(M L) = bigcap_r H^r(M L)},
\begin{equation}\label{C^infty(M L) = C^-infty(M L Omega)'}
C^\infty(M,L)=C^{-\infty}(M,L;\Omega)'\;.
\end{equation}

\Cref{p: dot C^-infty(M) and C^-infty(M) are barreled ...} and~\eqref{bfpi_*: C^-infty(bfM) cong C^-infty(M L)} have the following consequence.

\begin{cor}\label{c: C^-infty(M L) is barreled ...}
$C^{-\infty}(M,L)$ is a barreled, ultrabornological, webbed, acyclic DF Montel space, and therefore complete, boundedly retractive and reflexive.
\end{cor}

The transpose of the version of the first inclusion of~\eqref{C^infty(M L) subset C^infty(M)} with $\Omega M$ is a continuous linear restriction map
\begin{equation}\label{R :C^-infty(M) -> C^-infty(M L)}
R:C^{-\infty}(M)\to C^{-\infty}(M,L)\;,
\end{equation}
whose restriction to $C^\infty(M)$ is the identity. This map can be also described as the composition
\[
C^{-\infty}(M) \xrightarrow{\tilde\pi^*} C^{-\infty}(\widetilde M) \xrightarrow{R} C^{-\infty}(\bfM) \xrightarrow{\bfpi_*} C^{-\infty}(M,L)\;.
\]
The canonical pairing between $C^\infty(M)$ and $C^\infty(M,L;\Omega)$ defines a continuous inclusion
\begin{equation}\label{C^infty(M) subset C^-infty(M L)}
C^\infty(M)\subset C^{-\infty}(M,L)
\end{equation}
such that~\eqref{R :C^-infty(M) -> C^-infty(M L)} is the identity on $C^\infty(M)$. We also get commutative diagrams
\begin{equation}\label{CD dot C^-infty(bfM) -R-> C^-infty(bfM)}
\begin{CD}
C^\infty(\bfM) @<{\iota}<< \dot C^\infty(\bfM) \\
@A{\bfpi^*}AA @A{\cong}A{\bfpi^*}A \\
C^\infty(M) @<{\iota}<< C^\infty(M,L)\;,\hspace{-.2cm}
\end{CD}
\qquad
\begin{CD}
\dot C^{-\infty}(\bfM) @>R>> C^{-\infty}(\bfM) \\
@V{\bfpi_*}VV @V{\cong}V{\bfpi_*}V \\
C^{-\infty}(M) @>R>> C^{-\infty}(M,L)\;,\hspace{-.2cm}
\end{CD}
\end{equation}
the second one being the transpose of the density-bundles version of the first one.

\subsection{The space $C^{-\infty}_L(M)$}\label{ss: C^-infty_L(M)}

The closed subspaces of elements supported in $L$, \index{$C^{-\infty}_L(M)$} \index{$C^{\prime\,-k}_L(M)$} \index{$H^s_L(M)$}
\[
C^{-\infty}_L(M)\subset C^{-\infty}(M)\;,\quad C^{\prime\,-k}_L(M)\subset C^{\prime\,-k}(M)\;,\quad H^s_L(M)\subset H^s(M)\;,
\]
are the null spaces of restrictions of~\eqref{R :C^-infty(M) -> C^-infty(M L)}. These spaces satisfy continuous inclusions analogous to~\eqref{C^prime -k'(M E) supset C^prime -k(M E)},~\eqref{H^s(M) subset H^s'(M)} and~\eqref{H^-s(M) supset C^prime -k(M) supset H^-k(M)}. 

According to~\eqref{dot C^-infty_partial M(M) equiv C^-infty_partial M(breve M)} and \Cref{ss: cutting},
\begin{align}
\dot C^{-\infty}_{\partial\bfM}(\bfM)&\equiv C^{-\infty}_{\widetilde L}(\widetilde M)
\equiv C^{-\infty}_{\widetilde L}(\widetilde T)\equiv C^{-\infty}_L(T)\oplus C^{-\infty}_L(T)\notag\\
&\equiv C^{-\infty}_L(M)\oplus C^{-\infty}_L(M)\;.
\label{dot C^-infty_partial bfM(bfM) equiv C^-infty_L(M) oplus C^-infty_L(M)}
\end{align}
The maps~\eqref{bfpi_*: dot C^-infty(bfM) -> C^-infty(M)} and~\eqref{tilde pi^*: C^-infty(M) -> C^-infty(widetilde M)} have restrictions
\begin{equation}\label{bfpi_*: dot C^-infty_partial bfM(bfM) -> C^-infty_L(M)}
\bfpi_*=\tilde\pi_*:\dot C^{-\infty}_{\partial\bfM}(\bfM)\to C^{-\infty}_L(M)\;,\quad
\tilde\pi^*:C^{-\infty}_L(M)\to\dot C^{-\infty}_{\partial\bfM}(\bfM)\;.
\end{equation}
Using~\eqref{dot C^-infty_partial bfM(bfM) equiv C^-infty_L(M) oplus C^-infty_L(M)}, these maps are given by $\bfpi_*(u,v)=u+v$ and $\tilde\pi^*u=(u,u)$.

From~\eqref{dot C^-infty_partial bfM(bfM) equiv C^-infty_L(M) oplus C^-infty_L(M)}, \Cref{p: dot C^-infty_partial M(M) = bigcup_s H^s_partial M(M),c: dot C^-infty_partial M(M) is barreled ...}, we get the following.

\begin{cor}\label{c: C^-infty_L(M) = bigcup_k C^-k_L(M)}
$C^{-\infty}_L(M)$ is a limit subspace of the LF-space $C^{-\infty}(M)$.
\end{cor}

\begin{cor}\label{c: C^-infty_L(M) is barreled ...}
$C^{-\infty}_L(M)$ is a barreled, ultrabornological, webbed, acyclic DF Montel space, and therefore complete, boundedly retractive and reflexive.
\end{cor}

Moreover the right-hand side diagram of~\eqref{CD dot C^-infty(bfM) -R-> C^-infty(bfM)} can be completed to get the commutative diagram
\begin{equation}\label{CD: ... dot C^-infty(bfM) -R-> C^-infty(bfM) -> 0}
\begin{CD}
0\to\dot C^{-\infty}_{\partial\bfM}(\bfM) @>{\iota}>> \dot C^{-\infty}(\bfM) @>R>> C^{-\infty}(\bfM)\to0 \\
@V{\bfpi_*}VV @V{\bfpi_*}VV @V{\cong}V{\bfpi_*}V \\
0\to C^{-\infty}_L(M) @>{\iota}>> C^{-\infty}(M) @>R>> C^{-\infty}(M,L)\to0\;.\hspace{-.2cm}
\end{CD}
\end{equation}

\begin{prop}\label{p: bfpi_*: dot C^-infty(bfM) -> C^-infty(M) is surj top hom}
The maps~\eqref{bfpi_*: dot C^-infty(bfM) -> C^-infty(M)} and~\eqref{R :C^-infty(M) -> C^-infty(M L)} are surjective topological homomorphisms.
\end{prop}

\begin{proof}
In~\eqref{CD: ... dot C^-infty(bfM) -R-> C^-infty(bfM) -> 0}, the top row is exact in the category of continuous linear maps between LCSs by \Cref{c: 0 -> dot C^-infty_partial M(M) -> dot C^-infty(M) -> C^-infty(M) -> 0 is exact}, the left-hand side vertical map is onto by~\eqref{tilde pi_* tilde pi^* = 2}, and the right-hand side vertical map is a TVS-isomorphism. Then, by the commutativity of its right-hand side square, the map~\eqref{R :C^-infty(M) -> C^-infty(M L)} is surjective, and therefore the bottom row of~\eqref{CD: ... dot C^-infty(bfM) -R-> C^-infty(bfM) -> 0} is exact in the category of linear maps between vector spaces. 

By the above properties, chasing~\eqref{CD: ... dot C^-infty(bfM) -R-> C^-infty(bfM) -> 0}, we get that~\eqref{bfpi_*: dot C^-infty(bfM) -> C^-infty(M)} is surjective. Since $\dot C^{-\infty}(\bfM)$ is webbed (\Cref{p: dot C^-infty(M) and C^-infty(M) are barreled ...}) and $C^{-\infty}(M)$ ultrabornological, by the open mapping theorem, it also follows that~\eqref{bfpi_*: dot C^-infty(bfM) -> C^-infty(M)} is a topological homomorphism.

To get that~\eqref{R :C^-infty(M) -> C^-infty(M L)} is another surjective topological homomorphism, apply the commutativity of the right-hand side square of~\eqref{CD: ... dot C^-infty(bfM) -R-> C^-infty(bfM) -> 0} and the above properties.
\end{proof}

\begin{cor}\label{c: 0 -> C^-infty_L(M) -> C^-infty(M) -> C^-infty(M L) -> 0 is exact}
The bottom row of~\eqref{CD: ... dot C^-infty(bfM) -R-> C^-infty(bfM) -> 0} is exact in the category of continuous linear maps between LCSs.
\end{cor}

\begin{cor}\label{c: C^infty(M) subset C^-infty(M L) is dense}
The inclusion~\eqref{C^infty(M) subset C^-infty(M L)} is dense.
\end{cor}

\begin{proof}
Apply \Cref{p: bfpi_*: dot C^-infty(bfM) -> C^-infty(M) is surj top hom} and the density of $C^\infty(M)$ in $C^{-\infty}(M)$.
\end{proof}

\subsection{A description of $C^{-\infty}_L(M)$}\label{ss: description of C^-infty_L(M)}

According to \Cref{ss: diff ops,ss: cutting} and~\eqref{u in C^-infty_c mapsto delta_L^u}, we have the subspaces
\begin{equation}\label{partial_x^m C^-infty(L Omega^-1 NL) subset C^-infty_L(M)}
\partial_x^mC^{-\infty}(L;\Omega^{-1}NL)\subset C^{-\infty}_L(M)
\end{equation}
for $m\in\N_0$, and continuous linear isomorphisms
\begin{equation}\label{C^-infty(L Omega^-1 NL) cong partial_x^m C^-infty(L Omega^-1 NL)}
\partial_x^m:C^{-\infty}(L;\Omega^{-1}NL)\xrightarrow{\cong}\partial_x^mC^{-\infty}(L;\Omega^{-1}NL)\;.
\end{equation}
They induce a continuous linear injection
\begin{equation}\label{bigoplus_m C^0_m -> C^-infty_L(M)}
\bigoplus_{m=0}^\infty C^0_m\to C^{-\infty}_L(M)\;,
\end{equation}
where $C^0_m=C^{-\infty}(L;\Omega^{-1}NL)$ for all $m$.

\begin{prop}\label{p: bigoplus_m C^0_m -> C^-infty_L(M)}
The map~\eqref{bigoplus_m C^0_m -> C^-infty_L(M)} is a TVS-isomorphism, which restricts to TVS-isomorphisms
\begin{equation}\label{bigoplus_m^k C^m-k(L Omega^-1NL) cong C^prime -k_L(M)}
\bigoplus_{m=0}^kC^{m-k}(L;\Omega^{-1}NL)\xrightarrow{\cong}C^{\prime\,-k}_L(M)\quad(k\in\N_0)\;.
\end{equation}
\end{prop}

\begin{proof}
In the case where $M=\R^n$ and $L$ is a linear subspace, it is known that~\eqref{bigoplus_m^k C^m-k(L Omega^-1NL) cong C^prime -k_L(M)} is a linear isomorphism \cite[Theorem~2.3.5 and Example~5.1.2]{Hormander1983-I}, which is easily seen to be continuous. This can be easily extended to arbitrary $M$ by using charts of $M$ adapted to $L$. Then we get the continuous linear isomorphism~\eqref{bigoplus_m C^0_m -> C^-infty_L(M)} by taking the locally convex inductive limit of~\eqref{bigoplus_m^k C^m-k(L Omega^-1NL) cong C^prime -k_L(M)} as $k\uparrow\infty$. Since $\bigoplus_mC^0_m$ is webbed and $C^{-\infty}_L(M)$ ultrabornological (\Cref{c: C^-infty_L(M) is barreled ...}), the map~\eqref{bigoplus_m C^0_m -> C^-infty_L(M)} is a TVS-isomorphism by the open mapping theorem.
\end{proof}

\begin{rem}\label{r: bigoplus_m C^0_m -> C^-infty_L(M)}
\Cref{p: bigoplus_m C^0_m -> C^-infty_L(M)} reconfirms \Cref{c: C^-infty_L(M) = bigcup_k C^-k_L(M)}.
\end{rem}

\begin{rem}[{See \cite[Exercise~3.3.18]{Melrose1996}}]\label{r: description of dot C^infty_partial M(M)}
In \Cref{ss: dot C^-infty_partial M(M)}, for any compact manifold with boundary $M$, the analog of \Cref{p: bigoplus_m C^0_m -> C^-infty_L(M)} for $\dot C^{-\infty}_{\partial M}(M)$ follows from the application of \Cref{p: bigoplus_m C^0_m -> C^-infty_L(M)} to $C^{-\infty}_{\partial M}(\breve M)$.
\end{rem}

\begin{cor}\label{c: C^-infty(L Omega^-1 NL) cong partial_x^m C^-infty(L Omega^-1 NL) is a TVS-iso}
Every map~\eqref{C^-infty(L Omega^-1 NL) cong partial_x^m C^-infty(L Omega^-1 NL)} is a TVS-isomorphism.
\end{cor}

\subsection{Action of $\Diff(M)$ on $C^{-\infty}(M,L)$ and $C^{-\infty}_L(M)$}
\label{ss: action of Diff(M) on C^-infty(M L) and C^-infty_L(M)}

For every $A\in\Diff(M)$, $A^t$ preserves $C^\infty(M,L;\Omega)$, and therefore $A$ induces a continuous linear map $A=A^{tt}$ on $C^{-\infty}(M,L)$. By locality, it restricts to a continuous endomorphism $A$ of $C^{-\infty}_L(M)$.

\subsection{The space $J(M,L)$}\label{ss: J(M L)}

According to \Cref{ss: conormality at the boundary - Sobolev order,ss: C^-infty(M L)}, there is a LCHS $J(M,L)$, continuously included in $C^{-\infty}(M,L)$, so that~\eqref{bfpi_*: C^-infty(bfM) cong C^-infty(M L)} restricts to a TVS-isomorphism \index{$J(M,L)$}
\begin{equation}\label{bfpi_*: AA(bfM) cong J(M L)}
\bfpi_*:\AA(\bfM)\xrightarrow{\cong}J(M,L)\;,
\end{equation}
where $\AA(\bfM)$ is defined in~\eqref{dot AA(M) - AA(M)}. By~\eqref{dot AA(M)|_mathring M AA(M) subset C^infty(mathring M)}, there is a continuous inclusion
\[
J(M,L)\subset C^\infty(M\setminus L)\;.
\]
We also get spaces $J^{(s)}(M,L)$ \index{$J^{(s)}(M,L)$} and $J^m(M,L)$ \index{$J^m(M,L)$} ($s,m\in\R$) corresponding to $\AA^{(s)}(\bfM)$ and $\AA^m(\bfM)$ via~\eqref{bfpi_*: AA(bfM) cong J(M L)}. Extend $|x|$ to a function $\bfx$ on $M$ that is positive and smooth on $M\setminus L$. Its lift $\bfpi^*\bfx$ is a boundary defining function of $\bfM$, also denoted by $\bfx$. Using the first map of~\eqref{bfpi_*: C^-infty(bfM) cong C^-infty(M L)} and second map of~\eqref{bfpi^*: H^r(M L) cong dot H^r(bfM)}, and according to \Cref{ss: lift of diff ops}, we can also describe
\begin{align}
J^{(s)}(M,L)&=\{\,u\in C^{-\infty}(M,L)\mid\Diff(M,L)\,u\subset H^{\prime\,s}(M,L)\,\}\;,\label{J^(s)(M,L)}\\
J^m(M,L)&=\{\,u\in C^{-\infty}(M,L)\mid\Diff(M,L)\,u\subset\bfx^mL^\infty(M)\,\}\;,\notag
\end{align}
equipped with topologies like in \Cref{ss: conormality at the boundary - Sobolev order,ss: description of AA(M) by bounds}. These spaces satisfy the analogs of~\eqref{I^(s)(M L) subset I^(s')(M L)},~\eqref{dot AA(M) - AA(M)} and~\eqref{AA^m(M) subset AA^m'(M)}--\eqref{AA(M) = bigcup_m AA^m(M)}. By~\eqref{C^infty(M) subset bigcap_sH^prime s(M L)} and~\eqref{J^(s)(M,L)}, there are continuous inclusions, \index{$J^{(\infty)}(M,L)$}
\begin{equation}\label{C^infty(M) subset J^(infty)(M L)}
C^\infty(M)\subset J^{(\infty)}(M,L):=\bigcap_sJ^{(s)}(M,L)\;,\quad J(M,L)\subset C^{-\infty}(M,L)\;;
\end{equation}
in particular, $J(M,L)$ is Hausdorff. Moreover the following analogs of \Cref{p: dot AA^(s)(M) and dot AA^(s)(M) are totally reflexive Frechet sps,c: dot AA(M) and AA(M) are barreled,c: coincidence of tops on AA^m(M),c: Cinftyc(mathring M) is dense in AA(M),c: AA(M) is acyclic and Montel} hold true.

\begin{cor}\label{c: J^(s)(M L) is totally reflexive Frechet sps}
$J^{(s)}(M, L)$ is a totally reflexive Fr\'echet space.
\end{cor}

\begin{cor}\label{c: J(M L) is barreled}
$J(M,L)$ is barreled, ultrabornological and webbed.
\end{cor}

\begin{cor}\label{c: coincidence of tops on J^m(M L)}
If $m'<m$, then the topologies of $J^{m'}(M,L)$ and $C^\infty(M\setminus L)$ coincide on $J^m(M,L)$. Therefore the topologies of $J(M,L)$ and $C^\infty(M\setminus L)$ coincide on $J^m(M,L)$.
\end{cor}

\begin{cor}\label{c: Cinftyc(M setminus L) is dense in J(M L)}
For $m'<m$, $\Cinftyc(M\setminus L)$ is dense in $J^m(M,L)$ with the topology of $J^{m'}(M,L)$. Therefore $\Cinftyc(M\setminus L)$ is dense in $J(M,L)$.
\end{cor}

\begin{cor}\label{c: J(M L) is acyclic and Montel}
$J(M,L)$ is an acyclic Montel space, and therefore complete, boundedly retractive and reflexive.
\end{cor}

The analog of \Cref{r: Cinftyc(mathring M) is dense in AA(M)} makes sense for $J(M,L)$.

\subsection{A description of $J(M,L)$}\label{ss: description of J(M L)}

Take a b-metric $\bfg$ on $\bfM$ satisfying~\ref{i: g is of bounded geometry} and~\ref{i: A'}, and consider its restriction to $\mathring\bfM$. Consider also the boundary defining function $\bfx$ of $\bfM$ (\Cref{ss: J(M L)}). \Cref{c: AA^m(M) equiv x^m Hb^infty(M),c: Cinftyc(mathring M) is dense in AA^m(M),c: AA(M) equiv bigcup_m x^m Hb^infty(M)} and~\eqref{bfpi_*: AA(bfM) cong J(M L)} have the following direct consequences.

\begin{cor}\label{c: J^m(M L) equiv bfx^m Hb^infty(bfM)}
$J^m(M,L)\cong\bfx^m\Hb^\infty(\bfM)\equiv \bfx^{m+1/2}H^\infty(\mathring\bfM)$ $(m\in\R)$.
\end{cor}

\begin{cor}\label{c: Cinftyc(M setminus L) is dense in J^m(M L)}
$\Cinftyc(M\setminus L)$ is dense in every $J^m(M,L)$ and $J^{(s)}(M,L)$.
\end{cor}

\begin{cor}\label{c: J(M L) cong bigcup_m bfx^m Hb^infty(bfM)}
$J(M,L)\cong\bigcup_m\bfx^m\Hb^\infty(\bfM)=\bigcup_m\bfx^mH^\infty(\mathring\bfM)$.
\end{cor}

The analog of \Cref{r: Cinftyc(mathring M) is dense in AA^m(M)} makes sense for $J(M,L)$.

\subsection{$I(M,L)$ vs $\dot\AA(\bfM)$ and $J(M,L)$}\label{ss: I(M,L) vs dot AA(bfM)$ and J(M L)}

According to \Cref{ss: pull-back of conormal distribs,ss: push-forward of conormal distribs}, we have the continuous linear maps
\begin{equation}\label{tilde pi^* - tilde pi_*}
\tilde\pi^*:I(M,L)\to I(\widetilde M,\widetilde L)\;,\quad
\tilde\pi_*:I(\widetilde M,\widetilde L)\to I(M,L)\;,
\end{equation}
which are restrictions of the maps~\eqref{tilde pi^*: C^-infty(M) -> C^-infty(widetilde M)}, and therefore they satisfy~\eqref{tilde pi_* tilde pi^* = 2}. These maps are compatible with the Sobolev and symbol order filtrations because $\tilde\pi:\widetilde M\to M$ is a covering map (\Cref{ss: pull-back of conormal distribs,ss: push-forward of conormal distribs}).

Using the TVS-embedding $\dot\AA(\bfM)\subset I(\widetilde M,\widetilde L)$ (\Cref{c: dot AA(M) cong I_M(breve M partial M)}), which is compatible with the Sobolev and symbol order filtrations, the map $\tilde\pi_*$ of~\eqref{tilde pi^* - tilde pi_*} has the restriction
\begin{equation}\label{bfpi_*: dot AA(bfM) -> I(M L)}
\bfpi_*:\dot\AA(\bfM)\to I(M,L)\;.
\end{equation}
This map is compatible with the Sobolev and symbol order filtration by the above properties.

On the other hand, the map~\eqref{R :C^-infty(M) -> C^-infty(M L)} restricts to a continuous linear map
\begin{equation}\label{R: I(M L) -> J(M L)}
R:I(M,L)\to J(M,L)\;,
\end{equation}
which can be also described as the composition
\[
I(M,L) \xrightarrow{\tilde\pi^*} I(\widetilde M,\widetilde L) \xrightarrow{R} \AA(\bfM) \xrightarrow{\bfpi_*} J(M,L)\;.
\]
From the properties of~\eqref{R: I(breve M partial M) -> AA(M)},~\eqref{bfpi_*: AA(bfM) cong J(M L)} and~\eqref{tilde pi^* - tilde pi_*}, it follows that~\eqref{R: I(M L) -> J(M L)} is compatible with the Sobolev order filtration. According to~\eqref{C^infty(M) subset I^(infty)(M L)} and~\eqref{C^infty(M) subset J^(infty)(M L)}, the map~\eqref{R: I(M L) -> J(M L)} is the identity on $C^\infty(M)$.

\subsection{The space $K(M,L)$}\label{ss: K(M L)}

Like in \Cref{ss: KK(M)}, the condition of being supported in $L$ defines the LCHSs and $C^\infty(M)$-modules \index{$K^{(s)}(M,L)$} \index{$K^m(M,L)$} \index{$K(M,L)$}
\[
K^{(s)}(M,L)=I^{(s)}_L(M,L)\;,\quad K^m(M,L)=I^m_L(M,L)\;,\quad K(M,L)=I_L(M,L)\;.
\]
These are closed subspaces of $I^{(s)}(M,L)$, $I^m_L(M,L)$ and $I(M,L)$, respectively; more precisely, they are the null spaces of the corresponding restrictions of the map~\eqref{R: I(M L) -> J(M L)}. According to \Cref{c: dot AA(M) cong I_M(breve M partial M)}, the identity~\eqref{dot C^-infty_partial bfM(bfM) equiv C^-infty_L(M) oplus C^-infty_L(M)} restricts to a TVS-identity
\begin{equation}\label{KK(bfM) equiv K(M L) oplus K(M L)}
\KK(\bfM)\equiv K(M,L)\oplus K(M,L)\;.
\end{equation}
Furthermore the maps~\eqref{bfpi_*: dot C^-infty_partial bfM(bfM) -> C^-infty_L(M)} induce continuous linear maps
\begin{equation}\label{bfpi_*: KK(bfM) -> K(M L)}
\bfpi_*:\KK(\bfM)\to K(M,L)\;,\quad
\tilde\pi^*:K(M,L)\to\KK(\bfM)\;.
\end{equation}
Using~\eqref{KK(bfM) equiv K(M L) oplus K(M L)}, these maps are given by $\bfpi_*(u,v)=u+v$ and $\tilde\pi^*u=(u,u)$.

By~\eqref{dot AA^(s)(M) equiv I^(s)_M(breve M partial M)} and~\eqref{dot AA^m(M) = I^m_M(breve M partial M)}, $K^{(s)}(M,L)$ and $K^m(M,L)$ satisfy analogs of~\eqref{KK(bfM) equiv K(M L) oplus K(M L)}, using $\KK^{(s)}(\bfM)$ and $\KK^m(\bfM)$. Thus we get the following consequences of \Cref{p: KK(M) = bigcup_s KK^(s)(M),p: KK^(s)(M) is a totally reflexive Frechet sp,c: KK(M) is barreled,c: coincidence of tops on KK^m(M),c: KK(M) is acyclic and Montel}.

\begin{cor}\label{c: K(M L) = bigcup_s K^{(s)}(M L)}
$K(M,L)$ is a limit subspace of the LF-space $I(M,L)$.
\end{cor}

\begin{cor}\label{c: K^(s)(M L) is a totally reflexive Frechet sp}
$K^{(s)}(M,L)$ is a totally reflexive Fr\'echet space.
\end{cor}

\begin{cor}\label{c: K(M L) is barreled}
$K^{(s)}(M,L)$ is barreled, ultrabornological and webbed, and therefore so is $K(M,L)$.
\end{cor}

\begin{cor}\label{c: coincidence of tops on K^m(M L)}
For $m<m',m''$, the topologies of $K^{m'}(M,L)$ and $K^{m''}(M,L)$ coincide on $K^m(M,L)$.
\end{cor}

\begin{cor}\label{c: K(M L) is acyclic and Montel}
$K(M,L)$ is an acyclic Montel space, and therefore complete, boundedly retractive and reflexive.
\end{cor}

\begin{ex}\label{ex: Diff(M)}
With the notation of \Cref{ss: pseudodiff ops}, $\Diff(M)\equiv K(M^2,\Delta)$ becomes a filtered $C^\infty(M^2)$-submodule of $\Psi(M)$, with the order filtration corresponding to the symbol order filtration. In this way, $\Diff(M)$ also becomes a LCHS satisfying the above properties. If $M$ is compact, it is also a filtered subalgebra of $\Psi(M)$.
\end{ex}

\subsection{A description of $K(M,L)$}\label{ss: description of K(M L)}

By~\eqref{u mapsto delta_L^u conormal} and~\eqref{A: I^(s)(M L E) -> I^[s-k](M L E)},
\begin{equation}\label{partial_x^m C^infty(L Omega^-1 NL) subset K^(s-m)(M L)}
\partial_x^mC^\infty(L;\Omega^{-1}NL)\subset K^{(s-m)}(M,L)\quad(s<-1/2)\;,
\end{equation}
and~\eqref{C^-infty(L Omega^-1 NL) cong partial_x^m C^-infty(L Omega^-1 NL)} restricts to a continuous linear isomorphism
\begin{equation}\label{C^infty(L Omega^-1 NL) cong partial_x^m C^infty(L Omega^-1 NL)}
\partial_x^m:C^\infty(L;\Omega^{-1}NL)\xrightarrow{\cong}\partial_x^mC^\infty(L;\Omega^{-1}NL)\;.
\end{equation}

\begin{lem}\label{l: partial_x^m C^infty(L Omega^-1 NL) cap I^[-1/2-m]_c(M L) = 0}
For all $m\in\N_0$,
\[
\partial_x^mC^\infty(L;\Omega^{-1}NL)\cap K^{(-\frac12-m)}(M,L)=0\;.
\]
\end{lem}

\begin{proof}
We proceed by induction on $m$. The case $m=0$ is given by \Cref{p: C^infty_./c(L ...) subset H^s_./c(M E)}. Now assume $m\ge1$, and let $v\in C^\infty(L;\Omega^{-1}NL)$ with $u=\partial_x^m\delta_L^v\in K^{(-\frac12-m)}(M,L)$. Take any $A\in\Diff^2(L;\Omega NL)$ such that $-\partial_x^2+B\in\Diff^2(T)$ is elliptic, where $B=(1\otimes A^t)^t\in\Diff^2(T,L)$; for instance, given a Riemannian metric on $M$, $A$ can be the Laplacian of the flat line bundle $\Omega NL$. By~\eqref{partial_x^m C^infty(L Omega^-1 NL) subset K^(s-m)(M L)}, $u_0:=\partial_x^{m-1}\delta_L^v\in K^{(\frac12-m-\epsilon)}(M,L)$ for $0<\epsilon<1$. By~\eqref{[partial_x Diff^k(T L)] subset ...}, given any $B_0\in\Diff(M,L)$, there is some $B_1,B_2,B_3\in\Diff(M,L)$ such that $[\partial_x^2,B_0]=B_1+B_2\partial_x+B_3\partial_x^2$. So, according to \Cref{ss: diff opers on conormal distribs},~\eqref{A delta_L^u} and~\eqref{[partial_x Diff^k(T L)] subset ...}, for all $B_0\in\Diff(M,L)$,
\begin{align*}
(-\partial_x^2+B)B_0u_0&=-B_0\partial_xu-B_1u_0-B_2u-B_3\partial_xu+\partial_x^{m-1}\delta_L^{B_0'Av}+[B,B_0]u_0\\
&\in K^{(-\frac32-m)}(M,L)+K^{(\frac12-m-\epsilon)}(M,L)=K^{(-\frac32-m)}(M,L)\;.
\end{align*}
Hence $B_0u_0\in H^{\frac12-m}(M)$ by elliptic regularity. Since $B_0$ is arbitrary, we get $u_0\in K^{(\frac12-m)}(M,L)$. So $u_0=0$ by the induction hypothesis, yielding $u=\partial_xu_0=0$.
\end{proof}

By \Cref{p: bigoplus_m C^0_m -> C^-infty_L(M)}, the TVS-isomorphism~\eqref{bigoplus_m C^0_m -> C^-infty_L(M)} restricts to a linear injection
\begin{equation}\label{bigoplus_m C^1_m -> K(M L)}
\bigoplus_{m=0}^\infty C^1_m\to K(M,L)\;,
\end{equation}
where $C^1_m=C^\infty(L;\Omega^{-1}NL)$ for all $m\in\N_0$, which is easily seen to be continuous.

\begin{prop}\label{p: bigoplus_m C^1_m -> K(M L)}
The map~\eqref{bigoplus_m C^1_m -> K(M L)} is a TVS-isomorphism, which induces TVS-isomorphisms
\begin{equation}\label{bigoplus_m<-s-1/2 C^1_m cong K^(s)(M L)}
\bigoplus_{m<-s-\frac12}C^1_m\xrightarrow{\cong}K^{(s)}(M,L)\quad(s<-1/2)\;.
\end{equation}
\end{prop}

\begin{proof}
To prove that~\eqref{bigoplus_m C^1_m -> K(M L)} is surjective, take any $u\in K(M,L)$. By \Cref{p: bigoplus_m C^0_m -> C^-infty_L(M)}, we can assume $u\in \partial_x^mC^{-\infty}(L;\Omega NL)$ for some $m$. For any $A\in\Diff(L;\Omega NL)$, let $B=(1\otimes A^t)^t\in\Diff(T,L)$. Since $u\in K(M,L)$ and $B$ is local, it follows from the definition of $I(M,L)$ and~\eqref{H^-s(M) supset C^prime -k(M) supset H^-k(M)} that $Bu\in H^{-k}_L(T)\subset C^{\prime\,-k}_L(T)$ for some $k\ge m$. On the other hand, $u=\partial_x^m\delta_L^v$ for some $v\in C^{-\infty}(L;\Omega NL)$. Then~\eqref{A delta_L^u} and~\eqref{[partial_x 1 otimes Diff(L)] = 0} yield
\[
Bu=B\partial_x^m\delta_L^v=\partial_x^mB\delta_L^v=\partial_x^m\delta_L^{B'v}
=\partial_x^m\delta_L^{Av}\;.
\]
Therefore, by \Cref{p: bigoplus_m C^0_m -> C^-infty_L(M)},
\[
Bu\in C^{\prime\,-k}_L(M)\cap\partial_x^mC^{-\infty}(L;\Omega NL)=\partial_x^mC^{\prime\,m-k}(L;\Omega NL)\;.
\]
This means that $Av\in C^{\prime\,m-k}(L;\Omega NL)$. So $v\in C^\infty(L;\Omega NL)$ because $A$ is arbitrary, and therefore $u\in\partial_x^mC^\infty(L;\Omega NL)$.

The surjectivity of~\eqref{bigoplus_m<-s-1/2 C^1_m cong K^(s)(M L)} follows from \Cref{l: partial_x^m C^infty(L Omega^-1 NL) cap I^[-1/2-m]_c(M L) = 0} and the surjectivity of~\eqref{bigoplus_m C^1_m -> K(M L)}.

Finally,~\eqref{bigoplus_m C^1_m -> K(M L)} is open like in \Cref{p: bigoplus_m C^0_m -> C^-infty_L(M)}, using that $C^\infty(L;\Omega^{-1}NL)$ is webbed and $K(M,L)$ ultrabornological (\Cref{c: K(M L) is barreled}). So~\eqref{bigoplus_m<-s-1/2 C^1_m cong K^(s)(M L)} is also open.
\end{proof}

\begin{rem}\label{r: bigoplus_m C^1_m -> K(M L)}
\Cref{p: bigoplus_m C^1_m -> K(M L)} reconfirms \Cref{c: K(M L) = bigcup_s K^{(s)}(M L)}. It also follows from \Cref{p: bigoplus_m C^1_m -> K(M L)} that~\eqref{C^infty(L Omega^-1 NL) cong partial_x^m C^infty(L Omega^-1 NL)} is a TVS-isomorphism.
\end{rem}

\begin{rem}\label{r: description of KK(M)}
In \Cref{ss: KK(M)}, for any compact manifold with boundary $M$, the analog of \Cref{p: bigoplus_m C^0_m -> C^-infty_L(M)} for $\KK(M)$ follows from~\eqref{KK(M) cong I_partial M(breve M partial M)} and the application of \Cref{p: bigoplus_m C^0_m -> C^-infty_L(M)} to $K(\breve M,\partial M)$.
\end{rem}

\subsection{The conormal sequence}\label{conormal sequence}

The diagram~\eqref{CD: ... dot C^-infty(bfM) -R-> C^-infty(bfM) -> 0} has the restriction
\begin{equation}\label{CD: conormal seqs}
\begin{CD}
0\to\KK(\bfM) @>\iota>> \dot\AA(\bfM) @>R>> \AA(\bfM)\to0 \\
@V{\bfpi_*}VV @V{\bfpi_*}VV @V{\cong}V{\bfpi_*}V \\
0\to K(M,L) @>\iota>> I(M,L) @>R>> J(M,L)\to0\;.\hspace{-.2cm}
\end{CD}
\end{equation}
The bottom row of~\eqref{CD: conormal seqs} will be called the \emph{conormal sequence} of $M$ at $L$ (or of $(M,L)$). The following analog of \Cref{p: bfpi_*: dot C^-infty(bfM) -> C^-infty(M) is surj top hom} holds true with formally the same proof, using~\eqref{CD: conormal seqs}.

\begin{prop}\label{p: R: I(M L) -> J(M L) is a surj top hom}
The maps~\eqref{bfpi_*: dot AA(bfM) -> I(M L)} and~\eqref{R: I(M L) -> J(M L)} are surjective topological homomorphisms.
\end{prop}

\begin{cor}\label{c: conormal sequence of (M L) is exact}
The conormal sequence of $M$ at $L$ is exact in the category of continuous linear maps between LCSs.
\end{cor}

The surjectivity of~\eqref{R: I(M L) -> J(M L)} can be realized with partial extension maps, as stated in the following analog of \Cref{p: E_m}.

\begin{cor}\label{c: E_m: J^m(M L) -> I^(s)(M L)}
For all $m\in\R$, there is a continuous linear partial extension map $E_m:J^m(M,L)\to I^{(s)}(M,L)$, where $s=0$ if $m\ge0$, and $m>s\in\Z^-$ if $m<0$.
\end{cor}

\begin{proof}
By the commutativity of~\eqref{CD: conormal seqs}, we can take $E_m$ equal to the composition
\[
J^m(M,L) \xrightarrow{\bfpi_*^{-1}} \AA^m(\bfM) \xrightarrow{E_m} \dot\AA^{(s)}(\bfM) 
\xrightarrow{\bfpi_*} I^{(s)}(M,L)\;,
\]
where this map $E_m$ is given by \Cref{p: E_m}.
\end{proof}

\begin{cor}\label{c: C^infty(M) is dense in J(M L)}
$C^\infty(M)$ is dense in $J(M,L)$.
\end{cor}

\begin{proof}
Apply~\eqref{C^infty(M) subset J^(infty)(M L)}, \Cref{c: C^infty(M) is dense in I(M L),p: R: I(M L) -> J(M L) is a surj top hom}.
\end{proof}

\subsection{Action of $\Diff(M)$ on the conormal sequence}\label{ss: Diff(M) on the conormal seq}

According to \Cref{ss: diff opers on conormal distribs}, every $A\in\Diff(M)$ defines a continuous linear map $A$ on $I(M,L)$, which preserves $K(M,L)$ because $A$ is local. Therefore it induces a linear map $A$ on $J(M,L)$, which is continuous by \Cref{p: R: I(M L) -> J(M L) is a surj top hom}. This map satisfies the analog of~\eqref{A: I^(s)(M L E) -> I^[s-k](M L E)}.

The map $A$ on $J(M,L)$ can be also described as a restriction of $A$ on $C^{-\infty}(M,L)$ (\Cref{ss: action of Diff(M) on C^-infty(M L) and C^-infty_L(M)}). On the other hand,  according to \Cref{ss: diff ops on dot AA(M) and AA(M)}, the lift $\widetilde A\in\Diff(\bfM)$ defines continuous linear maps on the top spaces of~\eqref{CD: conormal seqs} which correspond to the operators defined by $A$ on the bottom spaces via the maps $\bfpi_*$. If $A\in\Diff(M,L)$, then it defines continuous endomorphisms $A$ of $J^{(s)}(M,L)$ and $J^m(M,L)$.

\subsection{Pull-back of elements of the conormal sequence}\label{ss: pull-back of the conormal seq}

Consider the notation and conditions of \Cref{ss: pull-back of conormal distribs}. By locality, the map~\eqref{phi^*: I(M L) -> I(M' L')} has a restriction $\phi^*:K(M,L)\to K(M',L')$. So it also induces a linear map $\phi^*:J(M,L)\to J(M',L')$,
which is continuous by \Cref{p: R: I(M L) -> J(M L) is a surj top hom}.

\subsection{Push-forward of elements of the conormal sequence}\label{ss: push-forward of the conormal seq}

Consider the notation and conditions of \Cref{ss: push-forward of conormal distribs}. As above, the map~\eqref{phi_*: I_c(M' L' Omega_fiber) -> I_c(M L)} has a restriction $\phi_*:K_{\text{\rm c}}(M',L';\Omega_{\text{\rm fiber}})\to K_{\text{\rm c}}(M,L)$. Thus it induces a linear map $\phi_*:J_{\text{\rm c}}(M',L';\Omega_{\text{\rm fiber}})\to J_{\text{\rm c}}(M,L)$, which is continuous by \Cref{p: R: I(M L) -> J(M L) is a surj top hom}.

\subsection{Case where $L$ is not transversely orientable}\label{ss: case where L is not transversely orientable}

If $L$ is not transversely orientable, we still have a tubular neighborhood $T$ of $L$ in $M$, but there is no defining function $x$ of $L$ in $T$ trivializing the projection $\varpi:T\to L$. We can cut $M$ along $L$ as well to produce a bounded compact manifold, $\bfM$, with a projection $\bfpi:\bfM\to M$ and a boundary collar $\bfT$ over $T$.

By using a boundary defining function $\bfx$ of $\bfM$, we get the same definitions, properties and descriptions of $C^{\pm\infty}(M,L)$ and $J(M,L)$ (\Cref{ss: C^-infty(M L),ss: J(M L),ss: description of J(M L)}).

$C^{-\infty}_L(M)$ and $K(M,L)$ also have the same definitions (\Cref{ss: C^-infty_L(M),ss: K(M L)}). However~\eqref{dot C^-infty_partial bfM(bfM) equiv C^-infty_L(M) oplus C^-infty_L(M)} and~\eqref{KK(bfM) equiv K(M L) oplus K(M L)} are not true because the covering map $\bfpi:\partial\bfM\to L$ is not trivial, and the descriptions given in \Cref{p: bigoplus_m C^0_m -> C^-infty_L(M),p: bigoplus_m C^1_m -> K(M L)} need a slight modification. This problem can be solved as follows.

Let $\check\pi:\check L\to L$ \index{$\check\pi$} \index{$\check L$} denote the two-fold covering of transverse orientations of $L$, and let $\check\sigma$ \index{$\check\sigma$} denote its deck transformation different from the identity. Since the lift of $NL$ to $\check L$ is trivial, $\check\pi$ on $\check L\equiv\{0\}\times\check L$ can be extended to a two-fold covering $\check\pi:\check T:=(-\epsilon,\epsilon)_x\times\check L\to T$, \index{$\check T$} for some $\epsilon>0$. Its deck transformation different from the identity is an extension of $\check\sigma$ on $\check L\equiv\{0\}\times\check L$, also denoted by $\check\sigma$. Then $\check L$ is transversely oriented in $\check T$; i.e., its normal bundle $N\check L$ is trivial. Thus $C^{-\infty}_{\check L}(\check T)$ and $K(\check T,\check L)$ satisfy~\eqref{dot C^-infty_partial bfM(bfM) equiv C^-infty_L(M) oplus C^-infty_L(M)},~\eqref{KK(bfM) equiv K(M L) oplus K(M L)} and \Cref{p: bigoplus_m C^0_m -> C^-infty_L(M),p: bigoplus_m C^1_m -> K(M L)}. Since $N\check L\equiv\check\pi^*NL$, the map $\check\sigma$ lifts to a homomorphism of $N\check L$, which induces a homomorphism of $\Omega^{-1}NL$ also denoted by $\check\sigma$. Let $L_{-1}$ be the union of non-transversely oriented connected components of $L$, and $L_1$ the union of its transversely oriented components. Correspondingly, let $\check L_{\pm1}=\check\pi^{-1}(L_{\pm1})$ and $\check T_{\pm1}=(-\epsilon,\epsilon)\times\check L_{\pm1}$. Since $\check\sigma^*x=\pm x$ on $T_{\pm1}$, \Cref{p: bigoplus_m C^0_m -> C^-infty_L(M),p: bigoplus_m C^1_m -> K(M L)} become true in this case by replacing $C^r(L;\Omega^{-1}NL)$ ($r\in\Z\cup\{\pm\infty\}$) with the direct sum of the spaces
\[
\{\,u\in C^r(\check L_{\pm1};\Omega^{-1}N\check L_{\pm1})\mid\check\sigma^*u=\pm u\,\}\;.
\]

Now the other results about $C^{-\infty}_L(M)$ and $K(M,L)$, indicated in \Cref{ss: C^-infty_L(M),ss: description of C^-infty_L(M),ss: K(M L),ss: description of K(M L)}, can be obtained by using these extensions of \Cref{p: bigoplus_m C^0_m -> C^-infty_L(M),p: bigoplus_m C^1_m -> K(M L)} instead of~\eqref{dot C^-infty_partial bfM(bfM) equiv C^-infty_L(M) oplus C^-infty_L(M)} and~\eqref{KK(bfM) equiv K(M L) oplus K(M L)}. \Cref{conormal sequence,ss: Diff(M) on the conormal seq,ss: pull-back of the conormal seq,ss: push-forward of the conormal seq} also have straightforward extensions.

\section{Dual-conormal sequence}\label{s: dual-conormal seq}

\subsection{The spaces $K'(M,L)$ and $J'(M,L)$}\label{ss: K'(M,L) and J'(M,L)}

Consider the notation of \Cref{s: conormal seq} assuming that $L$ is transversely oriented; the extension to the non-transversely orientable case can be made like in \Cref{ss: case where L is not transversely orientable}. Like in \Cref{ss: dual-conormal distribs,ss: AA'(M)}, let \index{$K'(M,L)$} \index{$J'(M,L)$} 
\[
K'(M,L)=K(M,L;\Omega)'\;,\quad J'(M,L)=J(M,L;\Omega)'\;.
\]
By~\eqref{bfpi_*: AA(bfM) cong J(M L)} and~\eqref{KK(bfM) equiv K(M L) oplus K(M L)},
\begin{equation}\label{dot AA'(bfM) equiv J'(M L)}
\KK'(\bfM)\equiv K'(M,L)\oplus K'(M,L)\;,\quad\dot\AA'(\bfM)\equiv J'(M,L)\;.
\end{equation}
Let also \index{$K^{\prime\,(s)}(M,L)$} \index{$K^{\prime\,m}(M,L)$} \index{$J^{\prime\,(s)}(M,L)$} \index{$J^{\prime\,m}(M,L)$}
\begin{equation}\label{K^prime[s](M L)}
\left\{
\begin{gathered}
K^{\prime\,(s)}(M,L)=K^{(-s)}(M,L;\Omega)'\;,\quad K^{\prime\,m}(M,L)=K^{-m}(M,L;\Omega)'\;,\\
J^{\prime\,(s)}(M,L)=J^{(-s)}(M,L;\Omega)'\;,\quad J^{\prime\,m}(M,L)=J^{-m}(M,L;\Omega)'\;,
\end{gathered}
\right.
\end{equation}
which satisfy the analogs of~\eqref{dot AA'(bfM) equiv J'(M L)}. Like in \Cref{ss: AA'(M)}, for $s<s'$ and $m<m'$, we get continuous linear restriction maps
\[
K^{\prime\,(s')}(M,L)\to K^{\prime\,(s)}(M,L)\;,\quad K^{\prime\,m}(M,L)\to K^{\prime\,m'}(M,L)\;,
\]
and continuous injections
\[
J^{\prime\,(s')}(M,L)\subset J^{\prime\,(s)}(M,L)\;,\quad J^{\prime\,m'}(M,L)\subset J^{\prime\,m}(M,L)\;,
\]
forming projective spectra. By~\eqref{dot AA'(bfM) equiv J'(M L)}, its analogs for the spaces~\eqref{K^prime[s](M L)} and according to \Cref{ss: AA'(M)}, the spaces $K^{\prime\,(s)}(M,L)$ and $K^{\prime\,m}(M,L)$ satisfy the analogs of~\eqref{sandwich for I'} and~\eqref{varprojlim I'^(s)(M L) equiv varprojlim I'^m(M L)}, and the spaces $J^{\prime\,(s)}(M,L)$ and $J^{\prime\,m}(M,L)$ satisfy the analogs of~\eqref{sandwich for ZZ} and~\eqref{bigcap_s ZZ^(s)(M) = bigcap_m ZZ^m(M)}. Using~\eqref{dot AA'(bfM) equiv J'(M L)}, we get the following consequences of \Cref{p: KK'(M) AA'(M) dot AA'(M) are complete and Montel,c: KK^prime [s](M) AA^prime [s](M) dot AA^prime [s](M) are bornological,c: AA'(M) equiv varprojlim AA^prime [s](M)}.

\begin{cor}\label{c: K'(M L) and J'(M L) are complete and Montel}
$K'(M,L)$ and $J'(M,L)$ are complete Montel spaces.
\end{cor}

\begin{cor}\label{c: K^prime [s](M L) J^prime [s](M L) are bornological}
$K^{\prime\,(s)}(M,L)$ and $J^{\prime\,(s)}(M,L)$ are bornological and barreled.
\end{cor}

\begin{cor}\label{c: J'(M L) equiv varprojlim J^prime [s](M L)}
$K'(M,L)\equiv\varprojlim K^{\prime\,(s)}(M,L)$ and $J'(M,L)=\bigcap_sJ^{\prime\,(s)}(M,L)$.
\end{cor}

Like in \Cref{ss: AA'(M)}, the versions of~\eqref{C^infty(M L) = C^-infty(M L Omega)'},~\eqref{C^infty(M) subset J^(infty)(M L)} and \Cref{c: Cinftyc(M setminus L) is dense in J(M L)} with $\Omega M$ induce continuous inclusions
\begin{equation}\label{C^infty(M L) subset J'(M L) subset C^-infty(M)}
C^{-\infty}(M)\supset J'(M,L)\supset C^\infty(M,L)\;.
\end{equation}

\subsection{A description of $J'(M,L)$}\label{ss: description of J'(M L)}

With the notation and conditions of \Cref{ss: description of J(M L)}, the identity~\eqref{dot AA'(bfM) equiv J'(M L)} and \Cref{c: dot AA^prime m(M) equiv x^m Hb^-infty(M),c: dot AA'(M) equiv bigcap_m x^m Hb^-infty(M),c: Cinftyc(mathring M) is dense in dot AA'(M)} have the following consequences.

\begin{cor}\label{c: J^prime m(M L) cong bfx^m Hb^-infty(bfM)}
$J^{\prime\,m}(M,L)\cong\bfx^m\Hb^{-\infty}(\bfM)=\bfx^{m-\frac12}H^{-\infty}(\mathring\bfM)$ $(m\in\R)$.
\end{cor}

\begin{cor}\label{c: J'(M L) cong bigcap_m bfx^m Hb^-infty(bfM)}
$J'(M,L)\cong\bigcap_m\bfx^m\Hb^{-\infty}(\bfM)=\bigcap_m\bfx^mH^{-\infty}(\mathring\bfM)$.
\end{cor}

\begin{cor}\label{c: Cinftyc(M setminus L) is dense in J'(M L)}
$\Cinftyc(M\setminus L)$ is dense in every $J^{\prime\,m}(M,L)$ and in $J'(M,L)$. Therefore the right-hand side inclusion of~\eqref{C^infty(M L) subset J'(M L) subset C^-infty(M)} is also dense.
\end{cor}

The analog of \Cref{r: Cinftyc(mathring M) is dense in dot AA'(M)} makes sense for $J'(M,L)$.

\subsection{Description of $K'(M,L)$}\label{ss: description of K'(M L)}

The version of \Cref{p: bigoplus_m C^1_m -> K(M L)} with $\Omega M$ has the following direct consequence, where we set
\[
C^2_m=C^{\infty}(L;\Omega^{-1}NL\otimes\Omega M)'=C^{\infty}(L;\Omega)'=C^{-\infty}(L)
\]
for every $m\in\N_0$.

\begin{cor}\label{p: K'(M L) cong prod_m C^2_m}
The transposes of the versions of~\eqref{bigoplus_m C^1_m -> K(M L)} and~\eqref{bigoplus_m<-s-1/2 C^1_m cong K^(s)(M L)} with $\Omega M$ are TVS-isomorphisms,
\[
K'(M,L)\xrightarrow{\cong}\prod_{m=0}^\infty C^2_m\;,\quad
K^{\prime\,(s)}(M,L)\xrightarrow{\cong}\prod_{m<s-1/2}C^2_m\quad(s>1/2)\;.
\]
\end{cor}

\subsection{Dual-conormal sequence}\label{ss: dual-conormal seq}

Transposing the density-bundles version of~\eqref{CD: conormal seqs}, we get the commutative diagram
\begin{equation}\label{CD: dual-conormal seqs}
\begin{CD}
0\leftarrow\KK'(\bfM) @<{\iota^t}<< \AA'(\bfM) @<{R^t}<< \dot\AA'(\bfM)\leftarrow0 \\
@A{\bfpi^*}AA @A{\bfpi^*}AA @A{\bfpi^*}A{\cong}A \\
0\leftarrow K'(M,L) @<{\iota^t}<< I'(M,L) @<{R^t}<< J'(M,L)\leftarrow0\;.\hspace{-.2cm}
\end{CD}
\end{equation}
Its maps are compatible with the Sobolev and symbol order filtrations. Its bottom row will be called the \emph{dual-conormal sequence} of $M$ at $L$ (or of $(M,L)$). The following analog of \Cref{p: dual-conormal seq is exact} holds true with formally the same proof, using \Cref{p: R: I(M L) -> J(M L) is a surj top hom,c: coincidence of tops on J^m(M L),c: J(M L) is acyclic and Montel,c: E_m: J^m(M L) -> I^(s)(M L)}.

\begin{prop}\label{p: dual-conormal seq of (M L) is exact}
The dual-conormal sequence of $M$ at $L$ is exact in the category of continuous linear maps between LCSs.
\end{prop}

\subsection{Action of $\Diff(M)$ on the dual-conormal sequence}\label{ss: action of Diff(M) on the dual-conormal seq}

With the notation of \Cref{ss: Diff(M) on the conormal seq}, consider the actions of $A^t$ and $\widetilde A^t$ on the bottom and top spaces of the version of~\eqref{CD: conormal seqs} with $\Omega M$ and $\Omega\bfM$. Taking transposes again, we get induced actions of $A$ and $\widetilde A$ on the bottom and top spaces of~\eqref{CD: dual-conormal seqs}, which correspond one another via the linear maps $\bfpi^*$. These maps satisfy the analogs of~\eqref{A: I^prime [s](M L E) -> I^prime (s-m)(M L E)}.

\subsection{Pull-back of elements of the dual-conormal sequence}
\label{ss: pull-back of the dual-conormal seq}

With the notation and conditions of \Cref{ss: pull-back of dual-conormal distributions}, besides~\eqref{phi^*: I'(M L) -> I'(M' L')}, we get continuous linear pull-back maps $\phi^*:K'(M,L)\to K'(M',L')$ and $\phi^*:J'(M,L)\to J'(M',L')$.

\subsection{Push-forward of elements of the dual-conormal sequence}
\label{ss: push-forward of the dual-conormal seq}

With the notation and conditions of \Cref{ss: push-forward of dual-conormal distributions}, besides~\eqref{phi_*: I'_c(M' L' Omega_fiber) -> I'_c(M L)}, we get continuous linear push-forward maps $\phi_*:K'_{\text{\rm c}}(M',L';\Omega_{\text{\rm fiber}})\to K'_{\text{\rm c}}(M,L)$ and $\phi_*:J'_{\text{\rm c}}(M',L';\Omega_{\text{\rm fiber}})\to J'_{\text{\rm c}}(M,L)$.

\subsection{$I(M,L)$ vs $I'(M,L)$}\label{ss: I(M L) vs I'(M L)}

\begin{lem}\label{l: bfpi^*(H^-m(M) cap I'(M L)) subset dot H^-m(bfM) cap BB(bfM)}
For all $m\in\N_0$, $\bfpi^*\big(H^{-m}(M)\cap I'(M,L)\big)\subset\dot H^{-m}(\bfM)\cap\AA'(\bfM)$.
\end{lem}

\begin{proof}
Using a volume form on $M$ and its lift to $\bfM$ to define a scalar product of $L^2(M)$ and $L^2(\bfM)$, it follows that $\bfpi^*:C^\infty(M)\to C^\infty(\bfM)$ induces a unitary isomorphism $\bfpi^*:L^2(M)\to L^2(\bfM)$. Hence the statement is true for $m=0$ because $L^2(\bfM)\equiv\dot H^0(\bfM)$ (\Cref{ss: supported and extendible Sobolev sps}). Then, for arbitrary $m$, by~\eqref{H^-s(M) = ...} and~\eqref{A:H^s(M) -> H^s-m(M)}, and according to \Cref{ss: cutting},
\begin{multline*}
\bfpi^*\big(H^{-m}(M)\cap I'(M,L)\big)\\
\begin{aligned}
&=\bfpi^*\big(\Diff^m(M)\,L^2(M)\cap I'(M,L)\big)
\subset\Diff^m(\bfM)\,\bfpi^*L^2(M)\cap\AA'(\bfM)\\
&=\Diff^m(\bfM)\,\dot H^0(\bfM)\cap\AA'(\bfM)
\subset\dot H^{-m}(\bfM)\cap\AA'(\bfM)\;.\qedhere
\end{aligned}
\end{multline*}
\end{proof}

\begin{lem}\label{l: bfpi^*(I(M L) cap I'(M L)) subset C^infty(bfM)}
$\bfpi^*\big(I(M,L)\cap I'(M,L)\big)\subset C^\infty(\bfM)$.
\end{lem}

\begin{proof}
For every $u\in I(M,L)\cap I'(M,L)$, there is some $m\in\N_0$ such that $u\in I^{(-m)}(M,L)$. Then, by \Cref{l: bfpi^*(H^-m(M) cap I'(M L)) subset dot H^-m(bfM) cap BB(bfM)}, for any $B\in\Diff(M,L)$,
\[
\widetilde B\bfpi^*u=\bfpi^*Bu\in\bfpi^*\big(H^{-m}(M)\cap I'(M,L)\big)\subset\dot H^{-m}(\bfM)\cap\AA'(\bfM)\;.
\]
Since the operators $\widetilde B$ ($B\in\Diff(M,L)$) generate $\Diff(\widetilde M,\widetilde L)$ as $C^\infty(\widetilde M)$-module, it follows that $u\in\dot\AA(\bfM)\cap\AA'(\bfM)=C^\infty(\bfM)$ by~\eqref{ABC}.
\end{proof}

\begin{thm}\label{t: bfpi^*(I(M L) cap I'(M L)) subset C^infty(bfM)}
$I(M,L)\cap I'(M,L)=C^\infty(M)$.
\end{thm}

\begin{proof}
Suppose there is some non-smooth $u\in I(M,L)\cap I'(M,L)$. However $\bfpi^*u\in C^\infty(\bfM)$ by \Cref{l: bfpi^*(I(M L) cap I'(M L)) subset C^infty(bfM)}. Then there is a chart $(V,y)$ of $L$ such that, for the induced chart $(U\equiv(-\epsilon,\epsilon)\times V,(x,y))$ of $M$, the function $u$ is smooth on $((-\epsilon,0)\cup(0,\epsilon))\times V$, and has smooth extensions to $(-\epsilon,0]\times V$ and $[0,\epsilon)\times V$, but $\partial_x^mu(0^-,y_0)\ne\partial_x^m(0^+,y_0)$ for some $m\in\N_0$ and $y_0\in V$. After multiplying $u$ by a smooth function supported in $U$ whose value at $y_0$ is nonzero, we can assume $u$ is supported in $(-\epsilon/2,\epsilon/2)\times V$. Then there is some $v\in C^\infty(L;\Omega)$ such that $\supp v\subset V$ and
\begin{equation}\label{int_x in V (u(0^- x) - u(0^+ x)) v(x) ne 0}
\int_{y\in V}(u(0^-,y)-u(0^+,y))\,v(y)\ne0\;.
\end{equation}

On the other hand, there is a sequence $\phi_k\in\Cinftyc(-\epsilon,\epsilon)$ so that the restrictions of $m$th derivatives $\phi_k^{(m)}$ to $(-\epsilon/2,\epsilon/2)$ are compactly supported and converge to $\delta_0$ in $C^{-\infty}(-\epsilon/2,\epsilon/2)$ as $k\to\infty$. For instance, we may take
\[
\phi_k(t)=h(t)\int_0^t\int_0^{t_{m-1}}\cdots\int_0^{t_1}f_k(t_0)\,dt_0\cdots dt_{m-1}\;,
\]
where $h,f_k\in\Cinftyc(-\epsilon,\epsilon)$ with $h=1$ on $(-\epsilon/2,\epsilon/2)$, $\supp f_k\subset(-\epsilon/2,\epsilon/2)$, $f_k$ is even, and $f_k\to\delta_0$ in $C^{-\infty}_{\text{\rm c}}(-\epsilon/2,\epsilon/2)$ and $f_k(0)\to\infty$ as $k\to\infty$. Thus
\begin{gather}
\phi_k^{(m)}(0)=f_k(0)\to\infty\;,\label{phi_k^(m)(0) = f_k(0) -> infty}\\
\int_{-\infty}^0a(t)\phi_k^{(m)}(t)\,dt\to\frac{a(0)}2\;,\quad
\int_0^\infty b(t)\phi_k^{(m)}(t)\,dt\to\frac{b(0)}2\;,\label{int_-infty^0 a(t) phi_k^(m)(t) dt to a(0)/2}
\end{gather}
for all $a\in\Cinftyc(-\infty,0]$ and $b\in\Cinftyc[0,\infty)$.

The sequence $w_k\equiv\phi_k(x)\,v(y)\otimes|dx|\in\Cinftyc(T;\Omega)\subset C^\infty(M;\Omega)$ satisfies
\[
\partial_x^mw_k\equiv\phi_k^{(m)}(x)\,v(y)\otimes|dx|\to\delta_0(x)\,v(y)\otimes|dx|\equiv\delta_L^v
\]
in $I(M,L;\Omega)$ as $k\to\infty$. Since $u\in I'(M,L)$ and $\partial_x^{m+1}w_k\in I(M,L;\Omega)$, it follows that $\langle u,\partial_x^{m+1}w_k\rangle\to\langle u,\partial_x\delta_L^v\rangle$ as $k\to\infty$. But
\begin{align*}
\langle u,\partial_x^{m+1}w_k\rangle
&=\int_{y\in V}\int_{-\epsilon/2}^0u(x,y)\,\phi_k^{(m+1)}(x)\,v(y)\,dx\\
&\phantom{={}}{}+\int_{y\in V}\int_0^{\epsilon/2}u(x,y)\,\phi_k^{(m+1)}(x)\,v(y)\,dx\\
&=\phi_k^{(m)}(0)\int_{y\in V}\big(u(0^-,y)-u(0^+,y)\big)\,v(y)\\
&\phantom{={}}{}-\int_{y\in V}\int_{-\epsilon/2}^0\partial_xu(x,y)\,\phi_k^{(m)}(x)\,v(y)\,dx\\
&\phantom{={}}{}-\int_{y\in V}\int_0^{\epsilon/2}\partial_xu(x,y)\,\phi_k^{(m)}(x)\,v(y)\,dx\;,
\end{align*}
which is divergent by~\eqref{int_x in V (u(0^- x) - u(0^+ x)) v(x) ne 0}--\eqref{int_-infty^0 a(t) phi_k^(m)(t) dt to a(0)/2}.
\end{proof}



\begin{thebibliography}{10}

\bibitem{Adams1975}
R.A. Adams, \emph{Sobolev spaces}, Pure and Applied Mathematics, vol.~65,
  Academic Press, Inc., New York-San Francisco-London, 1975. \MR{0450957}

\bibitem{AlvKordyLeichtnam2014}
J.A. {\'{A}}lvarez~L{\'{o}}pez, Y.A. Kordyukov, and E.~Leichtnam,
  \emph{Riemannian foliations of bounded geometry}, Math. Nachr. \textbf{287}
  (2014), no.~14-15, 1589--1608. \MR{3266125}

\bibitem{AlvKordyLeichtnam-ziomf}
\bysame, \emph{Zeta invariants of {M}orse forms}, arXiv:2112.03191, 2021.

\bibitem{AlvKordyLeichtnam-atffff}
\bysame, \emph{A trace formula for foliated flows}, arXiv:2402.06671, 2024.

\bibitem{Bourles2014}
H.~Bourl{\`{e}}s, \emph{On the closed graph theorem and the open mapping
  theorem}, arXiv:1411.5500 [math.FA], 2014.

\bibitem{DabrowskiBrouder2014}
Y.~Dabrowski and C.~Brouder, \emph{Functional properties of {H}{\"{o}}rmander's
  space of distributions having a specified wavefront set}, Comm. Math. Phys.
  \textbf{332} (2014), no.~3, 1345--1380. \MR{3262628}

\bibitem{DeWilde1978}
M.~De~Wilde, \emph{Closed graph theorems and webbed spaces}, Research Notes in
  Mathematics, vol.~19, Pitman (Advanced Publishing Program), London-San
  Francisco-Melbourne, 1978. \MR{518869}

\bibitem{Deninger2008}
C.~Deninger, \emph{Analogies between analysis on foliated spaces and arithmetic
  geometry}, Groups and analysis, London Math. Soc. Lecture Note Ser., vol.
  354, Cambridge Univ. Press, Cambridge, 2008,
  \url{https://doi.org/10.1017/CBO9780511721410.010}, pp.~174--190.
  \MR{2528467}

\bibitem{Edwards1965}
R.E. Edwards, \emph{Functional analysis. {T}heory and applications}, Holt,
  Rinehart and Winston, New York-Toronto-London, 1965. \MR{0221256}

\bibitem{Eichhorn1991}
J.~Eichhorn, \emph{The boundedness of connection coefficients and their
  derivatives}, Math. Nachr. \textbf{152} (1991), 145--158. \MR{1121230}

\bibitem{GuilleminSternberg1977}
V.~Guillemin and S.~Sternberg, \emph{Geometric asymptotics}, Math. Surveys,
  vol.~14, American Mathematical Society, Providence, R.I., 1977. \MR{0516965}

\bibitem{Hirsch1976}
M.W. Hirsch, \emph{Differential topology}, Graduate Texts in Mathematics,
  vol.~33, Springer-Verlag, New York, Heidelberg, Berlin, 1976.

\bibitem{Hormander1971}
L.~H{\"{o}}rmander, \emph{Fourier integral operators. {I}}, Acta Math.
  \textbf{127} (1971), no.~1-2, 79--183. \MR{388463}

\bibitem{Hormander1983-I}
\bysame, \emph{The analysis of linear partial differential operators. {I}.
  {D}istribution theory and {F}ourier analysis}, Grundlehren der Mathematischen
  Wissenschaften, vol. 256, Springer-Verlag, Berlin, 1983. \MR{717035}

\bibitem{Hormander1985-III}
\bysame, \emph{The analysis of linear partial differential operators. {III}.
  {Pseudodifferential} operators}, Grundlehren der Mathematischen
  Wissenschaften, vol. 274, Springer-Verlag, Berlin, 1985.

\bibitem{Horvath1966-I}
J.~Horv{\'a}th, \emph{Topological vector spaces and distributions}, vol.~I,
  Addison-Wesley Publishing Co., Reading, Mass.-London-Don Mills, Ont., 1966.
  \MR{0205028}

\bibitem{KohnNirenberg1965}
J.J. Kohn and L.~Nirenberg, \emph{An algebra of pseudo-differential operators},
  Commun. Pure Appl. Math. \textbf{18} (1965), 269--305. \MR{0176362}

\bibitem{Komatsu1967}
H.~Komatsu, \emph{Projective and injective limits of weakly compact sequences
  of locally convex spaces}, J. Math. Soc. Japan \textbf{19} (1967), 366--383.
  \MR{217557}

\bibitem{Kordyukov1991}
Y.A. Kordyukov, \emph{{$L^p$}-theory of elliptic differential operators on
  manifolds of bounded geometry}, Acta Appl. Math. \textbf{23} (1991), no.~3,
  223--260. \MR{1120831}

\bibitem{Kordyukov2000}
\bysame, \emph{{$L^p$}-estimates for functions of elliptic operators on
  manifolds of bounded geometry}, Russ. J. Math. Phys. \textbf{7} (2000),
  no.~2, 216--229. \MR{1836640}

\bibitem{Kothe1969-I}
G.~K{\"{o}}the, \emph{Topological vector spaces. {I}}, Die Grundlehren der
  mathematischen Wissenschaften, vol. 159, Springer-Verlag,
  Berlin-Heidelberg-New York, 1969, translated from the German by
  D.J.H.~Garling. \MR{0248498}

\bibitem{Kothe1979-II}
\bysame, \emph{Topological vector spaces. {II}}, Grundlehren der mathematischen
  Wissenschaften, vol. 237, Springer-Verlag, New York-Berlin, 1979. \MR{551623}

\bibitem{Kucera2004}
J.~Ku{\v{c}}era, \emph{Reflexivity of inductive limits}, Czechoslovak Math. J.
  \textbf{54} (2004), no.~1, 103--106. \MR{2040223}

\bibitem{Melrose1993}
R.B. Melrose, \emph{The {A}tiyah-{P}atodi-{S}inger index theorem}, Research
  Notes in Mathematics, vol.~4, A.K.~Peters, Ltd., Wellesley, MA, 1993.
  \MR{1348401}

\bibitem{Melrose1996}
\bysame, \emph{Differential analysis on manifolds with corners},
  \url{http://www-math.mit.edu/~rbm/book.html}, 1996.

\bibitem{Melrose2006}
\bysame, \emph{Lectures on pseudodifferential operators},
  \url{https://math.mit.edu/~rbm/18.157-F05.pdf}, 2006.

\bibitem{MelroseUhlmann2008}
R.B. Melrose and G.~Uhlmann, \emph{An introduction to microlocal analysis},
  Department of Mathematics, Massachusetts Institute of Technology, 2008,
  \url{https://books.google.es/books?id=Os2jswEACAAJ},
  \url{http://www-math.mit.edu/~rbm/book.html}.

\bibitem{NariciBeckenstein2011}
L.~Narici and E.~Beckenstein, \emph{Topological vector spaces}, second ed.,
  Pure and Applied Mathematics (Boca Raton), vol. 296, CRC Press, Boca Raton,
  FL, 2011. \MR{2723563}

\bibitem{PerezCarrerasBonet1987}
P.~P{\'{e}}rez~Carreras and J.~Bonet, \emph{Barrelled locally convex spaces},
  North-Holland Math. Stud., vol. 131, North-Holland Publishing Co., Amsterdam,
  1987, Notas de Matem\'{a}tica, 113. \MR{880207}

\bibitem{Roe1988I}
J.~Roe, \emph{An index theorem on open manifolds. {I}}, J. Differential Geom.
  \textbf{27} (1988), no.~1, 87--113. \MR{918459}

\bibitem{Schaefer1971}
H.H. Schaefer, \emph{Topological vector spaces}, Graduate Texts in Mathematics,
  vol.~3, Springer-Verlag, New York, Heidelberg, Berlin, 1971.

\bibitem{Schick1996}
T.~Schick, \emph{Analysis on $\partial$-manifolds of bounded geometry,
  {Hodge-De~Rham} isomorphism and {$L^2$}-index theorem}, Ph.D. thesis,
  Johannes Gutenberg Universit{\"{a}}t Mainz, Mainz, 1996.

\bibitem{Schick2001}
\bysame, \emph{Manifolds with boundary and of bounded geometry}, Math. Nachr.
  \textbf{223} (2001), no.~1, 103--120. \MR{1817852}

\bibitem{Seeley1964}
R.T. Seeley, \emph{Extension of {$C^{\infty}$} functions defined in a half
  space}, Proc. Amer. Math. Soc. \textbf{15} (1964), 625--626. \MR{0165392}

\bibitem{Shubin1992}
M.A. Shubin, \emph{Spectral theory of elliptic operators on noncompact
  manifolds}, Ast\'erisque \textbf{207} (1992), 35--108, M{\'e}thodes
  semi-classiques, Vol. 1 (Nantes, 1991). \MR{1205177}

\bibitem{Simanca1990}
S.R. Simanca, \emph{Pseudo-differential operators}, Pitman Research Notes in
  Mathematics Series, vol. 236, Longman Scientific \& Technical, Harlow;
  copublished in the United States with John Wiley \& Sons, Inc., New York,
  1990. \MR{1075017}

\bibitem{Taylor1981}
M.E. Taylor, \emph{Pseudodifferential operators}, Princeton Mathematical
  Series, vol.~34, Princeton University Press, Princeton, N.J., 1981.
  \MR{618463}

\bibitem{Valdivia1989}
M.~Valdivia, \emph{A characterization of totally reflexive {F}r{\'{e}}chet
  spaces}, Math. Z. \textbf{200} (1989), no.~3, 327--346. \MR{978594}

\bibitem{Wengenroth2003}
J.~Wengenroth, \emph{Derived functors in functional analysis}, Lecture Notes in
  Mathematics, vol. 1810, Springer-Verlag, Berlin, 2003. \MR{1977923}

\end{thebibliography}

\providecommand{\bysame}{\leavevmode\hbox to3em{\hrulefill}\thinspace}
\providecommand{\MR}{\relax\ifhmode\unskip\space\fi MR }
\providecommand{\MRhref}[2]{%
  \href{http://www.ams.org/mathscinet-getitem?mr=#1}{#2}
}
\providecommand{\href}[2]{#2}

\printindex

\end{document}